\documentclass[a4paper]{amsart}

\usepackage[applemac]{inputenc}
\usepackage[T1]{fontenc}
\usepackage[english]{babel}
\usepackage{babelbib}
\usepackage{enumerate}
\usepackage{url}
\usepackage{amssymb}

\usepackage{scalerel}

\DeclareFontFamily{U}{skulls}{}
\DeclareFontShape{U}{skulls}{m}{n}{ <-> skull }{}

\usepackage[hypertexnames=false,backref=page,
 	pdfpagemode=UseNone,
 	breaklinks=true,
 	extension=pdf,
 	colorlinks=true,
 	linkcolor=blue,
 	citecolor=red,
 	urlcolor=blue,
 ]{hyperref} 
\usepackage{color}

\usepackage{tikz}
\usepackage{tikz-cd}

\usetikzlibrary{matrix}
\usetikzlibrary{arrows}

\usepackage{graphicx}
\usepackage{mathrsfs}

\usepackage{cleveref}

\mathcode`A="7041 \mathcode`B="7042 \mathcode`C="7043 \mathcode`D="7044
\mathcode`E="7045 \mathcode`F="7046 \mathcode`G="7047 \mathcode`H="7048
\mathcode`I="7049 \mathcode`J="704A \mathcode`K="704B \mathcode`L="704C
\mathcode`M="704D \mathcode`N="704E \mathcode`O="704F \mathcode`P="7050
\mathcode`Q="7051 \mathcode`R="7052 \mathcode`S="7053 \mathcode`T="7054
\mathcode`U="7055 \mathcode`V="7056 \mathcode`W="7057 \mathcode`X="7058
\mathcode`Y="7059 \mathcode`Z="705A

\renewcommand{\mathcal}{\mathscr}

\newcommand{\bbC}{\mathbb{C}}

\newcommand{\bbF}{\mathbb{F}}
\newcommand{\bbG}{\mathbb{G}}

\newcommand{\bbN}{\mathbb{N}}

\newcommand{\bbP}{\mathbb{P}}
\newcommand{\bbQ}{\mathbb{Q}}

\newcommand{\bbS}{\mathbb{S}}

\newcommand{\bbV}{\mathbb{V}}

\newcommand{\bbZ}{\mathbb{Z}}

\newcommand{\cC}{\mathcal{C}}

\newcommand{\cE}{\mathcal{E}}
\newcommand{\cF}{\mathcal{F}}

\newcommand{\cH}{\mathcal{H}}

\newcommand{\cK}{\mathcal{K}}
\newcommand{\cL}{\mathcal{L}}
\newcommand{\cM}{\mathcal{M}}
\newcommand{\cN}{\mathcal{N}}
\newcommand{\cO}{\mathcal{O}}

\newcommand{\cV}{\mathcal{V}}

\newcommand{\cX}{\mathcal{X}}
\newcommand{\cY}{\mathcal{Y}}
\newcommand{\cZ}{\mathcal{Z}}

\newcommand{\rD}{\textup{D}}

\newcommand{\rH}{\textup{H}}

\newcommand{\rT}{\textup{T}}

\newcommand{\rZ}{\textup{Z}}

\newcommand{\frS}{\mathfrak{S}}

\newcommand{\rd}{\textup{d}}

\newcommand{\sfX}{\mathsf{X}}

\renewcommand{\ss}{\textup{ss}}

\newcommand{\tors}{\textup{tors}}

\newcommand{\sym}{{\bbS}}

\newcommand{\reg}{\textup{reg}}

\newcommand{\dashto}{\dashrightarrow}
\newcommand{\into}{\hookrightarrow}
\newcommand{\too}{\longrightarrow}
\renewcommand{\phi}{\varphi}
\renewcommand{\epsilon}{\varepsilon}
\renewcommand{\ker}{\Ker}

\newcommand{\iso}{\simeq}

\newcommand{\Rep}{\textup{Rep}}
\newcommand{\Dbc}{{\mathrm{D}^b_c}}
\newcommand{\Perv}{\textup{Perv}}
\newcommand{\rmS}{\textup{S}}
\newcommand{\rmT}{\textup{T}}
\newcommand{\Dbcbar}{\overline{\Dbc}}
\newcommand{\Pbar}{\overline{\Perv}}
\newcommand{\R}{\textup{R}}

\newcommand{\PLambda}{\scalerel*{{\rotatebox[origin=c]{180}{$\bbV$}}}{\bbV}}
\renewcommand{\wr}{\rotatebox[origin=c]{90}{$\sim$}}

\DeclareMathOperator{\Vect}{Vect}
\DeclareMathOperator{\Alb}{Alb}

\DeclareMathOperator{\Gal}{Gal}

\DeclareMathOperator{\rk}{rk}
\DeclareMathOperator{\pr}{pr}

\DeclareMathOperator{\Spec}{Spec}
\DeclareMathOperator{\Proj}{Proj}

\DeclareMathOperator{\Hom}{Hom}

\DeclareMathOperator{\im}{Im}
\DeclareMathOperator{\Ker}{Ker}

\DeclareMathOperator{\CH}{CH}
\DeclareMathOperator{\Pic}{Pic}
\DeclareMathOperator{\Stab}{Stab}

\DeclareMathOperator{\End}{End}

\DeclareMathOperator{\Sym}{Sym}

\DeclareMathOperator{\Aut}{Aut}
\DeclareMathOperator{\GL}{GL}
\DeclareMathOperator{\SL}{SL}
\DeclareMathOperator{\SO}{SO}
\DeclareMathOperator{\Sp}{Sp}

\DeclareMathOperator{\Spin}{Spin}

\DeclareMathOperator{\Supp}{Supp}
\DeclareMathOperator{\Lie}{Lie}

\DeclareMathOperator{\id}{id}

\DeclareMathOperator{\Hilb}{Hilb}

\DeclareMathOperator{\alb}{alb}
\DeclareMathOperator{\cone}{cone}

\DeclareMathOperator{\characteristic}{char}

\DeclareMathOperator{\codim}{codim}

\DeclareMathOperator{\Alt}{Alt}

\DeclareMathOperator{\cc}{cc}
\DeclareMathOperator{\CC}{CC}

\newcommand{\coLie}[1]{{(\Lie {#1})^\vee}}

\renewcommand{\le}{\leqslant}
\renewcommand{\ge}{\geqslant}

\theoremstyle{plain}
\newtheorem{theoremintro}{Theorem}
\newtheorem*{unlabeledthm}{Theorem}

\newtheorem*{maintheorem-monodromy}{Main theorem (monodromy version)}
\newtheorem*{maintheorem-tannaka}{Main theorem (Tannaka version)}
\newtheorem*{corollaryintro}{Corollary}
\newtheorem{theorem}{Theorem}[section]
\newtheorem{lemma}[theorem]{Lemma}
\newtheorem{proposition}[theorem]{Proposition}
\newtheorem{corollary}[theorem]{Corollary}
\newtheorem{claim}[theorem]{Claim}
\newtheorem{fact}[theorem]{Fact}
\newtheorem*{factintro}{Fact}

\newtheorem{assumption}[equation]{Assumption}

\theoremstyle{definition}
\newtheorem*{definitionintro}{Definition}
\newtheorem{definition}[theorem]{Definition}
\newtheorem{example}[theorem]{Example}

\newtheorem{remark}[theorem]{Remark}
\newtheorem*{exampleintro}{Example}

\numberwithin{equation}{section}

\begin{document}

\title{The monodromy of families of subvarieties on abelian varieties}

\author[A. Javanpeykar]{Ariyan Javanpeykar}
\address{Ariyan Javanpeykar \\
	IMAPP\\
	Radboud University Nijmegen\\	
	PO Box 9010\\
	6500GL Nijmegen\\
	The Netherlands.}
\email{peykar@uni-mainz.de}

\author[T. Kr\"amer]{Thomas Kr\"amer}
\address{Thomas Kr\"amer \\
	Institut f\"{u}r Mathematik\\
	Humboldt Universit\"at zu Berlin\\
	Rudower Chaussee 25,  12489 Berlin\\
	Germany.}
\email{thomas.kraemer@math.hu-berlin.de}

\author[C. Lehn]{Christian Lehn}
\address{Christian Lehn\\ 
Fakult\"at f\"ur Mathematik\\Ruhr-Universit\"at Bochum\\ 
Universit\"ats\-stra\ss e~150\\
Postfach IB 45\\
44801 Bochum, Germany}
\email{christian.lehn@rub.de}

\author[M. Maculan]{Marco Maculan}
\address{ Marco Maculan \\
	Institut de Math\'ematiques de Jussieu \\
Sorbonne Universit\'e \\
4, place Jussieu \\
75005 Paris \\
France}
\email{marco.maculan@imj-prg.fr}

\begin{abstract}
Motivated by recent work of Lawrence-Venkatesh and Lawrence-Sawin, we show that non-isotrivial families of  subvarieties in abelian varieties have big monodromy when twisted by generic rank one local systems.  While Lawrence-Sawin discuss the case of subvarieties of codimension one, our results hold for  subvarieties of codimension at least half the dimension of the ambient abelian variety. 
For the proof, we use a combination of geometric arguments and representation theory to show that the Tannaka groups of intersection complexes on such subvarieties are big.
 \end{abstract}
 
\makeatletter
\@namedef{subjclassname@2020}{
	\textup{2020} Mathematics Subject Classification}
\makeatother



 \maketitle 
 
 \setcounter{tocdepth}{1}

\tableofcontents

\section{Introduction}

\thispagestyle{empty} 

Recently, Lawrence and Venkatesh \cite{LV} have developed a technique to prove nondensity of integral points on varieties that are defined over a number field and support a geometric variation of Hodge structures with big monodromy. They used this method to give an alternative proof of the Mordell conjecture and to show nondensity for hypersurfaces in projective space of a given (high) degree with good reduction outside a fixed finite set of primes. Later, Lawrence and Sawin~\cite{LS20} applied this strategy to show that up to translation any abelian variety over a number field contains only finitely many smooth ample hypersurfaces with given N\'eron-Severi class and good reduction outside a fixed finite set of primes. The main novelty of their work lies in their way to control monodromy. The arguments of Lawrence and Venkatesh have a topological flavor. For the Mordell conjecture they rely on a judicious choice of Dehn twists; for hypersurfaces in projective space they use the computation of the integral monodromy of the universal family  by Beauville~\cite{Beauville} (based on the work of Ebeling~\cite{Ebeling} and Janssen~\cite{Janssen}), see also the discussion by Katz in~\cite{KatzLarsen}. Instead, the approach by Lawrence and Sawin involves Tannaka groups of perverse sheaves on abelian varieties introduced by Kr\"amer and Weissauer \cite{KWVanishing}; the relation of these groups to monodromy is reminiscent of the one between the monodromy group of a variation of Hodge structures and its  generic Mumford-Tate group \cite{AndreMumfordTate}.

\medskip

With a view towards new arithmetic applications along these lines \cite{KM}, we prove a big monodromy theorem for families of subvarieties of higher codimension in abelian varieties. Our results hold for all subvarieties of codimension at least half the dimension of the abelian variety. The geometry in this codimension range is very different from the codimension one case in~\cite{LS20}, and the results about Tannaka groups that we obtain on the way may be of independent interest.

\subsection{Big monodromy} \label{subsec:bigmonodromy}
Let $S$ be a smooth irreducible variety over an algebraically closed field $k$ of characteristic zero. Let $A$ be an abelian variety of dimension~$g$ over~$k$. Inside the constant abelian scheme~$A_S := A\times S$, let~$\cX\subset A_S$ be a closed subvariety which is smooth over $S$ with connected fibers of dimension~$d$. The goal of this paper is to understand the monodromy of rank one local systems on the smooth proper family $f\colon \cX \to S$ in the following diagram:
\[
\begin{tikzcd}
&  \cX \ar[d, hook, no head, xshift=-1pt] \ar[d, no head, shorten <=1.2pt, xshift=1pt] \ar[dl, bend right=30, swap, "\pi"] \ar[dr, bend left=30, "f"]  & \\
A & A_S \ar[l, swap, "\pr_A"] \ar[r, "\pr_S"]  & S
\end{tikzcd}
\]
Our results apply both in the analytic and in the algebraic setup, using topological local systems with coefficients in $\bbF = \bbC$ for $k = \bbC$ resp.~\'etale $\ell$-adic local systems with coefficients in $\bbF = \overline{\bbQ}_\ell$ for a prime $\ell$ over an arbitrary algebraically closed field $k$ of characteristic zero. Let $\pi_1(A, 0)$ be the topological resp.~\'etale fundamental group with the discrete resp.~profinite topology, and denote the group of its continuous characters by 
\[
 \Pi(A, \bbF)=\Hom(\pi_1(A, 0), \bbF^\times).
\]
In what follows, by a \emph{linear subvariety} we mean a subset $\Pi(B, \bbF)\subset \Pi(A, \bbF)$ for an abelian quotient variety $A\twoheadrightarrow B$ with $\dim B < \dim A$. 
We say that a statement holds for {\em most} $\chi \in \Pi(A, \bbF)$ if it holds for all $\chi$ outside a finite union of torsion translates of linear subvarieties. For $\chi \in \Pi(A, \bbF)$, let $L_\chi$ denote the associated rank one local system on $A$. It follows from generic vanishing~\cite{BSS, KWVanishing, SchnellHolonomic} that for most $\chi$ the higher direct images $R^i f_* \pi^* L_\chi$ vanish in all degrees $i\neq d$; we consider the local system
\[
V_{\chi} \;:=\; R^d f_* \pi^* L_{\chi}
\]
of rank $|e|$ where $e$ is the topological Euler characteristic of the fibers of $\cX \to S$. More generally, the study of finite \'etale covers of the subvariety $\cX \subset A_S$ induced by finite \'etale covers of $A$ leads to direct sums
\[
 V_{\underline{\chi}} \;:=\; V_{\chi_1} \oplus \cdots \oplus V_{\chi_n} 
\] 
where $\underline{\chi} = (\chi_1, \dots, \chi_n) \in \Pi(A, \bbF)^n$ is an $n$-tuple of characters of the fundamental group. Using the natural identification $\Pi(A, \bbF)^n = \Pi(A^n, \bbF)$ we will also apply the terminology {\em most} for such $n$-tuples of characters. Consider for $s\in S(k)$ the monodromy representation
\[
 \rho\colon \quad 
\pi_1(S, s) \;\too\; 
\GL(V_{\underline{\chi}, s})
\quad \textnormal{on the fiber} \quad
V_{\underline{\chi}, s} \;=\; \bigoplus_{i=1}^n \rH^d(\cX_s, L_{\chi_i}).
\]
The \emph{algebraic monodromy group} of the local system $V_{\underline{\chi}}$ is the Zariski closure of the image of $\rho$. By construction it is an algebraic subgroup of
\[
  \GL(V_{\chi_1, s}) \times \cdots \times \GL(V_{\chi_n, s})
 \;\subset\; 
 \GL(V_{\underline{\chi}, s}).
\]
This upper bound can sometimes be refined: We say that the subvariety $\cX \subset A_S$ is \emph{symmetric up to translation} if there exists $a\colon S\to A$ such that $\cX_t = - \cX_t + a(t)$ for all $t\in S(k)$. In this case, Poincar\'e duality furnishes a nondegenerate bilinear pairing
\[ 
\theta_{\chi,s} \colon 
\quad V_{\chi,s} \otimes  V_{\chi, s} \too L_{\chi, a(s)}
\]
for each $\chi \in \Pi(A, \bbF)$, because for the dual of a rank one local system and for its inverse image under the translation $\tau_{a(t)}\colon A\to A, x\mapsto x+a(t)$ we have natural isomorphisms
\begin{eqnarray*}
  L_\chi^\vee &\iso& [-1]^\ast L_\chi, \\
  \tau_{a(t)}^\ast L_\chi 
  &\iso & L_{\chi} \otimes_{\bbF} L_{\chi, a(t)}.
\end{eqnarray*} 
The pairing $\theta_{\chi, s}$ is symmetric if $d$ is even, and alternating otherwise. Since the pairing is compatible with the monodromy operation on the fiber, it follows that the algebraic monodromy group of $V_{\underline{\chi}}$ is contained in an orthogonal resp.~symplectic group in the two cases. This leads to the following definition:

\begin{definitionintro} 
We say that $V_{\underline{\chi}}$ has {\em big monodromy} if its algebraic monodromy group contains $G_1\times \cdots \times G_n$ as a normal subgroup where $G_i \subset \GL(V_{\chi_i, s})$ is defined by
\[
G_i \;:=\; 
\begin{cases}
\SL(V_{\chi_i, s}) & \textup{if $\cX$ is not symmetric up to translation}, \\
\SO(V_{\chi_i,s}, \theta_{\chi_i,s})  & \textup{if $\cX$ is symmetric up to translation and $d$ is even},  \\
\Sp(V_{\chi_i,s}, \theta_{\chi_i,s})  & \textup{if $\cX$ is symmetric up to translation and $d$ is odd}.
\end{cases}
\]
\end{definitionintro} 

\smallskip

Note that the connected component of the algebraic monodromy group of $V_{\underline{\chi}}$ is unaffected by base change along \'etale morphisms $S'\to S$. To take this into account we consider the fiber $\cX_{\bar{\eta}}$ of $\cX \to S$ at a geometric generic point $\bar{\eta}$ of $S$. There are four obvious cases where the local system~$V_{\underline{\chi}}$ does not have big monodromy: We say that $\cX_{\bar{\eta}} \subset A_{S, \bar{\eta}}$ is\medskip
 \begin{enumerate}
 \item  \emph{constant up to a translation} if it is the translate of a subvariety $Y \subset A$ along a point in $A(\bar{\eta})$. In this case the algebraic monodromy is finite. \medskip
 
  \item {\em divisible} if it is stable under translation by a torsion point $0\neq x\in A(\bar{\eta})$. In this case the algebraic monodromy of each $V_{\chi_i}$ is itself a group of block matrices which is normalized by the group generated by the point $x$. \smallskip 
   
 \item a \emph{symmetric power of a curve} if there is a smooth curve $C \subset A_{S, \bar{\eta}}$ such that the sum morphism $\Sym^d C \to A_{S, \bar{\eta}}$ is a closed embedding with image $\cX_{\bar{\eta}}$ and $d\ge 2$. 
After an \'etale base change over $S$, we may assume that~$C$ spreads out to a relative curve $\cC \subset A_{S}$ which is smooth and proper over $S$ such that the relative sum morphism $\Sym_{S}^d \cC \to A_{S}$ is a closed embedding with image $\cX$. Then we have an isomorphism compatible with monodromy:
\[ \rH^d(\cX_{s}, L_{\chi}) \iso \Alt^d \rH^1(\cC_{s}, L_{\chi}).\]

 \item a \emph{product} if there are smooth subvarieties $X_1, X_2 \subset A_{S, \bar{\eta}}$ with $\dim X_i > 0$ 
 such that the sum morphism $X_1 \times X_2 \to A_{S, \bar{\eta}}$ is a closed embedding with image $\cX_{\bar{\eta}}$.  Again, after an \'etale base change over $S$ we may assume that~$X_i$ spreads out to a subvariety $\cX_i \subset A_{S}$ which is smooth and proper over $S$ such that the relative sum morphism $\cX_1 \times \cX_2 \to A_{S}$ is a closed embedding with image $\cX$. 
Then we have the K\"{u}nneth isomorphism which is compatible with monodromy:
\[ \rH^d(\cX_{s}, L_{\chi}) \iso \bigoplus_{i_1 + i_2 = d}\rH^{i_1}(\cX_{1, s}, L_{\chi}) \otimes \rH^{i_2}(\cX_{2, s}, L_{\chi}).\]
 \end{enumerate}
 
If $\cX_{\bar{\eta}}$ is nondivisible, then 
condition (1) holds if and only if the family $\cX \to S$ is isotrivial; see \cref{cor:constant_up_to}. To avoid the appearance of the exceptional groups~$E_6$ and $E_7$
and some low-dimensional half-spin groups, we require a mild assumption on the topological Euler characteristic:
\begin{assumption} \label{Eq:NumericalConditions} 
The topological Euler characteristic $e$ of $\cX_{\bar{\eta}}$ satisfies
\begin{align} 
|e| \; &\neq \; 27 && \text{if $d\ge 2$ and $\cX$ is not symmetric up to a translation}, \notag \\
|e| \; &\neq \; 56 && \text{if $d \ge 3$ is odd and $\cX$ is symmetric up to translation}, \notag\\
|e| \; &\neq \; 2^{2m - 1} && \text{if $d \ge (g-1)/4$, $ m\in \{3, \dots, d\}$ 
has the same parity as $d$} \notag \\
& &&  \text{and $\cX$ is symmetric up to translation}. \notag
\end{align}
\end{assumption}
Note that $|e|\ge g$ if $\cX_{\bar{\eta}} \subset A_{S, \bar{\eta}}$ has ample normal bundle, see \cref{Lemma:LowerBoundTopCharAmpleNormalBundle}.
We do not know any example of 
a smooth subvariety of $A_{S, \bar{\eta}}$ with ample normal bundle and dimension $d<(g-1)/2$ whose Euler characteristic $e$ does not satisfy~\cref{Eq:NumericalConditions}.

\begin{maintheorem-monodromy} Suppose $\cX_{\bar{\eta}} \subset A_{S, \bar \eta}$ has ample normal bundle, dimension $d < (g-1)/2$, and~\cref{Eq:NumericalConditions} holds. Then the following are equivalent:
\begin{enumerate}
\item $\cX_{\bar \eta}$ is nondivisible, not constant up to translation, not a symmetric power of a curve and not a product; \smallskip
\item $V_{\underline{\chi}}$ has big monodromy for most 
torsion $n$-tuples $\underline{\chi} \in \Pi(A, \bbF)^n$.
\end{enumerate}
\end{maintheorem-monodromy}

Smooth proper subvarieties of a simple abelian variety have ample normal bundle. Therefore when $A$ is simple the preceding theorem is as general as it gets for smooth subvarieties of dimension $d < (g-1)/2$, save the finite list of exceptions~in~\cref{Eq:NumericalConditions}. When~$A$ is arbitrary, the theorem can be applied in the following concrete cases:

\begin{corollaryintro}
Suppose $\cX_{\bar \eta}\subset A_{S, \bar \eta}$ is nondivisible, not constant up to translation, and one of the following holds:
\begin{enumerate}
\item $\cX_{\bar{\eta}}$ is a curve generating $A_{S, \bar{\eta}}$ and $g\ge 4$; \smallskip 
\item  $\cX_{\bar{\eta}}$ is a surface with ample normal bundle which is neither a symmetric square of a curve nor a product, and $e \neq 27$, $g\ge 6$;\smallskip 
\item $\cX_{\bar{\eta}}$ is a complete intersection of ample divisors and $d<(g-1)/2$.\smallskip
\end{enumerate}
Then $V_{\underline{\chi}}$ has big monodromy for most $n$-tuples $\underline{\chi} \in \Pi(A, \bbF)^n$ of torsion characters.
\end{corollaryintro}

Indeed a smooth complete intersection of ample divisors is neither a symmetric power of a curve (\cref{Cor:SymmetricPowersOfCurvesAreNotCompleteIntersections}) nor a product (\cref{Rem:complete_intersection}) and its topological Euler characteristic satisfies $|e| \ge 2^g$ and $|e| \neq 27, 56$ (\cref{Cor:EulerCharIsNot56,,Cor:LowerBoundEulerChar}).

\medskip 

Over $k = \bbC$, the main theorem in the analytic setup is deduced from the algebraic one by the comparison between classical and \'etale topology; the hypothesis that the characters are torsion is only used here. For the proof in the algebraic setting, we start as in \cite{LS20} by relating the algebraic monodromy to the Tannaka group of the rank one local systems in question, seen as perverse sheaves on $\cX_{\bar{\eta}}$. The idea is similar to the study of monodromy groups via Mumford-Tate groups in the complex case~\cite{AndreMumfordTate}: An analog of the theorem of the fixed part due to Lawrence and Sawin says that the monodromy will be big if we can show that the Tannaka group of the geometric generic fiber is big (see \cref{Thm:FixedPart}); note that the property of the family being symmetric up to translation can be read off from its geometric generic fiber (\cref{cor:constant_up_to}).  Thus, we are left with a question about the Tannaka group of the geometric generic fiber of our family. In this setting, we will reset our notation and replace $k$ by an algebraic closure of the function field of~$S$.

\subsection{Big Tannaka groups} \label{sec:BigTannakaIntro}

As before, let $A$ be an abelian variety of dimension~$g$ over an algebraically closed field $k$ of characteristic zero.  Let $i\colon X\hookrightarrow A$ be the inclusion of a smooth connected closed subvariety of dimension $d$. We define the perverse intersection complex 
\[ \delta_X \;:=\; i_\ast\bbF_X[d] \]
as the pushforward of the constant sheaf, shifted in cohomological degree~$-d$ so that it becomes an object of the abelian category $\Perv(A, \bbF)$  of perverse sheaves on~$A$ as in \cite{BBDG}. As we will recall in \cref{section perverse sheaves on abelian varieties}, the group law on the abelian variety induces a convolution product on perverse sheaves, and the perverse intersection complex $\delta_X$ generates a neutral Tannaka category $\langle \delta_X \rangle$ with respect to this convolution.  
For the rest of this introduction, we fix a character~$\chi \in \Pi(A, \bbF)$ with~$\rH^i(X, L_\chi)= 0$ for all $i\neq d$. Such a character exists by generic vanishing. We then have a fiber functor 
\[
 \omega \colon \quad \langle \delta_X \rangle \;\too\; \Vect(\bbF), \quad P \;\longmapsto\; \rH^0(A, P\otimes L_\chi),
\]
see~\cite[th.~13.2]{KWVanishing}. Applying this fiber functor to $P=\delta_X$ we recover the vector space 
\[ V := \omega(\delta_X) = \rH^0(A, \delta_X \otimes L_\chi) = \rH^d(X, L_\chi). \]
The automorphisms of the fiber functor are represented by a reductive algebraic group $G_{X, \omega}:=G_{\omega}(\delta_X)\subset \GL(V)$ which we call the \emph{Tannaka group of $X$}, see also \cref{rem:definition-of-tannaka-groups}.  The definitions in \cref{subsec:bigmonodromy} with $S=\Spec(k)$ show that if $X\subset A$ is symmetric up to translation, then $V$ comes with a natural symmetric bilinear form $\theta$ which is induced by Poincar\'e duality. This bilinear form is symmetric or alternating depending on the parity of $d$, and it is preserved by the action of the group $G_{X, \omega}$ as in \cite[lemma~2.1]{KWGeneric}. Let $G_{X, \omega}^\circ\subset G_{X, \omega}$ be the connected component of the identity and 
\[
G_{X,\omega}^\ast := [G_{X, \omega}^\circ, G_{X, \omega}^\circ]
\]
its derived group, which is a connected semisimple algebraic group.

\begin{definitionintro} 
We say that the Tannaka group $G_{X, \omega}$
of $X$ is {\em big} if the derived group of its connected component of the identity is 
\[
G_{X, \omega}^* \;=\;
\begin{cases}
\SL(V) & \textup{if $X$ is not symmetric up to translation}, \\
\SO(V, \theta)  & \textup{if $X$ is symmetric up to translation and $d$ is even},  \\
\Sp(V, \theta)  & \textup{if $X$ is symmetric up to translation and $d$ is odd}.
\end{cases}
\]
\end{definitionintro}

The main theorem from the previous section is obtained by combining the analog of the theorem of the fixed part by Lawrence and Sawin (\cref{Thm:FixedPart}) with the following result, whose proof will be the main task of this paper. Again we need to exclude a finite list of values of the topological Euler characteristic $e$ of $X$, for which we refer to \cref{Eq:NumericalConditions} with $S=\mathrm{Spec}(k)$ and $\cX=X$.

\begin{maintheorem-tannaka} Suppose $X\subset A$ has ample normal bundle, dimension $d < (g-1)/2$, and \cref{Eq:NumericalConditions} holds. Then the following are equivalent:
\begin{enumerate}
\item  $X$ is nondivisible, not a symmetric power of a curve and not a product; \smallskip
\item  The Tannaka group $G_{X, \omega}$ is big. 
\end{enumerate}
\end{maintheorem-tannaka}

Similarly to the monodromy version, the preceding statement is substantially sharp in the simple case and can be applied in the following special cases:

\begin{corollaryintro} Suppose $X\subset A$ is nondivisible and one of the following holds:

\begin{enumerate}
\item $X$ is a curve generating $A$ and $g\ge 4$; \smallskip  
\item  $X$ is a surface with ample normal bundle which is neither a product nor the symmetric square of a curve, and $e \neq 27$, $g\ge 6$; \smallskip
\item $X$ is a complete intersection of ample divisors and $d<(g-1)/2$.\smallskip
\end{enumerate}
Then the Tannaka group $G_{X, \omega}$ is big. 
\end{corollaryintro}

The Tannaka version of the main theorem also applies when $X$ does not arise from a family as in \cref{subsec:bigmonodromy}, so it is stronger than the monodromy version. We also note that over the complex numbers both versions apply in many cases where we have no control on Mumford-Tate groups of the subvarieties. Again, when $X$ is a complete intersection of ample divisors, then automatically $|e| \neq 27, 56$, $X$ is not a symmetric power of a curve nor a product. 

\subsection{Sum morphisms} \label{Sec:SumMorphisms} Before we describe the proof of the main theorem, let us illustrate the meaning of big Tannaka groups with a simple application. Let~$X\subset A$ be a subvariety of dimension $d$. For any integer $r\ge 1$ the sum morphism induces a morphism
\[
 \Sym^r X \;\too\; W_r(X) \;:=\; X + \cdots + X \;\subset\; A
\]
onto the $r$-fold sum of the subvariety inside the abelian variety. If $X\subset A$ has ample normal bundle or more generally if it is nondegenerate in the sense of~\cref{section positivity} below, then for $r<g/d$ this sum morphism is generically finite onto its image. In general it will not be birational:

\begin{exampleintro}
Let $C$ be a smooth projective curve of genus $g\ge 2$, seen as a subvariety of its Jacobian variety~$A=\Pic^0(C)$ via the Abel-Jacobi embedding for a given base point. Then the subvariety
\[
 X \;:=\; W_d(C) \;\subset\; A
\]
is smooth if $C$ has gonality $>d$. The map
$
 \Sym^r X \to W_r(X) = W_{rd}(C)
$
is not birational. Note that the symmetric power $\Sym^r X$ is singular for $d>1$, but the image $W_r(X)=W_{rd}(C)$ is smooth if $C$ has gonality $> rd$.
\end{exampleintro} 

In the above example the Tannaka group $G_{X,\omega}$ is not big for $d>1$, see~\cref{ex:symmetric-power-of-curve} below. For subvarieties whose Tannaka group is big, which by our main theorem is true in most cases, we have:

\begin{unlabeledthm}
Let $X\subset A$ be a smooth subvariety of dimension $d$ with ample normal bundle. If the Tannaka group $G_{X,\omega}$ is big, then for any integer $2 \le  r < g/d$ the sum morphism 
$ \Sym^r X \rightarrow W_r(X)$
is birational, and $W_r(X)$ is singular for $d>1$.
\end{unlabeledthm}

The key point here is the birationality, which will be shown in~\cref{cor:birational-via-larsen}. Once the birationality is known, the smoothness of $W_r(X)$ implies that the sum morphism is an isomorphism by~\cref{Lem:BirationalMapIsIso}. In particular $\Sym^r X$ is then also smooth, so $d=1$ by~\cref{Prop:SmoothSymmetricPower}. The proof of~\cref{cor:birational-via-larsen} relates the direct image of the constant sheaf under the sum morphism to the decomposition of wedge or symmetric powers of $V \in \Rep_\bbF(G_{X, \omega})$. In fact Larsen's alternative yields a necessary and sufficient criterion for the Tannaka group to be big, using only the decomposition of the direct image of the constant sheaf under the sum morphism for $r=2$. But it seems hard to control this direct image in the generality needed for our main theorem, so for the proof of the main theorem we follow a different route that will be described in~\cref{sec:IntroSimplicity,,sec:IntroWedgePowers,,sec:IntroSpin}.

\subsection{Simplicity of Tannaka groups} \label{sec:IntroSimplicity}

The first step in our proof of the main theorem from the previous section will be to show that under the given assumptions, the algebraic group~$G_{X, \omega}^*$ is simple. We refine the methods in \cite[section~6]{KraemerMicrolocalII} to obtain the following simplicity criterion (see \cref{Thm:TannakaGroupSimple}):

\begin{theoremintro} \label{Thm:TannakaGroupSimpleIntro} 
Suppose $X\subset A$ has ample normal bundle and is nondivisible. Then for $g \ge 3$ the following are equivalent:
\begin{enumerate}
\item The algebraic group $G_{X, \omega}^\ast$ is not simple;
\item There are smooth positive-dimensional subvarieties $X_1, X_2 \subset A$ such that the sum morphism induces an isomorphism
\[ X_1 \times X_2 \stackrel{\sim}{\too} X.\]
\end{enumerate}\end{theoremintro}

A smooth projective curve $C \subset A$  generating $A$ has ample normal bundle, thus the algebraic group $G_{C, \omega}^\ast$ is simple for $g \ge 3$. When $g = 2$ the simplicity of $G_{C, \omega}^\ast$ remains open. More generally \cref{Thm:TannakaGroupSimpleIntro} implies that $G_{X, \omega}^\ast$ is simple when $X \subset A$ is nondivisible with ample normal bundle and
\begin{enumerate}
\item the image of the Albanese morphism  $X\to \Alb(X)$ is nondegenerate in the sense of \cref{section positivity};
\item the natural morphism $\phi \colon \Alb(X) \to A$ is an isogeny. By Debarre's Barth-Lefschetz theorem for abelian varieties (see \cite[th.~4.5]{Deb95} or \cref{Rem:complete_intersection}) this is the case as soon as $d > g/2$ or when $X$ is a complete intersection of ample divisors and $d \ge 2$.
\end{enumerate}
Note that situation (2) above is a particular case of (1).

\medskip
Our proof of \cref{Thm:TannakaGroupSimpleIntro} uses characteristic cycles on the cotangent bundle~$T^*A$ and their link with representation theory \cite{KraemerMicrolocalI, KraemerMicrolocalII}. The idea is roughly as follows: If the group $G_{X,\omega}^\ast$ is not simple, then the representation $V=\omega(\delta_X)$ is an external tensor product of representations. This allows to decompose the characteristic cycle of the perverse sheaf~$\delta_X$ as a Pontryagin product. But for smooth subvarieties the characteristic cycle is integral and equal to the conormal bundle to $X \subset A$. Using our assumption that the normal bundle is ample, we can then rule out decompositions as Pontryagin products via computations with Segre classes.
For convenience we recall some relevant background in \cref{section from rep to geo}, together with computations for the Dynkin types $A$, $B$, $D$ to be used later. 
The integrality of the characteristic cycle also implies that the representation $\omega(\delta_X) \in \Rep_\bbF(G_{X,\omega}^\ast)$ is minuscule in the sense that its weights for a maximal torus form a single Weyl group orbit, due to the following general result (see \cref{cor:minuscule}):  

\begin{factintro} 
Let $P\in \Perv(A, \bbF)$ be a perverse sheaf whose characteristic cycle is integral and not stable under any nontrivial translation on the abelian variety. Then $\omega(P)$ is a minuscule representation of $G_\omega(P)$.
\end{factintro} 

There are only few nontrivial minuscule representations $V$ of a simply connected simple algebraic group $G$, all of which are listed below: \medskip
 \begin{center} 
 \label{Table:MinusculeRepresentations}
 \renewcommand{\arraystretch}{1.3}
 \begin{tabular}{c|c|c|c}
 Dynkin type & G &  $V$ & $\dim V$ \\
 \hline \hline
 $A_{n}$ & $\SL_{n+1}$ &  $r$-th wedge power & $\tbinom{n+1}{r}$ \\
 $B_n$ & $\Spin_{2n + 1}$ & spin & $2^n$ \\
 $C_n$ & $\Sp_{2n}$ & standard & $2n$ \\
  $D_n$ & $\Spin_{2n}$ & standard of $\SO_{2n}$ & $2n$ \\
 $D_n$ & $\Spin_{2n}$ & half-spins & $2^{n-1}$ \\[0.3em]
 $E_6$ & $E_6$ & { \renewcommand{\arraystretch}{0.85}\begin{tabular}{c} smallest nontrivial \\ or its dual \end{tabular} \renewcommand{\arraystretch}{1.5}} & $27$ \\[0.1em]
 $E_7$ & $E_7$ & smallest nontrivial & $56$
\end{tabular} \smallskip
 \end{center}
The dimension of $\omega(\delta_X)$ is the absolute value of the topological Euler characteristic of $X$. Recall that the subvariety $X\subset A$ is symmetric up to a translation if and only if the vector space $\omega(\delta_X)$ carries a nondegenerate bilinear form preserved by the action of $G_{X, \omega}^\ast$, and this pairing is symmetric if $d$ is even and alternating if $d$ is odd; see \cite[lemma~2.1]{KWGeneric}. This rules out the occurence of $E_6$ for symmetric subvarieties; note that the group $E_6$ appears as the Tannaka group of the Fano surface in the intermediate Jacobian of a smooth cubic threefold, but $d = (g-1)/2$ here because $d = 2$ and $g = 5$. Similarly, the group $E_7$ preserves a nondegenerate alternating bilinear form on its $57$-dimensional irreducible representation, so subvarieties $X$ with $G_{X, \omega}^\ast \simeq E_7$ must be odd-dimensional. However, for $d = 1$ this does not happen as we show by a direct geometric argument (see~\cref{cor:curve-is-not-e7}), and in higher dimension we do not any such example. Altogether, to prove that the Tannaka group is big and conclude the proof of the main theorem from \cref{sec:BigTannakaIntro}, we are left with wedge powers and spin representations.  The next two sections will characterize the occurence of the former and rule out the latter.

\subsection{Wedge powers} \label{sec:IntroWedgePowers}

In contrast to the situation for hypersurfaces studied by Lawrence and Sawin in~\cite{LS20}, one cannot rule out the occurrence of nontrivial wedge powers for subvarieties of higher codimension by numerical arguments. In fact, wedge powers do appear, but we will use geometric arguments to obtain the following complete classification (see \cref{Thm:SmallWedgePowersAreSumsOfCurves}):

\begin{theoremintro}
\label{Thm:SmallWedgePowersAreSumsOfCurves_Intro}
Suppose $X\subset A$ has ample normal bundle and is nondivisible. Then for $d < (g - 1)/2$ the following are equivalent:\smallskip 
\begin{enumerate} 
\item  There are integers $r$ and $n$ with  $1<r \le n/2$ such that
    $G_{X,\omega}^{\ast}\simeq \Alt^r(\SL_n)$ and $\omega(\delta_X)$  is    the $r$-th wedge power of the standard representation.     \smallskip 
\item There is a nondegenerate irreducible smooth projective curve $C\subset A$ such that \begin{itemize}
\item $X=C+\cdots + C \subset A$ is the sum of $r$ copies of $C$, and\smallskip 
\item the sum morphism $\Sym^r C \to X$ is an isomorphism.
\end{itemize}
\end{enumerate} 
\end{theoremintro}
 
\subsection{Spin representations}  \label{sec:IntroSpin}
Recall that for $N\ge 3$ the group $\SO_N(\bbF)$ admits a double cover
\[
 \Spin_N(\bbF) \;\too\; \SO_N(\bbF)
\] 
by the \emph{spin group} $\Spin_N(\bbF)$, a simply connected algebraic group with a faithful representation $\bbS_N$, the \emph{spin representation}. We have $\dim \bbS_N = 2^n$ for $n=\lfloor N/2 \rfloor$, and if $N$ is odd, then the spin representation is irreducible. If $N=2n$ is even, then the spin representation $\bbS_N \simeq \bbS_N^+ \oplus \bbS_N^-$ splits as the direct sum  of two irreducible representations called the {\em half-spin representations}. They both have dimension $\dim \bbS_N^+ = \dim \bbS_N^- = 2^{n-1}$.
For odd $n=2m+1$, the half-spin representations are both faithful and dual to each other; for even $n=2m$, they are both self-dual and their images
\[
 \Spin_{4m}^\pm(\bbF) \;\subset\; \GL(\bbS_{4m}^\pm)
\]
are called the {\em half-spin groups}.   
We show that spin or half-spin groups do not occur for smooth nondivisible subvarieties of high enough codimension (see \cref{Thm:SmallSpinDoNotExist}):

\begin{theoremintro}\label{Thm:Spin_Intro} 
Suppose that $X\subset A$ has ample normal bundle, is nondivisible and has dimension $d<(g-1)/2$. Then the pair $(G_{X, \omega}, \omega(\delta_X))$ is not isomorphic to any of the above spin or half-spin groups with their spin or half-spin representations unless
\[
 (G_{X,\omega}^\ast,  \omega(\delta_X))\simeq (\Spin_{4m}^{\pm}(\bbF), \bbS_{4m}^{\pm})
 \quad \text{for some $m\in \{3,\dots, d\}$},
\]
in which case $X$ has topological Euler characteristic of absolute value $|e|=2^{2m-1}$ and is symmetric up to a translation, $d-m$ is even and $d\ge (g-1)/4$.
\end{theoremintro}

The main theorem in \cref{sec:BigTannakaIntro} now follows by combining theorems \ref{Thm:TannakaGroupSimpleIntro}, \ref{Thm:SmallWedgePowersAreSumsOfCurves_Intro}, \ref{Thm:Spin_Intro}, and from this we also obtain the main theorem in \cref{subsec:bigmonodromy} by the analog of the theorem of the fixed part given by \cref{Thm:FixedPart}.

 \medskip

\subsection{Conventions and notation}
We always work over a field $k$ of characteristic zero.
  A \emph{variety} over $k$ is a separated finite type $k$-scheme, and a \emph{subvariety} is a closed subvariety unless said otherwise. An \emph{algebraic group} is a finite type group scheme over a field.
 For a locally free sheaf $\cE$ (of finite rank) on a variety $X$, we denote by $\bbP(\cE):=\Proj \Sym^\bullet \cE^\vee$ the associated projective bundle.
 If~$A$ is an abelian variety over $k$, we denote  by $\Lie A$ its tangent space at the identity and define $\bbP_A:=\bbP((\Lie A)^\vee)$.
For a smooth projective connected variety $X$, we denote by $\Pic^0(X)$ the connected component of the identity in its Picard scheme. This is an abelian variety, and we denote by $\Alb(X)$ its dual abelian variety. Given a locally closed subvariety $Y$ of a variety $X$ over $k$, let $\cC_{Y/X}$ denote the \emph{conormal sheaf} of $Y$ in $X$, i.e.,   the $\cO_Y$-module $I / I^2$, where $I$ is the ideal sheaf of the closed immersion $i \colon Y \to U$ for a suitable open subset $U\subset X$.

\subsection*{Acknowledgments}
We would like to thank Daniele Agostini, Benjamin Bakker, Yohan Brunebarbe, Marco D'Addezio, Olivier Debarre and the referees for their helpful comments. 
A.J. gratefully acknowledges support by the IHES where part of this work was completed.
T.K. was supported by the DFG research grant Kr 4663/2-1.
C.L. was supported by the DFG research grants Le 3093/3-2 and Le 3093/5-1 and by the SMWK research grant SAXAG.
M.M. was supported by ANR grant ANR-18-CE40-0017.

\section{Gauss maps, positivity and nondegeneracy}

In this section, we recall from the view of conormal geometry various notions of positivity and nondegeneracy for subvarieties in abelian varieties. We denote by~$A$ an abelian variety over an algebraically closed field $k$ of characteristic zero.

\subsection{The stabilizer and the abelian variety generated} The \emph{stabilizer} of a subvariety $X\subset A$ is the algebraic subgroup $\Stab_A(X)\subset A$ whose $k$-points are 
\[
\Stab_A(X)(k) \;=\; \{ a\in A(k) \mid X + a = X \}.
\]
Write $\Stab(X)=\Stab_A(X)$ if the ambient abelian variety is clear from the context.

\begin{definition} We say $X\subset A$ is \emph{nondivisible} if it is integral and $\Stab(X)=\{0\}$.
\end{definition}

If $X\subset A$ is a connected subvariety, the \emph{abelian subvariety generated by $X$} is defined to be the smallest abelian subvariety $\langle X \rangle \subset A$ containing the image of the difference morphism $X \times X \to A$, $(x, x') \mapsto x - x'$. Note that this image $X-X \subset A$ is connected because $X\times X$ is so.

\subsection{Conormal varieties and Gauss maps} \label{Sec:GaussMap} 
Let us briefly recall the notion of conormal varieties and Gauss maps, which will be crucial later.  For abelian varieties, the cotangent bundle $\Omega^1_A$ is a trivial bundle with fiber $\rH^0(A, \Omega^1_A)=\coLie{A}$ of rank~$g=\dim A$. Consider the projection
\[
p \colon \quad \bbP(\Omega^1_A) \; \too \; \bbP_A = \bbP(\coLie{A}).
\]
If $\PLambda \subset \bbP(\Omega^1_A)$ is a $(g-1)$-dimensional integral subvariety, then for dimension reasons the morphism
\[
 \gamma_{\PLambda} \;:=\; p_{\vert \PLambda} \colon
 \quad \PLambda \;\too\; \bbP_A
\]
is either dominant (and hence generically finite) or not dominant. We say that~$\PLambda$ is \emph{clean} in the first case and \emph{negligible} in the second case. In the clean case we denote by $\deg \PLambda$ the generic degree of the generically finite dominant morphism ${\gamma}_{\PLambda}$, in the negligible case we formally put $\deg \PLambda = 0$.

\medskip 

We want to apply these definitions to conormal varieties, for which we need some more notation.
 For any subvariety $X \subset A$, its conormal sheaf $\cC_{X/A}$ fits in the exact sequence of coherent sheaves
\[ \cC_{X/A} \stackrel{i}{\too} \Omega^1_{A \rvert X} \too \Omega^1_X \too 0.\]
If $X \subset A$ is regular immersion, then $\cC_{X/A}$ is locally free and if $X$ is moreover integral, then $i$ is injective. If $X$ is smooth, then all three terms are locally free and the sequence is short exact.

\begin{definition} \label{Def:ConormalCycle} 
	For a reduced subvariety $X\subset A$ we define its (projective)  \emph{conormal variety} $\PLambda_X \subset \bbP(\Omega^1_A)$ to be the  closure of $\bbP(\cC_{X^\reg/A})$ in $\bbP(\Omega^1_A)$.  The \emph{Gauss map} of~$X$ is the morphism 
	$$\gamma_X \;:=\; {\gamma}_{\PLambda_X} \colon \quad \PLambda_X \;\too\; \bbP_A$$
	We denote by $\pr_X \colon \PLambda_X \to X$ the projection and $\PLambda_{X, x} := \pr_X^{-1}(x)$ for $x\in X(k)$. 
\end{definition}

\begin{remark} As we almost exclusively work with the projective conormal varieties and not with affine ones, we will usually drop the adjective \emph{projective}. We clearly have: 
	\begin{enumerate}
		\item The morphism $\gamma_{X \vert \PLambda_{X, x}} \colon \PLambda_{X, x} \to \bbP_A$ is injective.\smallskip
		\item If $X$ is smooth at a point $x$, then $\PLambda_{X, x} = \bbP(\cC_{X/A, x})$ where $\cC_{X/A, x}$ denotes the fiber at $x$ of the conormal bundle.\smallskip
		\item If $X$ is smooth, then $\PLambda_X = \bbP(\cC_{X/A})$.
	\end{enumerate}
\end{remark}

\noindent 
The effect of isogenies on conormal varieties is easy to control. For an integer $e\ge 1$ and an integral subvariety $X\subset A$ we denote by $[e](X)\subset A$ its image under the isogeny $[e]\colon A\to A$. We will always endow this image with the reduced subscheme structure, and we denote by $e_X := [e]_{\rvert X}\colon X \to [e](X)$ the finite morphism obtained by restriction of the isogeny to the given subvariety. By abuse of notation, we also denote by $[e]:A \times \bbP_A \to A \times \bbP_A$ the induced morphism. Then we have:

\begin{lemma} \label{lem:image-of-conormal-under-isogeny}
Let $X\subset A$ be an integral subvariety, and let $Y=[e](X)\subset A$ for an integer $e\ge 1$. Then we have an identity 
\[ [e]_\ast \PLambda_X \;=\; \deg(e_X)\cdot \PLambda_{Y}
\]
of cycles. In particular, if the subvariety $X\subset A$ is nondivisible, then $[e]_\ast \PLambda_X = \PLambda_Y$.
\end{lemma}

\begin{proof} 
The first claim follows easily from the fact that by construction the conormal variety to any integral subvariety is integral. The second claim is then clear because the morphism $e_X\colon X\to Y$ is birational if $X$ is nondivisible.
\end{proof}

\begin{corollary} \label{Corollary:Equidimensional-Fibers} Let $X\subset A$ be a smooth integral subvariety and $Y = [e](X)$ for an integer $e \ge 1$. Then the fibers of $\pr_Y\colon \PLambda_ Y \to Y$ are pure of dimension $\codim_A Y - 1$.
\end{corollary}
 
\begin{proof} 
\Cref{lem:image-of-conormal-under-isogeny} gives a commutative diagram
\[
\begin{tikzcd}
 \PLambda_X \ar[r, "{[e]}"] \ar[d, swap, "\pr_X"] & \PLambda_Y \ar[d, "\pr_Y"] \\
 X \ar[r, "e_X"] & Y
\end{tikzcd} 
\]
where the horizontal arrows are finite morphisms, and if $X$ is smooth, then the fibers of the morphism $\pr_X: \PLambda_X \to X$ are pure of dimension $\codim_A X - 1$.
\end{proof} 
 
%
\subsection{Positivity and nondegeneracy of subvarieties} \label{section positivity}

We now discuss various notions of positivity and nondegeneracy for subvarieties of an abelian variety. We say that an integral subvariety $X \subset A$ is {\em degenerate} if there exists a surjective morphism $\pi\colon A\to B$ of abelian varieties with 
\[ \dim \pi(X) < \min\{ \dim B, \dim X\}. \]
Otherwise, we say that $X$ is {\em nondegenerate}. Any closed point on the abelian variety is a nondegenerate subvariety, and so is the abelian variety  itself. Also note that if the abelian variety $A$ is simple, then any integral subvariety is nondegenerate. We say that a proper integral variety $X$ is  \emph{of general type} if there is a proper birational morphism $\nu \colon Y \to X$ from a smooth proper connected variety $Y$ with big canonical bundle.
For instance, we have:  \smallskip 
\begin{enumerate} 
	\item An integral effective divisor $X\subset A$ is nondegenerate if and only if it is ample. A curve $X\subset A$ is nondegenerate if and only if it generates $A$. See \cite[\S 1, examples]{Deb95}.\smallskip 
	\item For any elliptic curve $E$ and any simple abelian variety $B$ of dimension $\ge 3$, Debarre has constructed in~\cite[p. 189]{Deb95} a smooth subvariety 
	\[ X \subset A = E \times B \]
	of codimension $2$ which is nondegenerate but whose normal bundle is not ample. The smooth subvariety is obtained by choosing a general ample divisor $D\subset B$ and intersecting $E\times D$ with a general ample divisor in $A$.\smallskip 
	\item For $i = 1, 2$, let $A_i$ be an abelian variety and $X_i\subset A_i$ a nondegenerate integral subvariety. By considering the projections onto the factors, one sees that $X_1 \times X_2\subset A_1 \times A_2$ is of general type but degenerate.\smallskip
\end{enumerate}

\begin{remark} \label{Prop:NonDegerateIsogeny} 
Nondegeneracy is invariant under isogenies: Let $f \colon A \to A'$ be an isogeny of abelian varieties over $k$. Then an integral subvariety $X\subset A$ is nondegenerate if and only if $f(X)\subset A'$ is.
\end{remark}

In what follows we often consider the sum morphism $\sigma\colon X\times Y \to A$ for reduced subvarieties $X, Y\subset A$, and we denote by $X+Y \subset A$ its image. For nondegenerate subvarieties we have the following result by Debarre:

\begin{lemma} \label{Lem:SumOfNondegenerate}
Let $X, Y\subset A$ be integral subvarieties.
\begin{enumerate} 
\item If $X$ is nondegenerate, then $\dim(X+Y)=\min\{ \dim(X)+\dim(Y), \dim (A)\}$.
\item If $X$ and $Y$ are both nondegenerate, then so is $X+Y\subset A$.	
\end{enumerate} 	
\end{lemma} 

\begin{proof} 
See \cite[corollary~8.11]{DebarreAV}.	
\end{proof} 
The relations between the various notion of nondegeneracy and positivity that will play a role in this paper are summarized in the following diagram where for a smooth proper subvariety $X\subsetneq A$ we denote by $\cN_{X/A}$ its normal bundle:
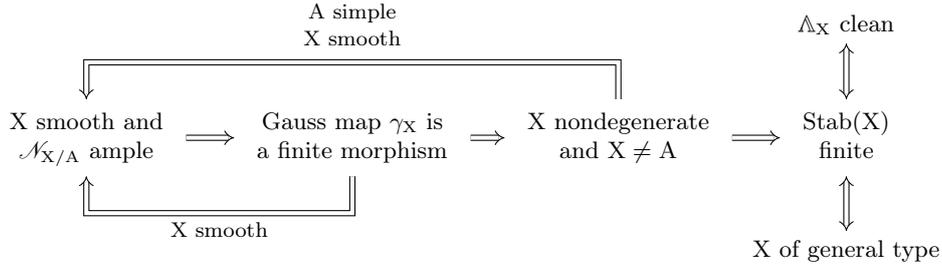
\[
\begin{tikzpicture}
\tikzstyle{every node}=[font=\small]
\def\r{1.25};
\node at (-10,0) (a) {\begin{tabular}{c}$X$ smooth and \\ $\cN_{X/A}$ ample \end{tabular}};
\node at (-6.5,0) (b) {\begin{tabular}{c} Gauss map $\gamma_X$ is \\ a finite morphism \end{tabular}};
\node at (-3, 0) (c) {\begin{tabular}{c} $X$ nondegenerate \end{tabular}};
\node at (0, -1.5) (d) {$X$ of general type};
\node at (0, 0) (e) {\begin{tabular}{c} $\Stab(X)$\\finite\end{tabular}};
\node at (0, 1.5) (f) {$\PLambda_X$ clean};

\draw[-implies,double equal sign distance] (a)--(b);
\draw[-implies,double equal sign distance] (b)--(c);
\draw[-implies,double equal sign distance] (c)--(e);
\draw[-implies,double equal sign distance] (b)--(-6.5, -1)--(-10, -1) node [midway, below] {\footnotesize $X$ smooth} --(a);
\draw[-implies,double equal sign distance] (c)--(-3, 1)--(-10, 1) node [midway, above] {\footnotesize \begin{tabular}{c}$A$ simple \\ $X$ smooth\end{tabular}} --(a);
\draw[implies-implies,double equal sign distance] (d)--(e);
\draw[implies-implies,double equal sign distance] (e)--(f);
\end{tikzpicture}
\]
More precisely, we have:

\begin{theorem} \label{Thm:PositivityNotions} Let $X\subset A$ be an integral subvariety with $0<\dim X < \dim A$.
	\begin{enumerate}
		\item The following are equivalent:
		\begin{enumerate}
			\item the conormal cone $\PLambda_X$ is clean;
			\item the algebraic group $\Stab(X)$ is finite;
			\item the variety $X$ is of general type.\smallskip
		\end{enumerate}
		\item If $X$ is nondegenerate, then $\Stab(X)$ is finite and $\langle X \rangle = A$.\smallskip
		\item If $\gamma_X \colon \PLambda_X \to \bbP_A$ is a finite morphism, then $X$ is nondegenerate.\smallskip
		\item Suppose $X$ smooth. Then the normal bundle $\cN_{X/A}$ is ample if and only the Gauss map $\gamma_X \colon \PLambda_X \to \bbP_A $ is a finite morphism. \smallskip
		\item If $A$ is a simple abelian variety and $X$ is of general type, then $X$ is   nondegenerate.  If $X$ is moreover smooth, then $\cN_{X/A}$ is ample.		
	\end{enumerate}	
\end{theorem}

\begin{proof}
	(1) The equivalence (a) $\Leftrightarrow$ (b) is shown in \cite[th.~1]{WeissauerArxiv2015a}, while  (b) $\Leftrightarrow$ (c) follows from Ueno's fibration theorem \cite[th.~3.10]{UenoCompositio}, \cite[th.~3]{Abr94}.\smallskip 
	
	(2) For the finiteness of the stabilizer,  denote by $p \colon A \to B:=A/\Stab(X)$ the quotient morphism. This quotient morphism is not surjective, since by construction we have $p^{-1}(p(X)) = X\neq A$. The nondegeneracy of $X$ then forces  $p \colon X \to \pi(X)$ to be generically finite, and it follows that $\Stab(X)$ is finite as desired. To show that $\langle X \rangle = A$, consider the quotient morphism $q \colon A \to A/\langle X \rangle$. The image $q(X)$ is a point, hence the nondegeneracy of $X$ and the assumption $\dim X > 0$ imply that $\dim A/\langle X \rangle = 0$, which shows that we have~$\langle X \rangle = A$.\smallskip  	
	
	(3) We prove the contrapositive. If $X\subset A$ is degenerate,  then there is a surjective morphism $\pi \colon A \to B$ of abelian varieties such that 
	$\dim Y < \min \{ \dim B,  \dim X \}$,
	where $Y := \pi(X)$. We have the following commutative of $\cO_X$-modules with exact rows 
	\[
	\begin{tikzcd}
	( \pi^\ast \cC_{Y/B})_{\rvert X} \ar[r, "j"] \ar[d, "\epsilon"]& (\pi^\ast \Omega^1_B)_{\rvert X} \ar[r] \ar[d, "\rd \pi"] & (\pi^\ast \Omega^1_{Y})_{\rvert X} \ar[r] \ar[d]& 0\\
	\cC_{X/A} \ar[r, "i"] & \Omega^1_{A \rvert X} \ar[r] & \Omega^1_X \ar[r] & 0 
	\end{tikzcd}
	\]
	where $\rd \pi$ is the pull-back of differential forms along $\pi$. Here $i$ is injective over the smooth locus $X^\reg \subset X$, and likewise $j$ is injective over $\pi^{-1}(Y^\reg)$: Indeed, the short exact sequence
	\[ 0 \too (\cC_{Y / B})_{\rvert Y^\reg} \too \Omega^1_{B \rvert Y^\reg} \too \Omega^1_{Y^\reg} \too 0\]
	of $\cO_{Y^\reg}$-modules is locally split because the $\cO_{Y^\reg}$-module $\Omega^1_{Y^\reg}$ is locally free; the pull-back along $\pi$ of the above short exact sequence hence stays exact. It follows that $\epsilon$ is also injective over the nonempty open subset 
	\[U := X^\reg \cap \pi^{-1}(Y^\reg).\]
	The hypothesis $\dim Y < \dim X$ implies that the induced morphism $\pi_{\rvert U} \colon U \to Y^\reg$ is not generically finite. Thus, for $y$ ranging over a dense open subset of $Y^\reg$, the fiber $Z:= \pi^{-1}(y) \cap U$ is positive-dimensional. 
	Pick a nonzero vector $v \in \cC_{Y, y}$, which exists because $\dim Y < \dim B$. Then
	\[
	 0\;\neq\; j(v) \;\in\; \bigcap_{x\in Z} \cC_{X, x}.
	\] 
	Thus, if we denote by $F:=\pr_X(\gamma_X^{-1}([j(v)])\subset X$ the image of $\gamma_X^{-1}([j(v)])$ under the projection $\PLambda_X \to X$, then the subset $Z$ is contained in $F$. This shows that the dimension of $\gamma_X^{-1}([j(v)])$ is positive.\smallskip

	(4) Since $X$ is smooth we have $\PLambda_X = \bbP(\cC_{X/A})$. The normal bundle $\cN_{X/A}$ is globally generated, thus the equivalence is \cite[Example 6.1.5]{LazarsfeldPositivityII}. \smallskip
	
	(5) When $A$ is a simple abelian variety, any integral subvariety is nondegenerate, and the ampleness of the normal bundle of a smooth subvariety $X$ in $A$ follows from \cite[prop.~4.1]{Har71}. 
\end{proof}

\subsection{Symmetric powers of curves in abelian varieties} We show here that symmetric powers of a (smooth projective) curve $C$ cannot be embedded as a complete intersection of ample divisors as claimed in \cref{sec:IntroWedgePowers}. Recall that the curve $C$ has gonality $\ge n + 1$ if and only if the sum map $\Sym^n C \to X := C + \cdots + C \subset \Pic^0(C)$ is an isomorphism. If so the normal bundle of  $X$ is ample \cite[\S 1, Examples (2)]{Deb95}. Imposing further positivity properties to the normal bundle is far more restrictive:

\begin{proposition} \label{prop:SymmetricPowersOfCurvesAreNotCompleteIntersections} 
Let $C \subset A$ a smooth irreducible projective curve such that the sum morphism $\Sym^n C \to X:= C + \cdots + C \subset A$ is an isomorphism for some $n\ge 2$. Then $C$ is nonhyperelliptic of genus $g \ge 3$ and the following hold:
\begin{enumerate}
\item If the normal bundle $\cN_{X/A} = \cV_1 \oplus \dots \oplus \cV_r$ is a direct sum of ample vector bundles, then
\[ n \le \max_{i = 1, \dots, r} \rk \cV_i + 1.\]
\item The normal bundle $\cN_{X/A}$ is a direct sum of ample line bundles if and only if $g=3$, $n=2$, and $A$ is isomorphic to $\Pic^0(C)$.
\end{enumerate}
\end{proposition}

\begin{proof} By Lefschetz's principle, we may assume $k = \bbC$. First of all, the curve $C$ is nonhyperelliptic of genus $g \ge 3$. Otherwise, $C$ would be symmetric when suitably embedded in its Jacobian. In particular, the sum morphism would contract the antidiagonal and thus would not induce an isomorphism $\Sym^n C \iso C + \cdots + C$.

\medskip

(1) Arguing by contradiction, suppose the inequality in the statement does not hold.  Then we can apply the Barth-Lefschetz theorem \cite[th.~4.5]{Deb95} to obtain isomorphisms
\[ \rH^i(A) \iso \rH^i(X), \qquad i = 1, 2,\]
of rational cohomology groups. On the other hand, the computation of cohomology of symmetric powers of curves \cite[1.2]{MacdonaldSymmetric} yields the following expressions:
\begin{align*}
\rH^1(X) &= \rH^1(\Sym^n C) \iso \rH^1(C^n)^{\frS_n} =  \rH^0(C) \otimes \rH^1(C),\\
\rH^2(X) &= \rH^2(\Sym^n C) \iso \rH^2(C^n)^{\frS_n} = \Alt^2 \rH^1(C) \oplus \rH^0(C) \otimes \rH^2(C)^{n - 1}.
\end{align*}
Recalling the equality $\rH^2(A) = \Alt^2 \rH^1(A)$ we obtain a contradiction.

\medskip

(2) If $C$ has genus $g=3$, the subvariety $C+C\subset \Pic^0(C)$ is a theta divisor and hence ample. Conversely, suppose that the normal bundle $\cN_{X/A}$ is a direct sum of ample line bundles. We first claim that then $A$ is isogenous to~$\Pic^0(C)$. Indeed, as above we have isomorphisms
\[  \rH^1(A) \iso \rH^1(X) \iso \rH^1(C). \]
Now we cannot conclude as in (1) because the Barth-Lefschetz theorem here only says that $\rH^2(A) \to \rH^2(X)$ is injective. Instead, write $\cN_{X / A} = \cL_1 \oplus \cdots \oplus \cL_{g-2}$ for ample line bundles $\cL_i$ on $X$. By looking at the short exact sequence
\[ 0 \too \rT_X \too \Lie A \otimes \cO_X \too \cN_{X / A} = \cL_1 \oplus \cdots \oplus \cL_{g-2} \too 0, \]
we see that the line bundles $\cL_i$ are globally generated and
\begin{equation} \label{Eq:CanonicalIsSumOfLineBundles} \cL_1 \otimes \cdots \otimes \cL_{g-2} \iso \cK_X  \end{equation}
where $\cK_X = \Alt^2 \Omega^1_X$ is the canonical bundle on $X$. We identify $X$ with $\Sym^2 C$ and write $\pi \colon C \times C \to X$ for the quotient morphism. Since $\pi$ ramifies exactly on the diagonal $\Delta$ of $C \times C$, we have $\pi^\ast \cK_X = \cK_{C \times C}(- \Delta)$. Let us fix a point $p \in C(k)$ and consider the embedding $f \colon C \to C \times C$, $x \mapsto (x, p)$. Then
\begin{equation} \label{Eq:CanonicalOfSymmetricSquare} f^\ast \pi^\ast \cK_X = \cK_C(-p). \end{equation}
On the other hand, for $i = 1, \dots, g- 2$, the line bundle $\cM_i := f^\ast \pi^\ast \cL_i$ on $C$  is ample and globally generated. Moreover, the curve $C$ being nonhyperelliptic, we necessarily have $\deg \cM_i \ge 3$. By combining \eqref{Eq:CanonicalIsSumOfLineBundles} and \eqref{Eq:CanonicalOfSymmetricSquare} and then by taking degrees, we obtain the inequality
\[ 2g - 3 = \deg \cK_C(-p) = \sum_{i = 1}^{g-2} \deg \cM_i \ge 3(g - 2).\]
This forces $g = 3$. For a suitable Abel-Jacobi embedding  $C \hookrightarrow \Pic^0(C)$, there exists an isogeny $\phi\colon\Pic^0(C) \to A$ such that the following diagram commutes:
	\[
	\begin{tikzcd}
	\Sym^2 C \ar[r, "\sim"] \ar[d, equal] & \Theta   \ar[d, "\wr"] \ar[r, hook] & \Pic^0(C) \ar[d, "\phi"] \\
	\Sym^2 C \ar[r, "\sim"]  & X \ar[r, hook]& A
	\end{tikzcd}
	\]
Here the leftmost horizontal arrows are induced by the sum and $\Theta \subset \Pic^0(C)$ is a theta divisor. The preimage $\phi^{-1}(X)$ is smooth, thus its connected components are irreducible. As $\Theta$ is one of them, the others are $\Theta + a$ for $a \in \ker \phi$. Since any two translates of an ample divisor meet, we have $\Theta = \phi^{-1}(X)$. But the isogeny $\phi$ induces an isomorphism $\Theta \iso X$, thus $\phi$ must be injective.
\end{proof}

\begin{corollary} \label{Cor:SymmetricPowersOfCurvesAreNotCompleteIntersections} Let $C \subset A$ be a smooth irreducible projective curve such that the sum morphism 
$\Sym^n C \to X:= C + \cdots + C \subset A$
is an isomorphism for some $n\ge 2$. Then $C$ is nonhyperelliptic of genus $g\ge 3$ and the following are equivalent:
\begin{enumerate}
\item The subvariety $X \subset A$ is a complete intersection of ample divisors.
\item We have $g=3$, $n = 2$, and $A$ is isomorphic to $\Pic^0(C)$.
\end{enumerate}
\end{corollary}

\begin{proof} 
For complete intersections of ample divisors, the normal bundle is a direct sum of ample line bundles. Hence, \cref{prop:SymmetricPowersOfCurvesAreNotCompleteIntersections} (2) applies.	
\end{proof} 

As an amusing aside, of no use in what follows, note that \cref{prop:SymmetricPowersOfCurvesAreNotCompleteIntersections} implies the classical bound for the gonality of a smooth projective curve:

\begin{corollary} A smooth projective curve of genus $g$ has gonality $\le (g + 3) / 2$.
\end{corollary}

\begin{proof} As already mentioned, the curve $C$ has gonality $\ge n + 1$ if and only if the sum morphism $\Sym^n C \to X := C + \cdots + C\subset \Pic^0(C)$ is an isomorphism, and if this is the case, then $X$ has ample normal bundle in $\Pic^0(C)$. Since the normal bundle has rank $g - n$ \cref{prop:SymmetricPowersOfCurvesAreNotCompleteIntersections} (1) implies $n \le g - n + 1$, that is $n + 1 \le (g + 3)/2$.
\end{proof}

In particular \cref{prop:SymmetricPowersOfCurvesAreNotCompleteIntersections} (1) is sharp in the two extremal cases---that of an indecomposable ample normal bundle and that of a sum of ample line bundles.

\subsection{Bounds for the topological Euler characteristic} We now pass to some numerics concerning the topological Euler characteristic of complete intersections. To ease notation below,  we define $g:= \dim A$. For a smooth subvariety $X \subset A$, let $e_X$ denote its topological Euler characteristic. By definition it is the top Chern class of the tangent bundle $T_X$ of $X$. Consider the short exact sequence of vector bundles on $X$,
\[ 0 \too T_X \too T_{A \rvert X} \too \cN_{X/A} \too 0.\]
Since the total Chern class is multiplicative in short exact sequences and the tangent bundle of $A$ is trivial, we have 
\[c(T_X) = c(T_{A \rvert X})c(\cN_{X/A})^{-1} = c(\cN_{X/A})^{-1}. \]
First of all, note that we have the following lower bound whenever the normal bundle is ample.

\begin{lemma} \label{Lemma:LowerBoundTopCharAmpleNormalBundle} Let $X \subsetneq A$ be a $d$-dimensional smooth subvariety with ample normal bundle. Then
\[ |e_X| \ge \max \{ g, 2^{\min \{ d, \lfloor \sqrt{g-1} \rfloor \} } \}.\]
\end{lemma} 

\begin{proof} We may suppose $k = \bbC$. By definition the inverse of the total Chern class is the total Segre class. We have $|e_X| = (-1)^d s_d(\cN_{X/A}) = s_d(\cN_{X/A}^\vee)$
where $d = \dim X$ and $s_d$ is the $d$-th Segre class. Now the normal bundle $\cN_{X/A}$ is ample and globally generated. Since $\rH^1(X, \bbC) \neq 0$ we have $|e_X| \ge g$ by \cite[Theorem~4]{BeltramettiSchneiderSommese}. According to \cite[Prop. 2.4]{EinLazarsfeldBo} we also have $|e_X| \ge 2^{\min \{ d, \lfloor \sqrt{g-1} \rfloor \}}$ because the cotangent bundle of $X$ is nef.\footnote{Beware that in both references the authors adopt the convention dual to the one in \cite{FultonIntersectionTheory} for the definition of Segre classes.}
\end{proof}

The previous lower bound is doubtlessly not sharp. Indeed for a smooth projective curve $X$ generating $A$ we have $|e_X| \ge 2g - 2$. For surfaces we have:

\begin{lemma} \label{Lemma:LowerBoundTopCharSurfaces} Let $X \subset A$ be a smooth projective surface generating $A$ and with finite stabilizer. Then
\[ e_X \ge 3g-9. \]
\end{lemma}

\begin{proof} Write $c_1 = c_1(T_X)$ and $c_2 = c_2(T_X) = e_X$ and $\chi = \chi(X, \cO_X)$ as usual. By \Cref{Thm:PositivityNotions} the surface $X$ is of general type. Thus the Bogomolov-Miyaoka-Yau inequality gives $c_1^2 \le 3 c_2$ which is equivalent to $3 \chi \le c_2$ by Noether's formula. On the other hand, let us write $q = h^1(X, \cO_X)$ and $p = h^2(X, \cO_X)$ so that $\chi = 1 - q + p$. The surface $X$ is minimal, thus we can apply the inequality $p \ge 2q - 4$ (see Beauville's appendix to \cite{DebarreNoether} for a proof), which is equivalent to $\chi \ge q - 3$. Combining these inequalities yields $c_2 \ge 3(q-3)$. Since $X$ generates $A$ by hypothesis, we have $q \ge g$ which concludes the proof.
\end{proof}

When the subvariety is a complete intersection of ample divisors the previous lower bounds can be drastically improved. In order to show this, for integers $n \ge 2$ and $r \in \{ 1, \dots, n-1\}$, consider the following subset of partitions of $n$,
\[ P(n, r) \; := \; \{ a = (a_1, \dots, a_r) \in \bbZ^r \mid a_1, \dots, a_r \ge 1, a_1 + \cdots + a_r = n \}. \]
Note that $P(n, r)$ has cardinality $\binom{n - 1}{n-r}$.

\begin{lemma} \label{Lemma:EulerCharCompleteIntersection} Let $X$ be a smooth complete intersection of ample divisors $D_1, \dots, D_r$ in $A$. Then
\[
e_X \;=\; (-1)^{\dim X}  \sum_{a \in P(g, r)} D_1^{a_1} \cdots D_r^{a_r}.
\]
\end{lemma}

\begin{proof} The hypothesis of $X$ being a complete intersection of the divisors $D_1, \dots, D_r$ implies that the normal bundle $\cN_{X/A}$ is the direct sum of (the restriction to $X$ of) the line bundles $\cO(D_1), \dots, \cO(D_r)$. In particular,
\[ c(\cN_{X/A}) = c(\cO(D_1)) \cdots c(\cO(D_r)) = (1 + D_1) \cdots (1 + D_r) \in \CH(X).\]
By inverting formally $1 + D_i$ we find the following expression
\[ c(T_X) = \sum_{n = 0}^{g - r} (-1)^{n}  \sum_{\substack{a_1, \dots, a_r \ge 0 \\ a_1 + \cdots + a_r = n}} D_1^{a_1} \cdots D_r^{a_r} \in \CH(X).\]
Looking at it in the Chow ring of $A$ amounts to multiplying it by $D_1 \cdots D_r$. We conclude by then taking the piece of degree $g$.
\end{proof}

Recall that, for an ample divisor $D \subset A$, the   self-intersection   $D^g$ is positive and divisible by $g!$,  as the ratio $D^g / g!$ is given by $h^0(A, \cO(D))$.

\begin{lemma} \label{Lemma:TeissierKhovanskii} For ample divisors $D_1, \dots, D_g \subset A$, we have $D_1 \cdots D_g \ge g!$.
\end{lemma}

\begin{proof} The Khovanskii-Teissier inequality \cite[Theorem 1.6.1]{LazarsfeldPositivityI} states that the lower bound $(D_1 \cdots D_g)^g \ge D_1^g \cdots D_g^g$ holds. Since each factor on the right-hand side is a positive multiple of $g!$,  this concludes the proof.
\end{proof}

\begin{proposition} \label{Cor:LowerBoundEulerChar} Let $X \subsetneq A$ be a smooth complete intersection of ample divisors of dimension $d \ge 1$. Then $e_X$ is even and 
\[ |e_X| \ge g! \tbinom{g-1}{d}.\]
\end{proposition}

\begin{proof} By assumption $X$ is a complete intersection of ample divisors, say $D_1, \dots, D_r$ where $r = g - d$ is the codimension of $X$. \Cref{Lemma:EulerCharCompleteIntersection} shows
\[ e_X \;=\; (-1)^{d}  \sum_{a \in P(g, r)} D_1^{a_1} \cdots D_r^{a_r}. \]
Since the divisors $D_1, \dots, D_r$ are ample, by \cref{Lemma:TeissierKhovanskii} we have $D_1^{a_1} \cdots D_r^{a_r} \ge g! $ for each $a \in P(g, r)$. Since the cardinality of $P(g, g-d)$ is $\binom{g-1}{d}$, the inequality in the statement follows. For the parity of $e_X$, by the Lefschetz principle, we may assume $k = \bbC$. Then each $[D_i]^{a_i} \in \rH^{2a_i}(A, \bbZ)$ is divisible by $a_i!$. Since $d \ge 1$, for each $a \in P(g, r)$ we have $a_i \ge 2$ for some $i$, thus we conclude that $e_X$ is even. 
\end{proof}

By Proposition \ref{Cor:LowerBoundEulerChar}, the absolute value of the Euler characteristic of a smooth connected complete intersection of ample divisors in $A$ is never equal to $27$.  We now prove that $|e_X| \neq 56$,  except in the case of curves in abelian surfaces and abelian threefolds (in which case there are examples). 

\begin{corollary} \label{Cor:EulerCharIsNot56} If $X\subsetneq A$ is a smooth complete intersection of ample divisors of dimension $d\geq 1$ and $(d,g)\neq (1,2), (1,3)$, then $|e_X|\neq 56$. 
\end{corollary}

\begin{proof} \Cref{Cor:LowerBoundEulerChar} implies $|e_X| \ge g! \binom{g-1}{d}$ which settles the matter for $g \ge 5$. On the other hand, if $X$ is itself a divisor, that is $d = g-1$, then $|e_X| = X^g$ is divisible by $g!$. The only two cases left are $(d, g) = (1, 4), (2, 4)$ for which $g! \binom{g-1}{d} = 72$. 
\end{proof}

 \Cref{Cor:LowerBoundEulerChar} furnishes another proof of \cref{Cor:SymmetricPowersOfCurvesAreNotCompleteIntersections}. Indeed the $n$-th symmetric power of a smooth projective curve of genus $g \ge 2$ has topological Euler characteristic $(-1)^{n}\binom{2g-2}{n}$; see \cite[4.4]{MacdonaldSymmetric}. Using that the gonality is $\le (g+3)/2$ we conclude because, for $g \ge 4$ and $n \le (g+1)/2$, we have $\tbinom{2g-2}{n} < g! \tbinom{g-1}{n}$.

\section{Perverse sheaves on abelian varieties}\label{section perv and groups}

In this section, we collect some general results about perverse sheaves on abelian varieties. We work over a field $k$ with $\characteristic(k)=0$, but as in~\cite[section~3]{LS20} we do not require this field to be algebraically closed; for  the relation with monodromy groups we will later need to work over function fields. For any variety $X$ over $k$ we denote by
\[
 \Perv(X, \bbF) \;\subset\; \rD^b_c(X, \bbF)
\]
the abelian category of perverse sheaves with coefficients in $\bbF = \overline{\bbQ}_\ell$ for a fixed prime number $\ell$.  For $k=\bbC$, we will later also consider perverse sheaves in the analytic sense with coefficients in $\bbF = \bbC$, and we will use the above notation also in this case. The results below work both in the $\ell$-adic setting over any field $k$ and in the analytic setting with $k=\bbF=\bbC$.
We let $\pi_1(A, 0)$ be the \'etale resp.~topological fundamental group in the two settings, with the profinite resp.~discrete topology, and write $\Pi(A, \bbF)=\Hom(\pi_1(A, 0), \bbF^\times)$ for the group of its continuous characters.

\medskip 

\subsection{Convolution on abelian varieties} \label{section perverse sheaves on abelian varieties}

For convenience, let us briefly recall the Tannakian description of perverse sheaves on abelian varieties $X=A$ given in~\cite{KWVanishing}.
The sum morphism $\sigma\colon A\times A \to A$ induces a convolution product
\[
*: \Dbc(A,\bbF) \times \Dbc(A,\bbF) \too \Dbc(A,\bbF), \quad K_1 * K_2:= R\sigma_* \left( K_1 \boxtimes K_2\right) 
\]
which endows the derived category with the structure of a rigid symmetric monoidal category~\cite{WeissauerRigidity} (in loc.~cit.~this is stated only over algebraically closed fields~$k$, but the proof works in the general case without changes). The subcategory of perverse sheaves is not stable under the convolution product, but it becomes so after passing to a certain quotient category. To explain this, recall that for any $P\in \Perv(A, \bbF)$ we have
\[
 \chi(A, P) \;:=\; \sum_{i\in \bbZ} \, (-1)^i \dim_\bbF H^i(A, P) \;\ge\; 0.
\]
Indeed, over $k=\bbC$ this was observed by Franecki and Kapranov~\cite[cor.~1.4]{FraneckiKapranov}; the case of an arbitrary algebraically closed field~$k$ of characteristic $0$ can be reduced to the complex case by choosing a model over some algebraically closed subfield of~$k$ which embeds into the complex numbers, see \cref{lem:reduction-to-complex-case}. The additivity of the Euler characteristic in short exact sequences then implies that perverse sheaves of Euler characteristic zero form a Serre subcategory
\[
\rmS(A, \bbF) \;:=\; \{ P \in \Perv(A, \bbF) \mid \chi(A, P)=0\}
\;\subset\; 
\Perv(A, \bbF)
\]
inside the abelian category of perverse sheaves. Let $\rmT(A, \bbF) \subset \Dbc(A, \bbF)$ be the full subcategory of sheaf complexes whose perverse cohomology sheaves are in $\rmS(A, \bbF)$; its objects will be called {\em negligible sheaf complexes}.  
\begin{proposition}  \label{prop:tannakian}
The triangulated quotient category $\Dbcbar(A, \bbF):=\Dbc(A, \bbF)/\rmT(A, \bbF)$ inherits
from the perverse $t$-structure on the derived category  a $t$-structure whose heart
\[
 \Pbar(A, \bbF) \;\subset\; \Dbcbar(A, \bbF)
\]
is equivalent to the abelian quotient category $\Perv(A, \bbF)/\rmS(A, \bbF)$. It also inherits the structure of a rigid symmetric monoidal category with respect to a convolution product 
\[
 *: \Dbcbar(A, \bbF) \times \Dbcbar(A, \bbF) \too \Dbcbar(A, \bbF)
\]
induced by the convolution product on the derived category. On the triangulated quotient category, this product is $t$-exact in both of its arguments. Thus, it  restricts to a product
\[
 *: \Pbar(A, \bbF) \times \Pbar(A, \bbF) \too \Pbar(A, \bbF),
\]
and $\Pbar(A, \bbF)$ is a neutral Tannaka category with respect to this product.
\end{proposition}

\begin{proof} 
Fix an algebraic closure $K\supset k$. Then the functor $\Dbc(A, \bbF)\to \Dbc(A_K, \bbF)$ is exact for the perverse $t$-structure, compatible with the convolution product, and preserves the subcategories of negligible objects. Hence, the result follows from the statement over algebraically closed fields in~\cite{KWVanishing, KraemerSemiabelian}; note that by Deligne's internal characterization of neutral Tannaka categories~\cite[\S 6.4]{Coulembier}, it suffices to construct a fiber functor on every finitely generated tensor subcategory.
\end{proof} 

In what follows, by an {\em abelian tensor category} we mean a rigid symmetric monoidal abelian $\bbF$-linear category.

\subsection{Tannaka groups of perverse sheaves}\label{subsec:tannaka of perv}

Let $\cC \subset \Pbar(A, \bbF)$ be a full abelian tensor subcategory and
\[
 \omega\colon \quad \mathcal{C} \;\too\; \Vect(\bbF)
\]
a given fiber functor on this subcategory. The existence of such fiber functors is guaranteed by \cref{prop:tannakian}; there is no canonical choice of such a fiber functor, but any two fiber functors on a neutral Tannaka category over an algebraically closed field $\bbF$ are noncanonically isomorphic~\cite[th.~3.2.(b)]{DM82}. Once we have chosen a fiber functor, we get an equivalence of abelian tensor categories between~$\cC$ and the category $\Rep_\bbF(G_\omega(\cC))$ of finite-dimensional algebraic representations of the affine group scheme 
\[
 G_\omega(\cC) \;:=\; \Aut^\otimes(\omega)
\] 
over $\bbF$ called the \emph{Tannaka group} of $\cC$. We are interested in algebraic quotients of this proalgebraic group scheme:

\begin{definition} \label{rem:definition-of-tannaka-groups}
For any $P\in \cC$, we obtain from the above construction an affine algebraic group 
\[
  G_\omega(P)  \;:=\; \im\bigl(G_\omega(\cC) \to \GL(\omega(P))\bigr) 
\]
over $\bbF$ with a faithful representation on the vector space $\omega(P)\in \Vect(\bbF)$ whose dimension is the Euler characteristic
\[
 \dim_\bbF(\omega(P)) \;=\; \chi(A, P),
\]
see~\cite[proof of cor.~4.2]{KWVanishing}. Let us denote by $\iota\colon \langle P \rangle \into \cC$ the smallest abelian tensor subcategory which contains the object $P$ and is stable under subobjects and quotients. Then $G_\omega(P)=G_{\omega\circ \iota}(\langle P\rangle)$ for the fiber functor $\omega\circ \iota\colon \langle P \rangle \to \Vect(\bbF)$ and we have a commutative diagram of abelian tensor categories:
\[
\begin{tikzcd}
 \langle P \rangle \ar[r, "\sim"] \ar[d, hook] & \Rep_\bbF(G_\omega(P)) \ar[d, hook] \\ 
 \cC \ar[r, "\sim"] & \Rep_\bbF(G_\omega(\cC))
\end{tikzcd}
\]
If $P\in \cC$ is a simple object, then the faithful representation $\omega(P)\in \Rep_\bbF(G_\omega(P))$ is irreducible and then $G_\omega(P)$ is reductive by~\cite[19.1~prop.~(b)]{Hum78}. This is in particular the case when $P = \delta_X$ is the intersection complex of an integral subvariety $X \subset A$, in which case we write \[ G_{X, \omega} := G_\omega(\delta_X).\]
\end{definition}

\noindent
For the rest of this section, we fix a full abelian tensor subcategory $\cC\subset \Pbar(A, \bbF)$ and a fiber functor $\omega\colon \cC \to \Vect(\bbF)$. When there is no risk of confusion, we also write $\omega$ for the restriction of the given fiber functor to any subcategory of $\cC$. 

\subsection{The derived group of the connected component} It is often convenient to pass from arbitrary reductive groups to connected semisimple groups: For a reductive group $G$, let $G^\circ\subset G$ be its connected component of the identity,  and note  that the derived group
\[
 G^\ast \;:=\; [G^\circ, G^\circ]
\] 
is   a connected semisimple group.  For the reductive Tannaka groups from section~\ref{subsec:tannaka of perv} we will understand the connected components and the center in terms of direct images of perverse sheaves under the morphisms $[d]\colon A\to A, x\mapsto dx$ for~$d\in \bbN$ and~$t_a\colon A\to A, x\mapsto x+a$ for $a\in A(k)$. For a perverse sheaf $Q\in \Perv(A, \bbF)$  and a point $a\in A(k)$, we define
\[
 Q_a \;:=\; t_{a*} P 
\]
and we say that $Q$ is {\em nondivisible} if it is simple and satisfies $Q_a\not \simeq Q$ for all $a\in A(k)$ with~$a\neq 0$. We denote by
\[
 \Gamma_P := \{ a\in A(k)_\tors \mid \delta_a \in \langle P \rangle \}
\]
the abelian group of torsion points whose associated skyscraper sheaf appears in the Tannaka category $\langle P \rangle$ generated by a perverse sheaf $P\in \cC$. Note that $\Gamma_P$ is finite: Indeed, every skyscraper sheaf $\delta_a \in \langle P \rangle$ defines a character of the Tannaka group $G_\omega(P)$ and algebraic groups have only finitely many torsion characters.~In fact the first part of the following result shows that {\em all} torsion characters of the Tannaka group are given by skyscraper sheaves in torsion points:

\begin{proposition} \label{prop:connected-component}
Let $k$ be algebraically closed and $P\in \cC$ a simple perverse sheaf.
\begin{enumerate} 
\item The group of connected components of the Tannaka group $G:=G_\omega(P)$ is given by
$$G/G^\circ \;\simeq\; \Hom(\Gamma_P, \bbG_m).$$ 
\item Fix an integer $d\ge 1$ with $d\cdot \Gamma_P = \{0\}$. Then for all $Q, Q'\in \langle P\rangle$ we have: 
\smallskip
\begin{eqnarray*}
\omega(Q)_{\rvert G^\circ} \;\simeq\; \omega(Q')_{\rvert G^\circ} 
&\quad \Longleftrightarrow \quad  &
[d]_* Q \;\simeq\; [d]_* Q' ,
\\[0.3em]
\textnormal{\em $\omega(Q)_{\rvert G^\circ}$ is irreducible} 
&\quad \Longleftrightarrow \quad &
\textnormal{\em $Q$ is nondivisible}.
\end{eqnarray*}
\item Let $\det(P)\in \langle P\rangle$ be the unique simple perverse sheaf which corresponds to the top wedge power of~$V:=\omega(P)$.  Then $\det(P)$ is a skyscraper sheaf. If $V_{\rvert G^\circ}$ is irreducible, we have:
\[
\textnormal{\em $G^\circ$ semisimple} \quad \Longleftrightarrow \quad 
\textnormal{\em $\Supp(\det(P))$ is a torsion point}.
\]
\end{enumerate} 
\end{proposition} 

\begin{proof} 
For $k=\bbC$, parts (1) and (2) are due to Weissauer~\cite{WeissauerAlmostConnected} who also shows that every invertible object in the Tannaka category of perverse sheaves is a skyscraper sheaf  (this in particular applies to $\det(P)$); alternatively one could use the Riemann-Hilbert correspondence and the results for holonomic $\mathcal{D}$-modules in~\cite[section~3.c]{KraemerMicrolocalI}. From $k=\bbC$ one can pass to an arbitrary algebraically closed field of characteristic zero because the Tannaka group is invariant under extensions of algebraically closed fields and any perverse sheaf is defined over the algebraic closure of a finitely generated field, see \cref{cor:algebraically-closed} resp.~\cref{lem:reduction-to-complex-case}.
The claim about semisimplicity in (3) follows since by Schur's lemma the center $Z=Z(G^\circ)$ acts on $V$ by scalars and hence $\det(V)$ has finite order if and only if $Z$ is finite.	
\end{proof}

\begin{definition}
For perverse sheaves $P\in \cC$ we denote the derived group of the connected component of the Tannaka group $G=G_\omega(P)$ by
\[
 G_\omega^\ast(P) \;:=\; [G^\circ, G^\circ].\]
 If $P=\delta_X$ is the intersection complex of a subvariety $X\subset A$, we put $G_{X,\omega}^\ast :=G_\omega^\ast(P)$.  
\end{definition}

\Cref{prop:connected-component} allows us to realize this group
as the Tannaka group of another perverse sheaf:

\begin{corollary} \label{cor:derived-group}
Suppose $k$ is algebraically closed. Let $P \in \cC$ be a simple perverse sheaf. Then for any integer $d\ge 1$ with $[d]_*P\in \cC$ and any $a\in A(k)$ with $P_a \in \cC$ the following properties hold:
\begin{enumerate} 
\item $G_\omega^\ast(P_a)\simeq G_\omega^\ast(P)$.\smallskip
\item $G_\omega^\circ([d]_* P) \simeq G_\omega^\circ(P)$.\smallskip 
\item $G_\omega([d]_* P)$ is connected if and only if $d\cdot \Gamma_P = 0$.\smallskip
\item Suppose $P$ is nondivisible with  $[d]_*\det(P_a)=\delta_0$ and $d\cdot \Gamma_{P_a} = 0$. If $[d]_* P_a$ belongs to $\cC$, then
\[ 
 G_\omega([d]_* P_a) \;\simeq\; G_\omega^\ast(P).
\]
\end{enumerate} 
\end{corollary} 

\begin{proof}
(1) By~\cite[lemma~4.3.2]{KraemerMicrolocalII} the inclusions $\langle P\rangle \subset \langle P\oplus \delta_a \rangle \supset \langle P_a\rangle$ induce isomorphisms 
$
 G_\omega^\ast(P) \simeq
 G_\omega^\ast(P\oplus \delta_a) \simeq
 G_\omega^\ast(P_a)
$.

\medskip 

(2) By~\cite{WeissauerAlmostConnected} or~\cite[cor.~1.6]{KraemerMicrolocalI}, the pushforward 
$
 [d]_*\colon \langle P \rangle \to \langle [d]_* P \rangle$ 
is a tensor functor which induces an isomorphism between the connected components of the identity of the respective Tannaka groups. 

\medskip 

(3) This follows from \cref{prop:connected-component} (1) applied to $[d]_*P$ since $d\cdot \Gamma_P = \Gamma_{[d]_*P}$.

\medskip 

(4) By the previous two steps, the group $G_\omega([d]_* P_a)$ is connected. One easily sees that the perverse sheaf $[d]_* P_a$ is nondivisible with $\det([d]_* P_a)=[d]_* \det(P_a)=\delta_0$ so that $G_\omega([d]_* P_a)$ is a semisimple group by the last part of \cref{prop:connected-component}. It is therefore equal to the derived group of its connected component of the identity, which by (1) and (2) coincides with $G_\omega^\ast(P)$.
\end{proof}

\begin{remark} \label{rem:connected-component-functorial}
The isomorphism $G_\omega^\circ([d]_*P)\simeq G_\omega^\circ(P)$ in \cref{cor:derived-group} (2) is not canonical, it involves the choice of an isomorphism between the two fiber functors~$\omega$ and $\omega \circ [d]_*$ on the tensor category $\langle P \rangle$. But we can choose the isomorphism in a contravariant functorial way with respect to monomorphisms in the full tensor subcategory
\[
 \cC \cap [d]_*^{-1}(\cC) \;:=\; \{ Q \in \cC \mid [d]_* Q \in \cC \}
 \;\subset\; \cC
\]
by fixing an isomorphism between the fiber functors $\omega$ and $\omega\circ [d]_*$ on this category.
\end{remark}

\subsection{Larsen's alternative}

Let $X\subset A$ be a subvariety such that $\delta_X \in \cC$. We are interested in criteria under which the Tannaka group $G_{X, \omega}$ is big. Suppose that~$X\subset A$ is nondegenerate and $2\dim X < \dim A$, so that by \cref{Lem:SumOfNondegenerate} the sum morphism 
\[ \sigma \colon \quad X\times X \;\too\; W \;:=\; X+X \;\subset \; A \]
is generically finite onto its image, and this image is nondegenerate. Let $U\subset W$ be a smooth open dense subset over which $\sigma$ is a finite \'etale cover. By adjunction, we have an inclusion $\delta_U \subset \sigma_*(\delta_{X\times X})_{\vert U}$ as a direct summand. The decomposition theorem~\cite{BBDG} extends this to an inclusion $\delta_W \subset \delta_X * \delta_X = \sigma_*(\delta_{X\times X})$ as a direct summand in the derived category of constructible sheaf complexes. More precisely, there exists a unique semisimple perverse sheaf $\epsilon_W \in \Perv(A, \bbF)$ without negligible direct summands, and a unique negligible complex $\nu_X \in \Dbc(A, \bbF)$, such that
\[
\delta_X * \delta_X \;=\; 
\delta_W \oplus \epsilon_X \oplus \nu_X.
\]
With this notation, we obtain the following criterion for big Tannaka groups:

\begin{lemma} \label{lem:pushforward-via-larsen}
	For any nondegenerate subvariety $X\subset A$ with $2\dim X < \dim A$, the following are equivalent:
	\begin{enumerate} 
		\item $G_{X, \omega}$ is big in the sense of \cref{sec:BigTannakaIntro}.
		\item $\epsilon_X$ is either a simple perverse sheaf, or a direct sum of a simple perverse sheaf and a skyscraper sheaf of rank one.
	\end{enumerate} 	
\end{lemma} 	

\begin{proof} 
Recall that $W\subset A$ is a proper nondegenerate subvariety, so it cannot be the support of a negligible sheaf complex. On the other hand, $\Supp(\epsilon_X \oplus \nu_X)=W$ since the morphism $\sigma\colon X\times X \to W$ has generic degree two. It follows that $\Supp (\epsilon_X) = W$. In particular, the representation $V=\omega(\delta_X)\in \Rep_\bbF(G_{X, \omega})$ must have dimension $\dim V > 2$, since otherwise~$\epsilon_X$ would be the skyscraper sheaf corresponding to $\det(V)$ by \cref{prop:connected-component} (3).
	
	\medskip 
	
	By applying the fiber functor $\omega$, one sees that the condition (2) is equivalent to saying that in the decomposition of the tensor square $V\otimes V$ there are only two irreducible direct summands of dimension $>1$. Since $\dim(V)>2$, this is equivalent to~(1) by Larsen's alternative \cite[p.~113]{KatzLFM} for the subgroup $G_{X, \omega} \subset \GL(V)$. 
\end{proof} 

\subsection{Symmetric powers} If the Tannaka group is big, similar arguments allow to control the sum morphism from symmetric powers of the subvariety:

\begin{lemma} \label{cor:birational-via-larsen}
	Let $X\subset A$ be a nondegenerate subvariety and $r \ge 1$ an integer such that $r\dim X < \dim A$. If $G_{X, \omega}$ is big, then the sum morphism 
\[
 \tau_r \colon \quad \Sym^r X \;\longrightarrow\; A
\]
is birational onto its image $W_r = X+\cdots + X$.
\end{lemma}

\begin{proof} 
Consider the following commutative diagram, where $q_r$ denotes the quotient morphism: 
		\[
\begin{tikzpicture}[scale=1]
\def\sqrtthree{1.73205080757};

\node[label={[shift={(-.65,-.35)}]$Z_r = X^r$}] at ( -\sqrtthree/2, 1/2) (a) { };
\node[label={[shift={(1.35,-.35)}]$W_r = X + \cdots + X$}] at (\sqrtthree/2, 1/2) (c) {};
\node[label={[shift={(0,-.55)}]$Y_r = \Sym^r X$}] at (0, -1) (b) {};

\draw[->, shorten <= 2pt, shorten >= 2pt] (a) edge  node[midway, above] {$\scriptstyle \sigma_r$}  (c);
\draw[->, shorten <= 2pt, shorten >= 2pt] (a) edge node[midway, left] {$\scriptstyle q_r $ }  (b);
\draw[->, shorten <= 2pt, shorten >= 2pt] (b) edge node[midway, right] {$\scriptstyle \tau_r$} (c);
\end{tikzpicture} 
\]
Since $q_r\colon Z_r\to Y_r$ is a finite branched cover with group $\frS_r$, the decomposition theorem shows that as an $\frS_r$-equivariant perverse sheaf the direct image $q_{r*}(\delta_{Z_r})$ is a direct sum
	\[
	q_{r*}(\delta_{Z_r}) \;\simeq\; \;\;
	\bigoplus_{\sigma} \sigma \boxtimes P_\sigma 
	\]
where $\sigma$ runs through all irreducible representations of the symmetric group $\frS_r$ and where each $P_\sigma$ is a semisimple perverse sheaf on $Y_r$. In this isotypic decomposition the action of the group $\frS_r$ on $\sigma\boxtimes P_\sigma$ is given by the action on $\sigma$. Since the action of the symmetric group on tensor powers of sheaf complexes involves a Koszul sign, the perverse intersection complex on $Y_r=\Sym^r X$ is the isotypic piece for the trivial representation $\mathbf{1}$ or the sign representation $\mathrm{sgn}$ of $\frS$ depending on the parity of $\dim X$: We have
\[
 \delta_{Y_r} \;\simeq \; P_\epsilon
 \quad \text{for} \quad 
 \epsilon \;=\; 
 \begin{cases} 
 \;\mathrm{sgn} & \text{if $\dim X$ is odd}, \\
 \;\mathbf{1} & \text{if $\dim X$ is even}.
 \end{cases}
\]
as one may check on the open dense subset where~$q_r$ is finite \'etale. So the direct image $\delta_{X,r} := R\tau_{r*}(\delta_{Y_r})$ corresponds to the representation
\[
 \omega(\delta_{X,r}) \;\simeq\; 
 \begin{cases} 
 \Alt^r V & \text{if $\dim X$ is odd}, \\
 \Sym^r V  & \text{if $\dim X$ is even},
 \end{cases} 
\]
where $V:=\omega(\delta_X)$ is the defining representation of the group $G_{X,\omega}$. If that group is big, then we are in one of the following cases:\smallskip
\begin{enumerate} 
\item $G_X = \SL(V)$. Then $\Alt^r V$ and $\Sym^r V$ are irreducible representations by Schur-Weyl duality~\cite[th.~6.3, part~(4)]{FultonHarris}. \smallskip
\item $G_X = \SO(V)$. Then we have an embedding $\Sym^{r-2} V\hookrightarrow \Sym^r V$ and the quotient $\Sym^r V /\Sym^{r-2} V$ is irreducible~\cite[th.~19.19]{FultonHarris}.\smallskip
\item $G_X = \Sp(V)$. Then we have an embedding $\Alt^{r-2} V \hookrightarrow \Alt^r V$ and again the quotient $\Alt^r V /\Alt^{r-2} V$ is irreducible~\cite[th.~17.11]{FultonHarris}.\smallskip
\end{enumerate} 
In the first case $\delta_{X,r}$ is a simple perverse sheaf modulo negligibles, while in the other two cases we have an embedding $\delta_{X,r-2} \hookrightarrow \delta_{X,r}$ whose cokernel is a simple perverse sheaf (note that $\dim X$ is even in case (2) and odd in case (3)). In all three cases the semisimple perverse sheaf~$\delta_{X,r}$ has a unique simple direct summand $\epsilon_{X,r} \subset \delta_{X,r}$ with full support, i.e.~with 
\[
 \Supp(\epsilon_{X,r}) \;=\; W_r \;=\; X+\cdots + X.
\] 
But the decomposition theorem for the generically finite morphism $\tau_r\colon Y_r \to W_r$ also shows
\[
 \delta_{W_r} \;\subset\; \delta_{X,r} \;=\; R\tau_{r*}(\delta_{Y_r}),
\]
hence $\epsilon_{X,r}=\delta_{W_r}$. In particular, there exists an open dense subset $U\subset W_r$ such that
\[
 (R\tau_{r*}(\delta_{Y_r}))_{\vert U} \;\simeq\; (\delta_{Y_r})_{\vert U}
\] 
and by comparing the generic rank on that open subset we obtain $\det(\tau_r)=1$. 
\end{proof} 

In fact the above argument does not require the group $G_{X,\omega}$ to be big, we only need to have sufficient control on the support dimension of the perverse sheaves that enter the relevant wedge or symmetric power. For instance we have the following result which goes beyond the case of big Tannaka groups:

\begin{corollary} 
Let $X\subset A$ be nondegenerate with $r\dim X < \dim A$, and consider the representation
\[
 V \;:=\; 
 \begin{cases} 
 \Alt^r \omega(\delta_X) & \text{if $2\nmid \dim X$}, \\
 \Sym^r \omega(\delta_X) & \text{if $2\mid \dim X$}.
 \end{cases} 
\]
If $V\in \Rep_\bbF(G_{X, \omega}^*)$ has at most one irreducible direct summand of dimension $>1$, then the sum morphism $\tau_r\colon \Sym^r X \to X+\cdots +X$ is birational.
\end{corollary} 

\begin{proof} 
By \cite{WeissauerAlmostConnected} or \cite[section~3.c]{KraemerMicrolocalI} all one-dimensional representations of the Tannaka group arise from skyscraper sheaves, so for $\dim X > 0$ they cannot contribute to the support $W_r = X+\cdots + X$. Hence we can apply the same argument as in the previous proof. 
\end{proof}

\begin{corollary} 
\label{cor:larsen-alternative}
Let $X\subset A$ be a smooth irreducible curve generating $A$, and assume $\dim A \ge 3$. If the representation $V=\Alt^2(\omega(\delta_X))\in \Rep_\bbF(G_{X, \omega}^*)$ is a sum of an irreducible representation and a one-dimensional trivial representation, then 
\begin{enumerate} 
\item $X=p-X$ for some point $p\in X$,
\item $\tau\colon Y= \Sym^2 X \to W=X+X$ is finite birational over $U=W\setminus \{p\}$,
\item $G_{X, \omega}^* = \Sp(\omega(\delta_X), \theta)$ for the natural symplectic form $\theta$ on $\omega(\delta_X)$.
\end{enumerate} 	
\end{corollary} 

\begin{proof} 
By assumption $\Alt^2(\omega(\delta_X))$ contains a one-dimensional trivial representation, so the representation $\omega(\delta_X)$ is isomorphic to its dual. Therefore $X = p -X$ for some point $p\in X$. Now for dimension reasons $\tau\colon Y\to W$ restricts to a finite morphism over the complement $U=W\setminus \Sigma$ of a finite set $\Sigma \subset X$ of points. Note that $Y=\Sym^2 X$ is smooth for a smooth curve $X$, so we have $\delta_Y=\bbF_Y[2]$. Base change then shows that for any point $q$ we have
\[
 \cH^0(R\tau_*(\delta_Y))_q \;\simeq\; 
 H^2(\tau^{-1}(q), \bbF) \;
 \begin{cases} 
 \;=\; 0 & \text{if $q\notin \Sigma$}, \\
 \;\neq\; 0 & \text{if $q\in \Sigma$}.
 \end{cases} 
\]	
Since $R\tau_*(\delta_Y)$ is a direct sum of a semisimple perverse sheaf $P$ and a negligible sheaf complex and since negligible sheaf complexes cannot have cohomology sheaves which are skyscraper sheaves, it follows that $P$ contains the skyscraper sheaves $\delta_q$ in all points $q\in \Sigma$. But by assumption $R\tau_*(\delta_Y)$ contains a unique skyscraper summand, hence it follows that $\Sigma = \{p\}$ and thus $\tau$ is finite over $U=X\setminus \{p\}$.

\medskip 

In particular $R\tau_*(\delta_Y^-)_{\vert U}$ is a perverse sheaf, and we have $\cH^i(R\tau_*(\delta_Y^-))|_{\vert U}=0$ in all degrees $i\neq -2$ because $\delta_Y^-$ is a constructible sheaf placed in degree $-2$. But any semisimple perverse sheaf on a surface with cohomology sheaves only in degrees $-2$ is the minimal extension of a local system on any open dense subset of the surface. In our case that local system has rank one because  $\delta_Y^-$ has generic rank one and $\deg(\tau)=1$. Local systems of rank one are simple, hence it follows that the minimal extension $R\tau_*(\delta_Y^-)$ is a simple perverse sheaf.

\medskip 

In conclusion, this shows that $\delta_X*\delta_X = R\tau_*(\delta_Y)\oplus R\tau_*(\delta_Y^-)$ is a sum of two simple perverse sheaves and a skyscraper sheaf. It then follows by the same argument as in \cite[th.~6.1]{KWSmall} that $G_{X, \omega}^* = \Sp(\omega(\delta_X), \theta)$; note that $\dim(\omega(\delta_X))=\chi(\delta_X)\ge g > 2$ since the curve $X$ generates $A$. 
\end{proof}

\begin{corollary} \label{cor:curve-is-not-e7}
Let $X\subset A$ be a smooth irreducible curve generating $A$, and assume $\dim A \ge 3$. Then the group $G_{X, \omega}^\ast$ is not isomorphic to $E_7$ acting on~$\omega(\delta_X)$ via its irreducible representation of dimension $56$.
\end{corollary} 

\begin{proof} 
For the $56$-dimensional irreducible representation $W$ of the group $E_7$ the alternating square $\Alt^2(W)$ is a sum of an irreducible and a one-dimensional trivial representation. However, \cref{cor:larsen-alternative} says that $\Alt^2(\omega(\delta_X))$ can be a sum of an irreducible and a one-dimensional trivial representation only if $G_{X, \omega}^\ast \simeq \Sp_{56}(\bbF)$. 
\end{proof}

\subsection{Character twists}

Recall that $\Pi(A, \bbF)=\Hom(\pi_1(A, 0), \bbF^\times)$ denotes the group of continuous characters of the \'etale resp.~topological fundamental group of the abelian variety. For $\chi \in \Pi(A, \bbF)$, let $L_\chi$ be the local system of rank one with monodromy representation given by the character $\chi$.
For $P\in \Perv(A, \bbF)$ we call $P_\chi := P\otimes_{\bbF} L_\chi \in \Perv(A, \bbF)$ the {\em twist} of the given perverse sheaf by the character. Such twists of perverse sheaves appear in the generic vanishing theorem of~\cite{KWVanishing, SchnellHolonomic, BSS}: Let us say that a subset of $\Pi(A, \bbF)$ is a {\em proper subtorus} if it has the form 
\[
\Pi(A/B, \bbF) \;\subset\; \Pi(A, \bbF)
\]
where $B\subset A$ is a nonzero abelian subvariety. Then the generic vanishing theorem says that there is a finite union $\mathcal{S}(P) \subset \Pi(A, \bbF)$ of translates of proper subtori such that
\[
H^i(A, P_\chi) \;=\; 0 \quad \textnormal{for all $i\neq 0$ and all $\chi \in \Pi(A, \bbF)\smallsetminus \mathcal{S}(P)$}.
\]
We will use this in \cref{subsec:splitting} to write down explicit fiber functors with a natural Galois action. Up to noncanonical isomorphism, the Tannaka group of a perverse sheaf does not change under twists:

\begin{lemma} \label{lem:tannakagroup-of-twist}
	Let $P\in \cC$. Then for every character $\chi\in \Pi(A, \bbF)$ with $P_\chi\in \cC$ we have
	\[
	G_\omega(P_{\chi}) \;\simeq \; G_\omega(P).
	\] 
\end{lemma} 

\begin{proof} 
	By~\cite[prop.~4.1]{KWVanishing}, twisting by $\chi$ gives rise to an equivalence of tensor categories
	\[
	\langle P \rangle \;\stackrel{\sim}{\too} \; \langle P_{\chi} \rangle, \quad Q \;\longmapsto\; Q_{\chi}
	\]
	in $\Pbar(A, \bbF)$. This equivalence need not be compatible with the fiber functor $\omega$ on the source and target, but since $\bbF$ is algebraically closed, any two fiber functors are noncanonically isomorphic; hence the same holds for the Tannaka groups.
\end{proof}

\section{Galois theory for perverse sheaves}\label{section tanaka and monodromy}

In this section we discuss the behavior of Tannaka groups of perverse sheaves under extension of the base field and recall the connection between such Tannaka groups and classical monodromy groups in~\cite[section~5]{LS20}. We mostly follow the arguments in loc.~cit.~but remove the assumption of geometric semisimplicity in the Galois exact sequence by using a result of D'Addezio and Esnault~\cite{DAE20}. 

\subsection{Extension of the base field} 
Let $K/k$ be a field extension, and consider the base change functor
\[
(-)_K\colon \quad \Perv(A, \bbF) \;\too\; \Perv(A_K, \bbF), \quad P \;\longmapsto\; P_K.
\]
Passing to the abelian quotient categories by the subcategories of perverse sheaves of Euler characteristic zero, we have:

\begin{lemma} \label{lem:basefield-extension}
The base change functor descends to a faithful exact $\bbF$-linear tensor functor
\[
 (-)_K\colon \quad \Pbar(A, \bbF) \;\too\; \Pbar(A_K, \bbF).
\]	
\end{lemma} 

\begin{proof} 
The functor $(-)_K\colon \Perv(A, \bbF) \to \Perv(A_K, \bbF)$ is a faithful $\bbF$-linear exact functor. Let $q_K = q\circ (-)_K$ denote its composite with the quotient functor $q$ as shown below: 
	\[
\begin{tikzpicture}[scale=1]
\def\sqrtthree{1.73205080757};

\node[label={[shift={(-.75,-.4)}]$\Perv(A, \bbF)$}] at ( -\sqrtthree/2, 1/2) (a) { };
\node[label={[shift={(.85,-.4)}]$\Perv(A_K, \bbF)$}] at (\sqrtthree/2, 1/2) (c) {};
\node[label={[shift={(0,-.55)}]$\Pbar(A_K, \bbF)$}] at (0, -1) (b) {};

\draw[->, shorten <= 2pt, shorten >= 2pt] (a) edge  node[midway, above] {$\scriptstyle (-)_K$}  (c);
\draw[->, shorten <= 2pt, shorten >= 2pt] (a) edge node[midway, left] {$\scriptstyle q_K $ }  (b);
\draw[->, shorten <= 2pt, shorten >= 2pt] (c) edge node[midway, right] {$\scriptstyle q$} (b);
\end{tikzpicture} 
\]
Since $q_K$ is an exact functor between abelian categories which sends all objects of the Serre subcategory $\rmS(A, \bbF) \subset \Perv(A, \bbF)$ to zero, it factors by the universal property of abelian quotient categories~\cite[cor.~2, p.~368]{Gabriel62}  through a unique exact functor
\[
 (-)_K\colon \quad \Pbar(A, \bbF) \too \Pbar(A_K, \bbF).
\]
This functor is clearly $\bbF$-linear, and it admits the structure of a tensor functor with respect to the natural isomorphisms $(P*Q)_K\simeq P_K*Q_K$ inherited from the derived category. Any exact $\bbF$-linear tensor functor of rigid abelian tensor categories with $\End(\mathbf{1}) = \bbF$ is automatically faithful~\cite[prop.~1.19]{DM82}, so the claim follows. 	
\end{proof} 

Starting from a given full abelian tensor subcategory $\cC \subset \Pbar(A, \bbF)$, let us now denote by
\[
  \cC_K \;=\; \{ Q \mid  \textnormal{$\exists P\in \cC$ such that $Q$ is a subquotient of $P_K$} \} \;\subset\; \Pbar(A_K, \bbF)
\]
the full abelian tensor subcategory generated by the essential image of $\cC$ under the functor $(-)_K$ from \cref{lem:basefield-extension}. The category $\cC_K$ is again neutral Tannaka as it is a full abelian tensor subcategory of the neutral Tannaka category $\Pbar(A_K, \bbF)$. In what follows, we fix a fiber functor 
\[
 \omega\colon \quad \cC_K \;\too\; \Vect(\bbF).
\]
Precomposing with the base extension functor $(-)_K$ we get a fiber functor on $\cC$ and we denote by
\begin{eqnarray*} 
 G_{\omega}(\cC_K) &\ = & \Aut^\otimes(\omega \mid \cC_K), \\[0.5em]
 G_{\omega}(\cC) & := & \Aut^\otimes(\omega \mid \cC),
\end{eqnarray*}
the corresponding Tannaka groups. 

\begin{corollary} \label{cor:geometric-subgroup}
We have a closed immersion
$G_{\omega}(\cC_K) \into G_\omega(\cC)$.
\end{corollary} 

\begin{proof} 
By construction the faithful exact $\bbF$-linear tensor functor $(-)_K\colon \cC \to \cC_K$ is compatible with our chosen fiber functors, hence it defines a homomorphism of Tannaka groups. The latter is a closed immersion by~\cite[prop.~2.21(b)]{DM82}, since every object of $\cC_K$ is isomorphic to a subquotient of $P_K$ for some $P\in \cC$.
\end{proof} 

\subsection{The Galois sequence}  Let $k' \subset K$ be the algebraic closure of $k$ in $K$.   
The category 
\[
 \Rep_\bbF(\Aut(k'/k)) 
\] 
of continuous finite-dimensional representations of the profinite group $\Aut(k'/k)$ over $\bbF$ is a neutral Tannaka category.  If $k'/k$ is Galois, then $\Aut(k'/k) = \Gal(k'/k)$ is a quotient of  the absolute Galois group of $k$. In this case we can identify objects of the above category with sheaves on $\Spec(k)$ and hence the pushforward under the neutral element $e\colon \Spec(k) \to A$ gives a fully faithful embedding
\[
\begin{tikzcd}[column sep=15pt]   e_*\colon \hspace{-11pt}& \Rep_\bbF(\Gal(k'/k)) \ar[r,hookrightarrow] &  \Pbar(A, \bbF). \end{tikzcd}
\]
We will view Galois representations as a full subcategory of skyscraper sheaves and drop the $e_*$ from the notation. Our chosen fiber functor on $\cC$ restricts to a fiber functor 
\[
 \omega\colon \quad \cC\cap \Rep_\bbF(\Gal(k'/k)) \;\too\; \Vect(\bbF).
\] 
Let
\[ G_{\omega, \cC}(k'/k) \;:=\; \Aut^\otimes(\omega \,|\, \cC\cap \Rep_\bbF(\Gal(k'/k)))
\]
denote its Tannaka group. Representations of this group correspond to skyscraper sheaves $P\in \cC$ in the origin, and we have a homomorphism $\Aut(k'/k)\to G_{\omega, \cC}(k'/k)$. 

\begin{theorem} \label{thm:Galois_sequence} Assume as above that $k'/k$ is Galois. Then we have a short exact sequence of proalgebraic groups
\[
 1 \too G_{\omega}(\cC_K) \too G_\omega(\cC) \too G_{\omega, \cC}(k'/k) \too 1.
\]
\end{theorem}

\begin{proof}  
\Cref{cor:geometric-subgroup} gives a closed immersion $i\colon G_\omega(\cC_K) \to G_\omega(\cC)$. Moreover, since $k'/k$ is a Galois extension, we have by the above an embedding as a full tensor subcategory 
\[
\begin{tikzcd}[column sep=15pt]   \cC\cap \Rep_\bbF(\Gal(k'/k)) \ar[r,hookrightarrow] &  \cC. \end{tikzcd}
\]
which is stable under subobjects, and this embedding is compatible with the chosen fiber functors on the source and target. By ~\cite[prop.~2.21(a)]{DM82} we then have an epimorphism
\[
\begin{tikzcd}[column sep=15pt] p\colon \hspace{-11pt} & G_\omega(\cC) \ar[r,twoheadrightarrow] & G_{\omega,\cC}(k'/k). \end{tikzcd}
\]
By construction, $p\circ i$ is trivial.  Thus, to complete the proof,  by~\cite[prop.~A.13]{DAE20}, it suffices to check that\smallskip
\begin{enumerate}
\item the functor $(-)_K\colon \cC \to \cC_K$ is observable \cite[Appendix~A]{DAE20},  and \smallskip 
\item for every $P\in \cC$ the maximal trivial subobject of $P_K$ lies in the essential image of the functor $e_*\colon \cC\cap \Rep_\bbF(\Gal(k'/k)) \to \cC$.
\end{enumerate} 
For part (1) it suffices by lemma~A.4(1) in loc.~cit.~to show that, for $P\in \cC$, any rank one subobject
\[
 S \;\subset\; P_K 
\]
is a direct summand in a semisimple object $Q_K$ with $Q\in \cC$. To check this, note that   the rank one objects in the Tannaka category of perverse sheaves are rank one skyscraper sheaves, and that the  sum of all perverse rank one skyscraper subsheaves of $P_K$ is semisimple, being a sum of simple objects. To conclude the proof of (1), it suffices to show that  this direct sum  descends to a perverse subsheaf $Q\subset P$, as it then follows that   $ S$ is a direct summand of $Q_K$ as desired.   
 To prove that the sum of all rank one skyscraper subsheaves descends to $k$,  we  first show that the maximal skyscraper subsheaf of $P_K$ descends to a subsheaf of $P$.  Indeed, the Verdier dual of the sum of all perverse skyscraper subsheaves is the maximal perverse skyscraper quotient of the Verdier dual $D(P_K)$, which is $\mathcal{H}^0(D(P_K)) = \mathcal{H}^0(D(P))_K$.   Hence, the maximal skyscraper subsheaf descends. Replacing the given perverse sheaf $P$ by the maximal skyscraper subsheaf supported at the origin, we are reduced to the case $A=\Spec k$. Then $P$ is given by a Galois representation $V  \in  \Rep_{\bbF}(\Gal(\bar{k}/k))$ and the claim reduces to the following to facts:
  \begin{itemize} \item  A subspace of $V$ is stable under $\Gal(\bar{K}/K)$ if and only if it is so under $\Gal(\bar{k}/k')$ (since $\Gal(\bar{K}/K)\to \Gal(\bar{k}/k')$ is surjective for $k'$ algebraically closed in $K$).
  \item  The sum of all one-dimensional subrepresentations of $V_ {\vert \Gal(\bar{k}/k')}$ is stable under $\Gal(\bar{k}/k)$ (since $\Gal(\bar{k}/k')$ is a normal subgroup of $\Gal(\bar{k}/k)$).
  \end{itemize}
For (2) we argue similarly: The unit object  of the tensor category $\cC_K$ is the skyscraper sheaf $\delta_0$ of rank one supported in the origin. So the maximal trivial subobject of $P_K$ is the maximal  subobject of the form $\delta_0^{\oplus n}$ for some  integer $n\ge 0$, and this subobject descends to a subobject $Q\subset P$ as before. 
\end{proof}

\begin{corollary} \label{cor:algebraically-closed}
If $k$ is algebraically closed, then for every extension $K/k$ we have a natural isomorphism
\[
 G_\omega(\cC_K) \;\stackrel{\sim}{\too}\; G_\omega(\cC).
\]
In particular,  for every perverse sheaf $P\in \cC$, we have $G_\omega(P_K)\simeq G_\omega(P)$.
\end{corollary} 

\begin{proof} 
If $k$ is algebraically closed, then $k'=k$ and hence $G_\omega(k'/k) \simeq \{1\}$.	
\end{proof}

\subsection{A splitting of the sequence}
\label{subsec:splitting}
We now apply the above when $K=\bar{k}$ is an algebraic closure of $k$. In the Galois sequence in \cref{thm:Galois_sequence} we have used the fully faithful functor
\[
\begin{tikzcd}[column sep=15pt]  
 e_*\colon \hspace{-11pt}& \Rep_\bbF(\Gal(\bar{k}/k)) \ar[r,hookrightarrow] &  \Pbar(A, \bbF). \end{tikzcd}
\]
that identifies a Galois representation with the corresponding skyscraper sheaf at the origin. We now describe a splitting of the sequence in \cref{thm:Galois_sequence} for a special category~$\cC$ such that the functor
$
 e_*\colon \cC\cap \Rep_\bbF(\Gal(\bar{k}/k)) \into \cC
$ 
has a left inverse. To do so, let
\[
 \Perv_0(A, \bbF)
\]
be the full subcategory of all $P\in \Perv(A, \bbF)$ for which all simple subquotients $Q$ of $P_{\bar{k}}$ satisfy 
\[
 H^i(A_{\bar{k}}, Q) \;=\; 0 
 \quad \textnormal{for all} \quad i\;\neq\; 0.
\]
Its image
\[
\Pbar_0(A, \bbF) \;\subset\; \Pbar(A, \bbF)
\]
is a full abelian tensor subcategory which is equivalent to $\Perv_0(A, \bbF)/S_0(A, \bbF)$, where $S_0(A, \bbF):=S(A, \bbF)\cap \Perv_0(A, \bbF)$ is the full subcategory of perverse sheaves~$P$ with the property that all the subquotients $Q$ of $P_{\bar{k}}$ satisfy $H^\bullet(A_{\bar{k}}, Q)=0$. We then get a functor
\[
 \omega\colon \quad \Pbar_0(A, \bbF) \;=\; \Perv_0(A, \bbF)/S_0(A, \bbF)\;\too\; \Vect(\bbF), \quad Q \;\longmapsto\; H^0(A_{\bar{k}}, Q)
\]
which is exact by definition of the source category. Moreover, $\omega$ is a tensor functor by the K\"unneth isomorphism
\[
 H^\bullet(A_{\bar{k}}, P*Q) \;\simeq\; H^\bullet(A_{\bar{k}}, P) \otimes H^\bullet(A_{\bar{k}}, Q),
\]
since for $P, Q\in \Perv_0(A, \bbF)$ only the cohomology in degree zero contributes. For the fiber functor obtained in this way, we can summarize the relation between the Tannaka groups over $k$ and over $\bar{k}$ as follows:

\begin{theorem} \label{lem:galois-in-normalizer}
For $\cC = \Pbar_0(A, \bbF)$ with the fiber functor $\omega := H^0(A_{\bar{k}}, -)$, the above construction induces a splitting of the short exact sequence
\[
  1 \too G_{\omega}(\cC_{\bar{k}}) \too G_\omega(\cC) \too G_{\omega, \cC}(\bar{k}/k) \too 1
\] 
In particular, we have an isomorphism
\[
 G_\omega(\cC) \;\simeq\; G_\omega(\cC_{\bar{k}}) \rtimes G_{\omega, \cC}(\bar{k}/k),
\]
and for any $P\in \Pbar_0(A, \bbF)$, the action of $\Gal(\bar{k}/k)$ on $V=\omega(P)$ factors through the normalizer
\[
 N(G_\omega(P_{\bar{k}})) \;\subset\; \GL(V).
\]
\end{theorem}

\begin{proof} 
 While the fiber functor $\omega = H^0(A_{\bar{k}}, -)$ is only defined on $\cC := \Pbar_0(A, \bbF)$, it comes with a natural Galois action in the sense that we have a commutative diagram 
	\[ \hspace{50pt}
\begin{tikzpicture}[scale=1]
\def\sqrtthree{1.73205080757};

\node[label={[shift={(-0.1,-.35)}]$\cC$}] at ( -\sqrtthree/2, 1/2) (a) { };
\node[label={[shift={(1.15,-.43)}]$\Rep_\bbF(\Gal(\bar{k}/k))$}] at (\sqrtthree/2, 1/2) (c) {};
\node[label={[shift={(0,-.55)}]$\Vect(\bbF)$}] at (0, -1) (b) {};

\draw[->, shorten <= 2pt, shorten >= 2pt] (a) edge  node[midway, above] {$\scriptstyle \exists$}  (c);
\draw[->, shorten <= 2pt, shorten >= 2pt] (a) edge node[midway, left] {$\scriptstyle \omega $ }  (b);
\draw[->, shorten <= 2pt, shorten >= 2pt] (c) -- (b);
\end{tikzpicture}
\]
where the top row is a left inverse of the functor $e_*\colon \Rep_\bbF(\Gal(\bar{k}/k)) \to \cC$.	
\end{proof}

\subsection{Big monodromy from big Tannaka groups}  
\label{subsec:monodromy}

Now again assume that $k$ is an algebraically closed field of characteristic zero. Consider the constant abelian scheme $A_S := A\times_k S$, where $S$ is an integral scheme over $k$. We denote by $\bar \eta$ a geometric point over the generic point $\eta$ of $S$.  Let $\cX\subset A_S$ be an irreducible closed subscheme which is smooth over $S$. We want to control the monodromy of the family $ \cX\to S$ twisted by a generic rank one local system as in \cite{LS20}. In this context, the following terminology will be useful.

\begin{definition}
We say that $\cX\subset A_S$ is \emph{constant up to translation in $A(S)$} if there is a subvariety $Y\subset A$ and a point $a\in A(S)$ such that  $\cX = Y_S + a$.
\end{definition} 

In favorable situations, this condition can be read off from the geometric generic fiber of $ \cX\to S$ via the following descent result:

\begin{lemma} \label{lem:constant-and-symmetric}
Suppose $S$ is a smooth and irreducible variety. Let $\cY, \cZ \subset A_S$ be subvarieties which are flat over $S$. If the subvariety $\cY_{\bar \eta}\subset A_{S, \bar \eta}$ has trivial stabilizer, then the following are equivalent:
\begin{enumerate} 
\item $\cZ = \cY + a$ for some $a \in A(S)$.\smallskip
\item $\cZ_{\bar{\eta}} = \cY_{\bar{\eta}} + a$ for some $a \in A(\bar{\eta})$.
\end{enumerate} 	
\end{lemma} 

\begin{proof} 
Clearly, the first property implies the second. Conversely, suppose that we have $\cZ_{\bar \eta} = \cY_{\bar \eta} + a$ for some point $a\in A(\bar{\eta})$. First, we claim that the point $a$ comes from a point $a\in A(\eta)$. Indeed, let $F$ be the function field of $S$ and let~$y = [\cY_{\eta}]$ and~$z = [\cZ_{\eta}]$ be the $F$-points of the Hilbert scheme $\Hilb(A)$ defined by the generic fibers of $ \cY \to S$ and $\cZ \to S$, seen as subvarieties of $A_{S, \eta}$. Now, the abelian variety~$A$ acts on the Hilbert scheme by translation. The transporter
\[ T = \{ t \in A_{S, \eta} \mid z = y + t\} \]
is a subvariety of $A_{S, \eta}$. Note that $T(\bar \eta)$ is nonempty, as it contains the point $a$. Actually, the point $a$ is the only one of $T(\bar{\eta})$. For, note that the stabilizer of the subvariety $\cY_{\bar{\eta}} \subset A_{S, \bar{\eta}}$ acts freely and transitively on the base-change of $T$ to $\bar{\eta}$. On the other hand, the stabilizer of $\cY_{\bar{\eta}}$ is trivial by assumption, so the transporter~$T(\bar{\eta})$ must be a singleton. The variety $T$ is defined over $F$ and has only one point over an algebraically closed field, thus $T = \Spec F$ which proves the claim.

\medskip 

The point $a \in A(\eta)$ can be seen as a rational map $a \colon S \dashto A$, which is moreover everywhere defined by smoothness of $S$ \cite[th.~3.1]{MilneAV}. To conclude the proof,  note that the generic fibers of  $\cY + a$ and $\cZ$   coincide, hence $\cZ = \cY + a$ by flatness.
\end{proof} 

\begin{corollary} \label{cor:constant_up_to}  If $S$ is a smooth irreducible variety and if the subvariety $\cX_{\bar \eta}\subset A_{S, \bar \eta}$ is nondivisible, then the following are equivalent:
\begin{enumerate} 
\item $\cX\subset A_S$ is constant (resp.~symmetric) up to translation in $A(S)$.\smallskip
\item $\cX_{\bar \eta} \subset A_{S, \bar \eta}$ is constant (resp.~symmetric) up to translation in $A(\bar{\eta})$. 
\end{enumerate} 
Moreover, the subvariety $\cX\subset A_S$ is constant up to translation in $A(S)$ if and only if the family $\cX \to S$ is isotrivial.
\end{corollary} 

\begin{proof}
The equivalence of (1) and (2) follows directly from \cref{lem:constant-and-symmetric}. Now suppose that the family $\cX \to S$ is isotrivial. In order to prove that the subvariety $\cX \subset A_S$ is constant up to translation, we may by the equivalence of (1) and (2) replace $S$ by an \'etale cover and hence assume $\cX \simeq Y_S$ for some $Y\subset A$.  Fixing $y\in Y(k)$, we get a section $x\colon S\to \cX$ that gives rise to a commutative diagram:
\[
\begin{tikzcd}[column sep=40pt]
\cX \ar[r, "\alb_x"] \ar[d, "\wr"] & \Alb(\cX/S) \ar[r] \ar[d, "\wr"] & A_S \ar[d, equals] \ar[r, "z\mapsto z+x"] & A_S \ar[d] \\
Y_S \ar[r, "\alb_y"] & \Alb(Y_S/S) \ar[r] & A_S \ar[r, "z\mapsto z+y"] & A_S
\end{tikzcd} 
\]
Here $\alb_a$ and $\alb_y$ are the relative Albanese morphisms and the composite of the horizontal arrows are the inclusions $\cX \subset A_S$ resp.~$Y_S\subset A_S$. Hence, $\cX\subset A_S$ is constant up to translation.
\end{proof} 

\begin{example} 
The nondivisibility is needed in the above: Let $Y \subset A$ be a subvariety with finite stabilizer $\Stab(Y)\neq \{0\}$. Viewing $S:= A/\Stab(Y)$ as the orbit of the point $[Y]$ in $\Hilb(A)$ under the translation action of $A$, we get by restriction of the universal subvariety of $A\times_k \Hilb(A)$ a subvariety $\cX\subset A_S$ with fiber $Y + a$ over a point $[a] \in S(k)$. Then the family $ \cX\to S$ is not constant up to translation in~$A(S)$, but it is so up to translation by a section in $A(\bar \eta)$.
\end{example}

\medskip 

We now assume that $\cX\subset A_S$ is not constant up to translation in~$A(S)$. Then the monodromy of the smooth family $ \cX\to S$ twisted by a generic rank one local system is related to the Tannaka group of the perverse sheaf $\delta_{X}\in \Perv(A_{S,\bar \eta}, \bbF)$ on the geometric generic fiber
 \[X \;:=\; \cX_{\bar{\eta}}\] 
as follows. For $\chi \in \Pi(A, \bbF)$, let $L_\chi$ denote the corresponding rank one local system on $A$. The generic vanishing theorem for perverse sheaves~\cite{BSS, KWVanishing, SchnellHolonomic} shows that 
\[ \delta_{X, \chi} \;:=\; \delta_X \otimes L_\chi \;\in\; \Pbar_0(A_{S,\bar \eta}, \bbF) \]
for most $\chi \in \Pi(A, \bbF)$, where {\em most} means all characters $\chi$ outside a finite union of torsion translates of linear subvarieties of $\Pi(A, \bbF)$. From \cref{subsec:splitting} we get a fiber functor 
\[
 \omega \;:=\; H^0(A_{S,\bar \eta}, -)\colon 
 \quad 
 \langle \delta_{X, \chi} \rangle 
 \;\longrightarrow\; \Vect(\bbF),
\]
and we denote by
\[
 G^*_{X, \chi} \;:=\; [G_\omega^\circ(\delta_{X, \chi}), G_\omega^\circ(\delta_{X, \chi})]
\]
the derived group of the connected component of the Tannaka group. Note that by \cref{lem:tannakagroup-of-twist} the isomorphism type of this group does not depend on the chosen character; we say that $X$ has a {\em simple derived connected Tannaka group} if $G_{X, \chi}^\ast$ is simple for some (hence every) character $\chi$ with the above vanishing properties.

\medskip 

To define the monodromy of the family $f\colon \cX \to S$ twisted by a rank one local system, let $\pi\colon \cX \to A$ be the projection to the abelian variety. Using generic vanishing on the geometric generic fiber of $X\subset A_{S, \bar \eta}$, one sees that for most $\chi$ the higher direct images $R^i f_* \pi^* L_\chi$ vanish in all degrees $i\neq d$, where $d$ denotes the relative dimension of the family $f\colon \cX \to S$. For such $\chi$ the remaining direct image
\[
V_{\chi} \;:=\; R^d f_* \pi^* L_{\chi}
\]
is a local system. More generally we consider for $\underline{\chi} = (\chi_1, \dots, \chi_n) \in \Pi(A, \bbF)^n$ the direct sum
\[
 V_{\underline \chi} \;:=\; V_{\chi_1} \oplus \cdots \oplus V_{\chi_n}.
\]
Let
$
 \rho\colon \pi_1(S, \bar \eta) \rightarrow \GL(V_{\underline{\chi}, \bar \eta})
$ 
be the corresponding monodromy representation on the geometric generic fiber. We define the \emph{algebraic monodromy group} of $V_{\underline{\chi}}$ as the Zariski closure 
\[
 M(V_{\underline{\chi}}) \;:=\; \overline{\im(\rho)} \;\subset\; 
 \GL(V_{\underline{\chi}, \bar \eta}).
\]
The link between our main theorem from the introduction and the Tannaka groups introduced above is the following result by Lawrence and Sawin, an analog of the theorem of the fixed part:

\begin{theorem}\label{Thm:FixedPart}  
Let $S$ be a  smooth  integral variety over $k$, and let $ \cX\subset A_S$ an integral subvariety such that
\begin{enumerate}
\item the family $f \colon \cX\to S$ is smooth of relative dimension~$d$, it is not constant up to translation in $A(S)$, and 
\item the geometric generic fiber $X=\cX_{\bar \eta} \subset A_{S,\bar \eta}$ is nondivisible and has a simple derived connected Tannaka group.
\end{enumerate}
 Then for most $\underline{\chi} \in \Pi(A, \bbF)^n$ we have  
 \[ G^\ast_{X, \chi_1} \times \cdots \times G^\ast_{X, \chi_n} \;\unlhd\; M(V_{\underline \chi}).
 \]
\end{theorem}

\begin{proof} In \cite[th.~5.6]{LS20} this is stated for hypersurfaces, but the proof works for smooth subvarieties of any codimension. For convenience, we recall the main ideas in our setup: The fiber
\[
V_{\underline{\chi}, \bar \eta}
\;=\; 
\bigoplus_{i=1}^n \rH^d(X, L_{\chi_i})
\]	
comes with a monodromy action of $\pi_1(S, {\bar \eta})$ preserving the summands on the right-hand side; the algebraic monodromy is the Zariski closure of the image of $\pi_1(S, {\bar \eta})$ inside 
\[ 
\GL(V_{\chi_1, \bar{\eta}}) \times \cdots \times \GL(V_{\chi_n, \bar{\eta}})
\quad \textnormal{where} \quad 
V_{\chi_i, \bar{\eta}} \;=\; \rH^d(X, L_{\chi_i}).
 \]
Since $S$ is smooth, this algebraic monodromy is the Zariski closure of the image of the absolute Galois group of the function field of $S$. By \cref{lem:galois-in-normalizer} the Galois action normalizes the subgroups $G_{X, \chi_i}^\ast \subset \GL(V_{\chi_i})$, in fact the algebraic monodromy is a subgroup
\[
 M(V_{\underline{\chi}}) \;\subset\; G_{X_0, \chi_1} \times \cdots \times G_{X_0, \chi_n}
\]
where $X_0 := \cX_\eta$ denotes the generic fiber and $G_{X_0, \chi_i} := G_\omega(\delta_{X_0, \chi_i})$.
We must show that this upper bound on the algebraic monodromy is almost sharp in the sense that for most $\underline{\chi}=(\chi_1, \dots, \chi_n)$, the algebraic monodromy contains the normal subgroup
\[ G^\ast_{X, \chi_1} \times \cdots \times G^\ast_{X, \chi_n} \;\unlhd \; G_{X_0, \chi_1} \times \cdots \times G_{X_0, \chi_n}.
\]
In what follows, it will be convenient to identify all factors on the left-hand side with a fixed simple algebraic group. For this, we fix a fiber functor~$\xi\colon \langle \delta_X \rangle \to \Vect(\bbF)$ and pick an isomorphism between $\rH^0(A_{\bar \eta}, -)$ and the fiber functor obtained as the composite
\[
\begin{tikzcd}
 \langle \delta_{X_0, \chi_i} \rangle 
 \ar[r, "\sim"]
 & 
 \langle \delta_{X_0} \rangle 
 \ar[r, "(-)_{\bar \eta}"] 
 & 
 \langle \delta_X \rangle 
 \ar[r, "\xi"]
 & \Vect(\bbF),
\end{tikzcd} 
\]
where the isomorphism on the left is the inverse of $P\mapsto P_{\chi_i}$. We get a commutative diagram
\[
\begin{tikzcd} 
& G_{X, \chi_1}^* \times \cdots \times G_{X, \chi_n}^* \ar[r, "\sim"] \ar[d, hook] & 
 (G^*_X)^n \ar[d, hook] 
 \\
M(V_{\underline{\chi}}) \ar[r, hook] & G_{X_0, \chi_1} \times \cdots \times G_{X_0, \chi_n} \ar[r, "\sim"]  &
 (G_{X_0})^n 
\end{tikzcd} 
\]
where $G_X^* := [G^\circ_\xi(\delta_X), G^\circ_\xi(\delta_X)]\subset G_{X_0} := G_\xi(\delta_{X_0})$. Note that $G_{X_0}$ is contained in the normalizer
\[
 N(G_X^*) \subset \GL(\xi(\delta_{X_0}))
\]
by \cref{lem:galois-in-normalizer}. Now we use the following general observation \cite[lemma 5.4]{LS20}:  

\begin{fact} \label{representation list}
Let $G\subset \GL(V)$ be a simple algebraic group, and $N(G)\subset \GL(V)$ its normalizer. Then for every integer $n\ge 1$ there exists a finite list of irreducible representations 
\[
 W_\alpha \;=\; W_{\alpha,1} \boxtimes \cdots \boxtimes W_{\alpha, n} \;\in\; \Rep_\bbF(N(G)^n)
 \quad (\alpha \in \{1, \dots, N\})
\]
such that for any reductive subgroup $H\subset N(G)^n$ the following two properties are equivalent:
\begin{enumerate} 
\item $G^n \subset H$. 
\item $H$ has no invariants on any of the representations $W_\alpha$. 
\end{enumerate} 
In particular, the group $G^n$ has no invariants on any of the representations $W_\alpha$.
\end{fact}

\medskip 

We apply this to $V=\xi(\delta_{X_0})$, $G=G_X^*$ and $H=M(V_{\underline{\chi}})$. Since $G_{X_0} \subset N(G_X^*)$, each $W_{\alpha, i} \in \Rep_\bbF(N(G))$ defines a representation of the Tannaka group $G_{X_0}$ and hence a perverse sheaf
\[
 P_{\alpha,i} \;\in\; \langle \delta_{X_0} \rangle. 
\]
The representation obtained by pullback under the isomorphism $G_{X_0, \chi_i} \to G_{X_0}$ then corresponds to the perverse sheaf
$(P_{\alpha,i})_{\chi_i} \in \langle \delta_{X_0, \chi_i} \rangle$.
By construction, we have an isomorphism
\begin{eqnarray} \label{eq:representation_from_list} \nonumber
 W_\alpha 
 &\;=\;&
 W_{\alpha,1} \boxtimes \cdots \boxtimes W_{\alpha, n} \\
 &\;\simeq\;&
  \rH^0(A_{S, \bar \eta}, (P_{\alpha,1})_{\chi_1}) \boxtimes \cdots \boxtimes \rH^0(A_{S,\bar \eta},(P_{\alpha,n})_{\chi_n} )
\end{eqnarray}
of representations of $N(G_{X, \chi_1}^*)\times \cdots \times N(G_{X, \chi_n}^*)$. To keep track of how the Galois action on the right-hand side depends on the chosen characters, it will be convenient to pass to $A_{S, \bar \eta}^n = A_{S, \bar \eta} \times \cdots \times A_{S, \bar \eta}$ via the K\"unneth isomorphism. Consider the perverse sheaf
\[
 K_0 \;:=\;  e_{1*} \delta_{X_0} \oplus \cdots \oplus e_{n*} \delta_{X_0}
 \;\in\; \Pbar_0(A_{S,  \eta}^n, \bbF)
\]
where $e_i\colon A_{S,\eta} \hookrightarrow A_{S,  \eta}^n$ denotes the $i$-th coordinate inclusion. Let $K\in \Pbar_0(A_{S, \bar \eta}^n, \bbF)$ be the base change of the perverse sheaf $K_0$ to the geometric generic fiber. Note that $ G^*_{\zeta}(K) = (G_{X}^*)^n $ for the fiber functor $\zeta := \xi \boxtimes \cdots \boxtimes \xi\colon \langle K \rangle \to \Vect(\bbF)$ and that we have
\[
Q_{\alpha} 
\;:=\; 
e_{1*} P_{\alpha,1} * \cdots * e_{n*} P_{\alpha, n}
\;=\;
P_{\alpha,1} \boxtimes \cdots \boxtimes P_{\alpha, n}
\;\in\; 
\langle K_0 \rangle.
\]
Returning to character twists again, consider the local system $L_{\underline{\chi}} := L_{\chi_1} \boxtimes \cdots \boxtimes L_{\chi_n}$ and put 
\[
 K_{0, \underline{\chi}} \;:=\; 
 K_0 \otimes L_{\underline{\chi}},
 \quad
 K_{\underline{\chi}} 
 \;:=\; K \otimes L_{\underline{\chi}}
 \quad \textnormal{and} \quad 
 Q_{\alpha, \underline{\chi}} 
 \;:=\; Q_\alpha \otimes L_{\underline{\chi}}.
\]
Then we have
\[
 G_\omega^*(K_{\underline{\chi}}) \;=\; G_{X, \chi_1}^*\times \cdots \times G_{X, \chi_n}^*
 \quad \textnormal{and} \quad 
 Q_{\alpha, \underline{\chi}} \;\in\; \langle K_{0, \underline{\chi}} \rangle
\]
Combining \eqref{eq:representation_from_list} with the K\"unneth isomorphism we obtain a Galois equivariant isomorphism
\[
 W_\alpha \;\simeq\;  \rH^0(A^n_{S, \bar \eta}, Q_{\alpha, \underline \chi}),
\]
where the Galois group acts on the left-hand side via $\Gal(\bar{\eta}/\eta)\to M(V_{\underline{\chi}})$ and on the right-hand side by the natural Galois action.

\medskip 

Now recall that by the last claim in \cref{representation list}, the group $(G_X^*)^n$ has no invariants on~$W_\alpha$. But $(G_X^*)^n = G_\zeta^*(K)$ is the derived group of the connected component of the Tannaka group of the perverse sheaf $K$, and $W_\alpha = \zeta(Q_\alpha)$ is the representation defined by $Q_\alpha \in \langle K_0\rangle$. Hence, we can apply~\cite[lemma 5.2]{LS20}: The vanishing of invariants of the derived connected Tannaka group on the geometric generic fiber implies that $Q_\alpha$ has no perverse subquotient coming by pullback from $A$ via~$A_{S,\eta} \to A$. By a spreading out argument~\cite[lemma~5.3]{LS20} the last property implies that for most $\underline{\chi}$ the Galois invariants of $\rH^0(A_{S,\bar \eta}, Q_{\alpha, \underline{\chi}})$ vanish. Thus, the algebraic monodromy has no invariants on any of the representations $W_\alpha$ and hence by the equivalence of (1) and (2) above it contains all of $(G_X^*)^n$ as required.
\end{proof}

\begin{remark}
For $n\ge 2$, the above proof gives more precise information on the dependence of $n$ of the locus of characters on which the conclusion of \cref{Thm:FixedPart} holds: There exists a finite union
$
\Sigma \subset \Pi(A, \bbF)^2
$ 
of torsion translates of proper linear subvarieties such that the conclusion of the theorem holds for all $n\ge 2$ and all $\underline{\chi}=(\chi_1, \dots, \chi_n)\in \Pi(A, \bbF)^n$ with
\[ 
 (\chi_i, \chi_j)\notin \Sigma \quad \textnormal{for all} \quad i\neq j.	
\]
This follows from the fact that the list of representations constructed in the proof of \cref{representation list} arises from a finite list of representations of $N(G)^2$ by pullback under the various projections $N(G)^n \to N(G)^2$. 
\end{remark} 
  
\section{From representations to geometry}\label{section from rep to geo}

In this section, we explain the link between representations and characteristic cycles, which will be our main tool to show that under certain assumptions the Tannaka group of a smooth subvariety will be big. We work over an algebraically closed field $k$ with $\characteristic(k)=0$, and starting from \cref{subsec:clean-cc} we assume $k=\bbC$.

\subsection{The ring of clean cycles} \label{sec:RingCleanCycles} Over the complex numbers an important invariant of a perverse sheaf is its characteristic cycle, which is a formal sum of conormal varieties adapted to a suitable Whitney stratification. As we recall in \cref{subsec:clean-cc}, the convolution product of perverse sheaves is mirrored by a `convolution product' on their characteristic cycles. To define the latter, we need to introduce a convolution product of conormal varieties, which can be done over any algebraically closed field~$k$ of characteristic zero as follows.  

\medskip

Recall that, for an integral subvariety $Z \subset A$, its conormal variety $\PLambda_Z$ is said to be \emph{clean} if its Gauss map $\gamma_Z \colon  \PLambda_Z \to \bbP_A$ is dominant---by \cref{Thm:PositivityNotions} this is the case if and only if the variety $Z$ is of general type---and \emph{negligible} otherwise. 

\begin{definition}
The group of clean cycles $\cL(A)$ is the free abelian group generated by the projective conormal cones $\PLambda_Z$ of integral subvarieties $Z \subset A$, modulo the subgroup generated by the negligible ones. The projection onto the quotient induces an isomorphism
\[
   \bigoplus_{Z\subset A} \bbZ\cdot \PLambda_Z \; \stackrel{\sim}{\too} \;  \cL(A),
\]
where the direct sum ranges over the integral subvarieties of general type $Z \subset A$. 

\medskip 

A \emph{clean cycle} is an element of $\cL(A)$ and, by means of the preceding isomorphism, will always be seen as a finite formal sum $\sum_Z m_Z \PLambda_Z$, $m_Z \in \bbZ$, indexed by the integral subvarieties $Z \subset A$ of general type.
\end{definition}

Recall that in \cref{Def:ConormalCycle} we defined the conormal variety $\PLambda_Z$ for a reduced but not necessarily irreducible subvariety $Z \subset A$. For simplicity, we still write $\PLambda_Z$ for the conormal variety seen as a cycle on $A \times \bbP_A$, or merely as a clean cycle. In particular, in the latter case, we have
\[ \PLambda_Z = \sum_{Z' \subset Z} \PLambda_{Z'},\]
the sum ranging over the irreducible components $Z' \subset Z$ of general type. We will consistently perpetrate this abuse of notation by writing
\[ \PLambda_{Z} = m_1 \PLambda_{Z_1} + \cdots + m_n \PLambda_{Z_n} \]
for a cycle $Z = m_1 Z_1 + \cdots + m_n Z_n$ on $A$, with $m_i \in \bbZ$ and $Z_i \subset A$ integral.

\begin{definition} \label{defn:convolution-of-clean-cycles} 
Let $\PLambda_{X_1}, \PLambda_{X_2}$ be clean conormal varieties. Let $U\subset \bbP_A$ be an open dense subset of the projective cotangent space to the abelian variety such that over this open subset the Gauss maps
$
 \PLambda_{X_i\mid U} := \gamma_{X_i}^{-1}(U) \to U
$
are finite \'etale covers. The fiber product of these two finite \'etale covers embeds into $A\times A\times U \subset A\times A \times \bbP_A$, and we denote by
\[
 \PLambda \;:=\; \overline{\PLambda_{X_1 \rvert U} \times_U \PLambda_{X_2 \rvert U}}
 \;\subset\; A\times A \times \bbP_A 
\]
its Zariski closure. We define the {\em convolution} of the conormal varieties to be the clean cycle
\[
 \PLambda_{X_1} \circ \PLambda_{X_2} \;:=\; 
 \sigma_*(\PLambda ) \;\in\; \cL(A)
\]
arising by pushforward under the sum morphism $\sigma \colon A\times A \times \bbP_A \to A\times \bbP_A$. We extend this product $\circ$ on conormal varieties bilinearly to a product on the group of clean cycles
\[
\circ\colon \quad \cL(A) \times \cL(A) \;\too\; \cL(A).
\]
This endows the group $\cL(A)$ with a natural ring structure. The product $\circ$ should not be confused with an intersection of cycles, indeed the intersection product of any two cycles in $\cL(A)$ is zero for dimension reasons. For any integer $n\ne 0$ the pushforward 
\[
 [n]_*\colon \quad \cL(A) \;\too\; \cL(A)
\]
is a ring homomorphism. For $\PLambda \in \cL$ we denote by $\langle \PLambda \rangle \subset \cL(A)$ the smallest subring of $\cL(A)$ which contains $\PLambda$ and is stable under passing from a clean cycle to its irreducible components.
\end{definition}

\subsection{A reminder on Segre classes} 

In the discussion of wedge powers and spin representations to be carried out in \crefrange{sec:wedge}{sec:spin}, we will need to control the effect that certain tensor constructions on clean cycles have on the dimension of their base. For this we recall in this section some basic facts about Segre classes, or Chern-Mather classes\footnote{In the case of abelian varieties Segre classes and Chern-Mather classes are the same, since the cotangent bundle to abelian varieties is trivial.} in the terminology of~\cite[section~3]{KraemerMicrolocalII}:

\begin{definition} 
	The {\em Segre classes} of a cycle $\PLambda$ on $A\times \bbP_A$ of pure dimension $g-1$ are defined as the cycle classes
	\[
	s_d(\PLambda) \;:=\; (\pr_A)_*([\PLambda] \cdot [A\times H_d] )
	\;\in\; 
	\CH_d(A)
	\]
	where $H_d \subset \bbP_A$ is a general linear subspace of dimension $d< g=\dim A$ and $\CH_d(A)$ denotes the Chow group of dimension $d$ algebraic cycles with $\bbZ$-coefficients, modulo rational equivalence. 
\end{definition} 

The following observation allows to control the dimension of the base of a clean cycle in terms of its Segre classes:

\begin{remark}
For any subvariety $Z\subset A$ we have $s_d(\PLambda_Z)=0$ for all $d>\dim Z$. On the other hand, if $Z$ has a top-dimensional irreducible component of general type, then the Segre classes $s_d(Z)$ are represented by nonzero effective cycles for all $d\in \{0,1,\dots,\dim Z\}$; this follows from the dominance of the Gauss map $\gamma_Z$ and Kleiman's generic transversality theorem~\cite[lemma 3.1.2 (3)]{KraemerMicrolocalII}. 
The top degree Segre class is the fundamental class
\[
 s_{\dim Z}(Z) \;=\; [Z].
\]
\end{remark}

Since clean cycles live on the {\em projective} cotangent bundle, there is no Segre class in degree $g=\dim A$. We view the {\em total Segre class} 
$s(\PLambda) :=  s_0(\PLambda) + \cdots + s_{g-1}(\PLambda)$ as an element of quotient
\[
\CH_{<g}(A) \;:=\; \CH_\bullet(A)/\CH_g(A).
\]
To define a ring structure on this quotient, recall that the additive group $\CH_\bullet(A)$ comes with a natural ring structure where the product is given by the Pontryagin product
\[
[X]*[Y] \;:=\; \sigma_* [X\times Y] 
\quad \textnormal{for the sum morphism} \quad 
\sigma\colon X\times Y \to X+Y\subset A.
\]
and that $\CH_g(A) \subset \CH_\bullet(A)$ is an ideal for the Pontryagin product. Working with the truncated Chow ring has the advantage that the total Segre class is compatible with the convolution product of clean cycles in the following sense:

\begin{lemma} \label{Lemma:SegreClassAndConvolution}
	Let $\PLambda_1, \PLambda_2\in \cL(A)$. If both Gauss maps
	$
	\gamma_{\PLambda_i}\colon \Supp(\PLambda_i) \to \bbP_A
	$	
	are finite morphisms, then 
	\[s(\PLambda_1 \circ \PLambda_2) = s(\PLambda_1)*s(\PLambda_2)
	\quad \textnormal{\em in} \quad \CH_{<g}(A). \]
\end{lemma} 

\begin{proof} 
	See \cite[lemma~3.3.1]{KraemerMicrolocalII}.	
\end{proof} 

Thus, the convolution product of clean cycles can be controlled via Pontryagin products of Segre classes. For the latter, one can use the following observation:

\begin{lemma} \label{Lemma:NonVanishingSegreClassPontryaginProduct} Let $X,Y\subset A$ be  proper reduced subvarieties. Suppose that every irreducible component of maximal dimension in $Y$ is of general type and that at least one irreducible component of~maximal dimension in~$X$ is nondegenerate. Then the cycle 
\[ s(\PLambda_X) \ast s(\PLambda_Y) \]
is nonzero and effective in all degrees $\le \min \{ \dim X + \dim Y, \dim A - 1\}$.
\end{lemma}

\begin{proof}
	Since the Pontryagin product is bilinear and the Pontryagin product of two effective cycles is effective or zero, it suffices to show the statement when $X$ and~$Y$ are both integral. Let $d=\dim X$, $e=\dim Y$ and $g=\dim A$, and consider the Segre class
	\[
	s_{m-d}(\PLambda_Y)
	\;\in\; \CH_{m-d}(A)
	\quad \textnormal{for} \quad 
	d\;\le\; m\;\le \; \min\{ d+e, g - 1 \}.
	\]
	This class is represented by an effective cycle since $m-d\in \{0,1,\dots, e\}$. For any irreducible component $Z_{m-d}\subset A$ of an effective cycle representing this class, we have
	\[
	s_d(\PLambda_X) * s_{m-d}(\PLambda_Y) \;=\; [X]*s_{m-d}(\PLambda_Y) 
	\;=\; [X] * [Z_{m-d}] + \cdots \;\in\; \CH_{m}(A)
	\]
	where $\cdots$ stands for a cycle which is effective or zero. Since by assumption $X$ is nondegenerate, we furthermore know from~\cref{Lem:SumOfNondegenerate} that the sum morphism 
	\[
	\sigma\colon \quad X\times Z_{m-d} \;\too\; X+Z_{m-d} \;\subset\; A
	\]
	is generically finite onto its image. So $[X]*[Z_{m-d}] = \sigma_*([X\times Z_{m-d}])$ is an effective class in $\CH_m(A)$. In conclusion, we see that the Pontryagin product 
	$s(\PLambda_X) * s(\PLambda_Y)$ 
	is nonzero and effective in all degrees $m$ with $d\le m \le \min\{d+e, g-1\}$. In the remaining range $0\le m<d$ the effectivity of the Pontryagin product is trivial because in that range we can look at $s_m(\PLambda_X) * s_0(\PLambda_Y) = \deg(\PLambda_Y) \cdot s_m(\PLambda_X)$; note that $\deg(\PLambda_Y)>0$ because $Y$ is of general type.
\end{proof}

\begin{corollary} \label{cor:DimensionOfConvolutionViaSegre}
Let $X,Y\subset A$ be reduced subvarieties, possibly reducible. If the Gauss maps $\gamma_X$, $\gamma_Y$ are both finite morphisms, then
\[ 
	\dim \pi(\Supp(\PLambda_X\circ \PLambda_Y)) 
	\;=\;
	\min \{ \dim X + \dim Y, \dim A - 1 \},
\]
where $\pi \colon A \times \bbP_A \to \bbP_A$ is the projection.
\end{corollary}

\begin{proof} 
Combine \cref{Lemma:SegreClassAndConvolution} and \cref{Lemma:NonVanishingSegreClassPontryaginProduct}.
\end{proof} 

\subsection{Clean characteristic cycles} \label{subsec:clean-cc} From now on and until the end of this section, we work over $k=\bbC$. Recall that to any perverse sheaf $P\in \Perv(A, \bbC)$ one may attach a characteristic cycle~\cite[definition~4.3.19]{DimcaSheaves}, a finite formal sum of conormal varieties
\[
 \CC(P) \;=\; \sum_{Z\subset A} m_Z(P) \cdot \Lambda_Z
 \quad \textnormal{with} \quad 
 m_Z(P) \;\in\; \bbN.
\]
Here the sum runs over all integral subvarieties $Z\subset A$, and only finitely many~$m_Z(P)$ are nonzero. These cycles contain a lot of information, e.g., the Dubson-Kashiwara index formula shows that we can read off the topological Euler characteristic as
\[
 \chi(A, P) \;=\; \sum_{Z\subset A} m_Z(P) \cdot \deg(\PLambda_Z)
\]
where $\deg(\PLambda_Z)$ is the degree of the Gauss map from \cref{Sec:GaussMap}, see \cite{FraneckiKapranov}. Passing from $CC(P)$ to its projectivization and discarding all components which are not clean, we define the {\em clean characteristic cycle} by
\[
\cc(P) \;:=\; \sum_{\substack{\textup{$Z\subset A$ of}\\ \textup{general type}}} m_Z(P) \cdot \PLambda_Z.
\]
It contains all information needed for the Dubson-Kashiwara index formula.  This index formula implies that for $P\in \Perv(A, \bbC)$ we have $\cc(P)=0$ if and only if~$P\in \rmS(A, \bbC)$. So the clean characteristic cycle of perverse sheaves is defined on the abelian quotient category $\Pbar(A, \bbC)=\Perv(A, \bbC)/\rmS(A, \bbC)$. By additivity in short exact sequences we then obtain a group homomorphism from the Grothendieck group of this abelian quotient category to the group of clean cycles:
\[
 \cc\colon \quad K(\Pbar(A, \bbC)) \;\too\; \cL(A)
\]
The Grothendieck group of an abelian tensor category is not just an abelian group, but also a ring with the product given by the tensor product, which in our case is the convolution product $*$ of perverse sheaves. By~\cite[th.~2.1.1 and ex.~1.3.2]{KraemerMicrolocalII} we have
\[
 \cc(P_1 * P_2) \;=\; \cc(P_1) \circ \cc(P_2) 
 \quad \textnormal{for all} \quad 
 P_1, P_2 \;\in\; \Pbar(A, \bbC),
\]
where $\circ$ is the convolution product of clean cycles introduced previously.

\begin{example} 
 Passing to characteristic cycles is useful since the convolution of clean cycles is easier to control than the convolution of perverse sheaves. For instance, in \cref{cor:derived-group} we have seen that for any perverse sheaf $P\in \Perv(A, \bbC)$ the connected component of its Tannaka group can be realized as the Tannaka group of $[d]_*P$ for any integer $d \ge 1$ with $d\cdot \Gamma_P=\{0\}$. Here
 \[
 \Gamma_P \;:=\; \bigl \{ a\in A(\bbC)_\tors \mid \delta_a \;\in\; \langle P \rangle \bigr\} 
 \;\subset\; A(\bbC)
 \]
 is a finite abelian group, but usually hard to control. However, we have $\Gamma_P \subset \Gamma_{\cc(P)}$ for 
 \[
 \Gamma_{\cc(P)} \;:=\; \bigl\{ a\in A(\bbC)_\tors \mid \PLambda_{\{a\}} \;\in\; \langle \cc(P)\rangle \bigr\}, 
 \]
 and this latter group depends only on the characteristic cycle of the given perverse sheaf. This will be useful in the proof of part (2) in \cref{thm:cc-and-weights} below. 
\end{example} 

\subsection{Highest weight theory} We want to use characteristic cycles to study the tensor category generated by a given semisimple perverse sheaf $P\in \Perv(A, \bbC)$. To pass from Tannaka groups to their connected component, we introduce the following notation:

\begin{definition}\label{def:m}
Let $m\ge 1$ be the smallest positive integer with $m\cdot \Gamma_{\cc(P)} = \{0\}$, and 
\[
 P^\circ \;:=\; [m]_*P.
\]
Fix a fiber functor $\xi \colon \langle P^\circ \rangle \to \Vect(\bbC)$, and consider the fiber functor obtained as the composite
\[ 
\omega\colon \quad 
\begin{tikzcd} 
 \langle P \rangle \ar[r, "{[m]_*}"] 
 &
 \langle P^\circ \rangle \ar[r, "\xi"]
 & 
 \Vect(\bbC).
\end{tikzcd} 
\]
Its associated Tannaka groups are
\[ G \;:=\; G_\omega(P) \;\supseteq\; G^\circ \;=\; G_{\xi}(P^\circ),
\]
where the rightmost equality follows from \cref{cor:derived-group}.
\end{definition}

Recall that $G$ is a reductive group over $\bbC$ and $\omega$ induces an equivalence of abelian tensor categories
\[
 \omega: \quad \langle P \rangle  \;\stackrel{\sim}{\too} \;  \Rep_\bbC(G).
\] 
In what follows, we will assume that the perverse sheaf $P\in \Perv(A, \bbC)$ is nondivisible and $\det(P) = \delta_0$ (the latter can be achieved by replacing~$P$ with a translate). Then by \cref{prop:connected-component} and \cref{cor:derived-group} the connected component $G^\circ$ is a semisimple group. By highest weight theory, its representation ring has the form
\[
\R(G^\circ) \;:=\; K(\Rep_\bbC(G^\circ)) \;=\; \bbZ[\sfX]^W
\]
where $\sfX:=\Hom(T, \bbG_m)$ is the character group of a maximal torus $T\subset G$, endowed with the natural action of the Weyl group $W=N_G(T)/Z_G(T)$, and we denote by~$\bbZ[\sfX]^W\subset \bbZ[\sfX]$ the subring of Weyl group invariants. Recall that for semisimple groups the universal cover inherits the structure of an algebraic group and the covering map is an isogeny
\[
\begin{tikzcd}[column sep=15pt] p\colon \hspace{-11pt} & \tilde{G} \ar[r,twoheadrightarrow] & G^\circ \;\subset\; G. \end{tikzcd}
\]
Hence, $\tilde{T}:=p^{-1}(T)$ is a maximal torus in $  \tilde{G}$, and $p$ induces an isomorphism of Weyl groups
\[
 N_{\tilde{G}}(\tilde{T}) / \tilde{T} \;\stackrel{\sim}{\too}  \; 
  N_{G}(T) / T \; =: \; W
\]
by means of which these groups are identified in what follows. Moreover, $p$ embeds the character group as a subgroup $\sfX \subset \tilde{\sfX}:=\Hom(\tilde{T}, \bbG_m)$ of finite index. 

\begin{definition}\label{def:d}
We denote by $d \ge 1$ the smallest positive integer with $d\cdot \tilde{\sfX} \subset \sfX$. 
\end{definition} 

The multiplication by $d$ then gives a morphism $[d]\colon \tilde{\sfX} \to \sfX$, and we have a commutative diagram
\[
\begin{tikzcd} 
\R(G^\circ) \ar[r, hook] \ar[d, equals] 
& \R(\tilde{G}) \ar[r] \ar[d, equals]
& \R(G^\circ) \ar[d, equals] 
\\
\bbZ[\sfX]^W \ar[r, hook] 
& \bbZ[\tilde{\sfX}]^W \ar[r, "{[d]_*}"] 
& \bbZ[\sfX]^W
\end{tikzcd} 
\]
where the top row is the $d$-th Adams operation $\Psi_d: \R(G^\circ) \to \R(G^\circ)$. 
 Even though the universal cover $\tilde{G}$ might not be realized as the Tannaka group of a perverse sheaf, we can relate its representations to clean cycles as follows (here we say that a statement holds for a \emph{very general} point of $\bbP_A(\bbC)$ if it holds for all points outside a countable union of proper subvarieties):

\begin{theorem} \label{thm:cc-and-weights}
Let $m, d\ge 1$ be as above. \smallskip
\begin{enumerate} 
\item The following diagram of ring homomorphisms is commutative:
\[
\begin{tikzcd}
\R(G) \ar[rr, "(-)_{\rvert G^\circ}"] \ar[d, swap, "\omega^{-1}"] 
&&
 \R(G^\circ) \ar[r, hook] \ar[d, swap, "\xi^{-1}"]
& \R(\tilde{G}) \ar[r] 
& \R(G^\circ) \ar[d, "\xi^{-1}"] 
\\
K(\langle P \rangle) \ar[d, swap, "\cc"] \ar[rr, "{[m]_*}"]
&&
 K(\langle P^\circ \rangle) \ar[d, swap, "\cc"] 
&& K(\langle P^\circ \rangle) \ar[d, "\cc"]
\\
\cL(A) \ar[rr, "{[m]_*}"] 
&&
 \cL(A) \ar[rr, "{[d]_*}"] 
&& \cL(A)
\end{tikzcd} 
\]
\item For very general $v\in \bbP_A(\bbC)$, there is a group homomorphism $\varphi_v\colon \sfX\to A(\bbC)$ such that the following diagram commutes:
\[
\begin{tikzcd}
R(G^\circ) \ar[r, equals] \ar[d, swap, "\cc \circ \xi^{-1}"] 
& \bbZ[\sfX]^W \ar[r, hook] 
& \bbZ[\sfX] \ar[d, "\varphi_v"] 
\\
\cL(A) \ar[rr, "\PLambda \mapsto \PLambda_{v}"] 
&& \bbZ[A(\bbC)] = Z_0(A) \hspace*{-5em}
\end{tikzcd} 
\]
where $\PLambda_v$ is the fiber of the Gauss map $\PLambda \to \bbP_A$ seen as a $0$-cycle on $A$.
\end{enumerate} 
\end{theorem}

\begin{proof} 
See~\cite[th.~2.2.3]{KraemerMicrolocalII}. For part (2), let $\Gamma \le A(\bbC)$ be the subgroup generated by the points in the fiber of the Gauss map $\gamma\colon \cc(P^\circ) \to \bbP_A$ over a very general point $v$, then loc.~cit.~gives $\varphi_v\colon \sfX \to \Gamma/\Gamma_\tors$ with the required properties. It then only remains to note that in our situation the group $\Gamma$ is free since our assumption $m\cdot \Gamma_{\cc(P)}=\{0\}$ implies that the subring $\langle \cc(P^\circ)\rangle \subset \cL(A)$ does not contain any conormal variety to a torsion point in $A(\bbC)$.
\end{proof}

\begin{definition} \label{defn:cc-lambda}
	For $\beta \in \tilde{\sfX}$ let $[\beta]\in \bbZ[\tilde{\sfX}]$ denote the corresponding basis vector in the group algebra. Note that the multiplication of basis vectors in the group algebra is defined by $[\alpha] \cdot [\beta] = [\alpha + \beta]$, and $[\alpha + \beta] \neq [\alpha] + [\beta]$.
	The subring~$\bbZ[\tilde{\sfX}]^W\subset \bbZ[\tilde{\sfX}]$ of Weyl group invariants has as its underlying additive group the free abelian group with $\bbZ$-basis consisting of the vectors
	\[
	[ W.\alpha ]  \;:=\; \sum_{\beta \in W.\alpha } [\beta] 
	\;\in\; \bbZ[\tilde{\sfX}]^W
	\]	
	where $\alpha$ runs through the dominant integral weights in $\tilde{\sfX}$ and $W.\alpha \subset \tilde{\sfX}$  denotes its orbit under the Weyl group. Multiplying by the integer $d$ from above, we obtain an element
	$[W.d\alpha] \in \bbZ[\sfX]^W = \R(G^\circ)$. Applying the inverse of $\xi\colon K(\langle P^\circ \rangle) \stackrel{\sim}{\too} \R(G^\circ)$ to this element of the representation ring and taking its characteristic cycle, we obtain a clean cycle
	\[
	\cc(P, \alpha) \;:=\; \cc( \xi^{-1} [W.d\alpha] ) \;\in\;  \cL(A).
	\]
	Note that for any integer $n\neq 0$ we have $\cc(P, n\alpha) = [n]_* \cc(P, \alpha)$.
\end{definition} 

\begin{remark} \label{Rmk:CharCycleWeightDefinedOnSubfield}
By construction, $\cc(P, \alpha)$ lies in the subring $\langle \cc(P^\circ)\rangle \subset \cL(A)$. In particular, if the cycle $\cc(P)$ is defined over a given algebraically closed subfield of~$\bbC$, then so is $\cc(P, \alpha)$. 
\end{remark} 

\begin{lemma}\label{lem:weight_multiplicities} For any $\alpha \in \tilde{\sfX}$ the cycle $\cc(P, \alpha) \in \cL(A)$ is effective. Moreover,  with $m$ and $d$ as in   \cref{def:m} and \cref{def:d},   we  have
	\[
	[dm]_* \cc(Q) \;=\; \sum_{\alpha} m_\alpha(V) \cdot \cc(P, \alpha)
	\quad  
	\]
	for any $Q\in \langle P \rangle$ and $V=\omega(Q) \in \Rep(G)$ with weight multiplicities $m_\alpha(V)\in \bbN$.	
\end{lemma}

\begin{proof} 
	Recall that the homomorphism $\varphi_v\colon \bbZ[\sfX] \to \bbZ[A(\bbC)]$ is induced by the group homomorphism $\varphi_v\colon \sfX \to A(\bbC)$ in part (2) of \cref{thm:cc-and-weights}. Hence, it sends the submonoid $\bbN[\sfX]\subset \bbZ[\sfX]$ into  the submonoid  $\bbN[A(\bbC)]$ of effective $0$-cycles. This construction works for {\em very} general $v\in \bbP_A(\bbC)$ only. However, if a clean cycle is known to have effective fibers over a very general cotangent vector $v\in \bbP_A(\bbC)$, then the cycle is effective.
	Hence, the claim about effectivity follows. The formula for multiplicities holds by construction. 
\end{proof} 

By \cref{lem:weight_multiplicities},  the weight multiplicities $m_\alpha(V)$ give us information about the multiplicities in characteristic cycles and vice versa. For example, we say that a representation of a connected reductive group is  \emph{minuscule} if it is an irreducible nontrivial representation whose weights for a maximal torus form a single orbit under the Weyl group. A representation of an arbitrary reductive group is called minuscule if its restriction to the connected component of the identity is so. This is a very restrictive condition: For the simple Dynkin types the table in~\cref{sec:IntroSimplicity} shows that the only minuscule representations other than standard representations of classical groups are the wedge powers of the standard representation in type $A$, spin and half-spin representations in types $B$ and $D$, and the representations of dimension $27$ and $56$ of the exceptional groups of type $E_6$ and  $E_7$. The previous lemma shows that for any perverse sheaf whose characteristic cycle is integral, the corresponding representation must be minuscule~\cite[cor.~1.10]{KraemerMicrolocalI}:

\begin{corollary} \label{cor:minuscule}
	Let $Y\subset A$ be a nondivisible subvariety, and let $P\in \Perv(A, \bbC)$ be a simple perverse sheaf with clean characteristic cycle $\cc(P)=\PLambda_Y$. Then
	 $\omega(P)$ is a minuscule representation of the group $G=G_\omega(P)$.
\end{corollary}

\begin{proof} 
	Since $Y$ is nondivisible, the cycle $[dm]_* \cc(P)$ is integral. In the above identity for the weight multiplicities of the representation $V=\omega(\delta_Y)$ then $m_\alpha(V)\neq 0$ for at most one dominant weight $\alpha$, and this weight must enter with multiplicity one.
\end{proof} 

We now want to give a geometric description of the cycles $\cc(P,\alpha)\in \cL(A)$ in cases when $G=G_\omega(P)$ is a classical group. The idea is to express arbitrary weights in terms of the weights in the standard representation, similar to the argument in~\cite{KraemerThetaSummands}. We do this for each of the classical Dynkin types $A, B, D$ separately; the computation for the Dynkin type $C$ is similar, but we omit it since in type~$C$ there are no minuscule representations other than the standard representation.

\subsection{Weyl orbits for type \texorpdfstring{$A$}{A}} \label{subsec:typeA}
 For $\tilde{G} = \SL_n$, let us denote by $\epsilon_1,\dots, \epsilon_n\in \tilde{\sfX}$ the weights of the standard representation. The triviality of the determinant implies that $\epsilon_1 + \cdots + \epsilon_n = 0$, and the wedge powers of the standard representation have as highest weights the fundamental weights
\[
\varpi_i \;=\; \epsilon_1 + \cdots + \epsilon_i 
\quad \textnormal{for $0< i < n$}.
\]
The dominant integral weights are the $\bbN$-linear combinations of the fundamental weights, i.e.,~the weights 
\[
\alpha \;=\; \sum_{\nu=1}^{n-1} \alpha_\nu \epsilon_\nu
\quad \textnormal{with integers} \quad 
\alpha_1 \ge \cdots \ge \alpha_{n-1} \ge 0.
\]

More generally, for $\alpha = (\alpha_1, \dots, \alpha_n) \in \bbZ^n$, let $\ell = \max\{ \nu \mid \alpha_\nu \neq 0\}$ be the length of the $n$-tuple $\alpha$, and let $I(\ell, n)$ be the set of injective maps 
\[
\begin{tikzcd}[column sep=15pt]   \iota \colon \hspace{-11pt}& \{1,\dots, \ell\} \ar[r,hookrightarrow] & \{1,\dots, n\}. \end{tikzcd}
\]
Then $\alpha$ has the Weyl group orbit
\[
W.\alpha \;=\; \bigl\{  \alpha_1 \epsilon_{\iota(1)} + \cdots + \alpha_\ell \epsilon_{\iota(\ell)} \in \tilde{\sfX} \mid \iota \in I(\ell, n) \bigr\}.
\]
Each weight in this orbit is obtained for precisely $N(\alpha)$ distinct choices of $\iota \in I(\ell, n)$, where
\[
N(\alpha) \;:=\; \prod_{i\in \bbZ} \ell_i! 
\quad \textnormal{with} \quad 
\ell_i \;:=\; \# \{ \nu \mid \alpha_\nu = i \},
\]
with the convention $0!=1$ so that the above product is finite. Hence, we have
\[
[ W.\alpha ] \;=\; \frac{1}{N(\alpha)} \cdot \sum_{\iota \in I(\ell, n)} \bigl [ \alpha_1 \epsilon_{\iota(1)} + \cdots + \alpha_\ell \epsilon_{\iota(\ell)} \bigr]
\;\in\; \bbZ[\sfX].
\]
To describe the clean cycle $\cc(P,\alpha)$ in these terms, we need some more notation:

\begin{definition} \label{defn:fiber-product-minus-diagonal}
Let $\PLambda \in \cL(A)$ be a reduced clean cycle. Let $U\subset \bbP_A$ be any open dense subset over which the Gauss map $\gamma_{\PLambda}\colon \PLambda \to \bbP_A$ restricts to a finite flat morphism, and let $\PLambda_{\mid U} := \gamma_{\PLambda}^{-1}(U)$ be its preimage. For an integer $\ell \ge 1$ consider the fiber product
\[
 \PLambda^{\times \ell}_{\vert U} \;:=\; \PLambda_{\vert U} \times_U \times \cdots \times_U \PLambda_{\vert U} 
 \;\subset\; A^\ell \times U
\]
and inside it the big diagonal
\[
 \Delta_\ell \;:=\; 
 \{ (p_1, \dots, p_\ell, v) \mid p_i = p_j \; \textnormal{for some $(i,j)$ with $i\neq j$} \} \;\subset\; \PLambda^{\times \ell}_{\vert U}.
\]	
Since the fiber product of finite flat morphisms is again a finite flat morphism, every irreducible component of the fiber product $\PLambda^{\times \ell}_{\mid U}$ is a finite flat cover of $U$. So the Zariski closure 
\[
  \PLambda^{[\ell]} \;:=\; \overline{\PLambda^{\times \ell}_{\mid U} \smallsetminus \Delta_{\ell} }
  \;\subset\; A^\ell \times \bbP_A
\]
does not depend on the specific choice of the open dense subset $U\subset \bbP_A$ over which~$\gamma_{\PLambda}$ is finite and flat. For any $\alpha = (\alpha_1, \dots, \alpha_\ell)\in \bbZ^\ell$ we then define a clean cycle as the pushforward
\[
 \PLambda^{\alpha} \;:=\; \frac{1}{N(\alpha)} \cdot \sigma_{\alpha \ast} (\PLambda^{[\ell]}) 
 \;\in\; \cL(A)
\]
for the `sum' morphism 
\[ \sigma_\alpha\colon A^\ell \times \bbP_A \too A\times \bbP_A, \quad (p_1, \dots, p_\ell, v) \mapsto (\alpha_1 p_1 + \cdots + \alpha_\ell p_\ell, v).\]
\end{definition} 

\begin{remark} \label{rem:symmetric-fiber-product-of-conormal}
The group $\frS_\ell$ acts on $A^\ell \times \bbP_A$ by permutations of the abelian variety factors, and this action restricts to an action on $\PLambda^{[\ell]} \subset A^\ell \times \bbP_A$. The morphism $\sigma_\alpha$ factors through the quotient by the subgroup 
\[
 \frS_\alpha \;:=\; \prod_{i \in \bbZ} \frS_{\ell_i}
 \;\subset\; \frS_\ell 
 \quad \textnormal{where} \quad 
 \ell_i \;:=\; \# \{ \nu \mid \alpha_\nu = i\},
\] 	
as shown in the following commutative diagram:
	\[
\begin{tikzpicture}[scale=1]
\def\sqrtthree{1.73205080757};

\node[label={[shift={(-0.52,-.35)}]$A^\ell \times \bbP_A$}] at ( -\sqrtthree/2, 1/2) (a) { };
\node[label={[shift={(.5,-.35)}]$A \times \bbP_A$}] at (\sqrtthree/2, 1/2) (c) {};
\node[label={[shift={(0,-.55)}]$(A^\ell \times \bbP_A)/\frS_\alpha$}] at (0, -1) (b) {};

\draw[->, shorten <= 2pt, shorten >= 2pt] (a) edge  node[midway, above] {$\scriptstyle \sigma_\alpha$}  (c);
\draw[->, shorten <= 2pt, shorten >= 2pt] (a) edge node[midway, left] {$\scriptstyle q $ }  (b);
\draw[->, shorten <= 2pt, shorten >= 2pt] (b) edge node[midway, right] {$\scriptstyle  \exists ! \tilde{\sigma}_\alpha $ }  (c);
\end{tikzpicture}
\]
Here, the quotient morphism $q$ is finite of degree $N(\alpha)$. The restriction of $q$ to the subvariety $\PLambda^{[\ell]}\subset A^\ell \times \bbP_A$ is still finite of the same degree over its image, as one may see by looking at a general fiber of the Gauss map. Hence, \cref{defn:fiber-product-minus-diagonal} amounts to
\[
 \PLambda^\alpha \;=\; \tilde{\sigma}_{\alpha \ast}(\tilde{\PLambda}^\alpha)
 \quad \textnormal{where} \quad 
 \tilde{\PLambda}^\alpha \;:=\; \PLambda^{[\ell]}/\frS_\alpha \;\subset\; (A^\ell \times \bbP_A)/\frS_\alpha.
\]
In particular, it follows that if $\PLambda^\alpha$ is reduced resp.~integral, then so is $\tilde{\PLambda}^\alpha$.
\end{remark} 

We now obtain from \cref{thm:cc-and-weights} the following description of the clean cycles corresponding to the dominant weights:

\begin{lemma} \label{lem:cc-in-type-A}
	Let $G=G_\omega(P)$ be semisimple with universal cover $\tilde{G} \simeq \SL_n$, and let~$\alpha \in \tilde{\sfX}$ be a dominant weight. If the clean cycle $\cc(P,\alpha)$ is reduced resp.~integral, then  
	$
	\PLambda :=\cc(P, \epsilon_1)
	$
	is reduced resp.~integral, and 
	\[	\cc(P,\alpha) \;=\; 
	\PLambda^\alpha
	\]
	where on the right-hand side we identify $\alpha=\sum_{i=1}^n \alpha_i \epsilon_i$ with the tuple $(\alpha_1, \dots, \alpha_n)$.
\end{lemma} 

\begin{proof} 
	Let $v\in \bbP_A$ be a very general cotangent direction, and let $p_1, \dots, p_n\in A(\bbC)$ be the points in the fiber of the Gauss map $\PLambda \to \bbP_A$ over this direction, counted with multiplicities so that 
	\[ \PLambda_v \;=\; [p_1] + \cdots + [p_n] 
    \] 
	inside $\bbZ[A(\bbC)]$. Recall from \cref{thm:cc-and-weights} that the points in the fiber are precisely the images of the weights in the Weyl group orbit $W\cdot d\epsilon_1 = \{ d\epsilon_1, \dots, d\epsilon_n\}\subset \sfX$ under the homomorphism $\varphi_v\colon \sfX\to A(\bbC)$. Up to relabelling indices, we may assume that $\varphi_v(d\epsilon_i) = p_i$ for all $i$. Writing the weight as $\alpha = (\alpha_1, \dots, \alpha_\ell)$ with $\ell< n$, we have
	\begin{eqnarray*}
		\cc(P,\alpha)_v 
		\;=\; \cc( \xi^{-1} [W.d\alpha] ) _v 
		&\;=\;& \varphi_v([W.d\alpha])
		\\[0.6em]
		&\;=\;& 
		\frac{1}{N(\alpha)} \cdot \sum_{\iota \in I(\ell, n)} [ \varphi_v(\alpha_1 d \epsilon_{\iota(1)} + \cdots + \alpha_\ell d \epsilon_{\iota(\ell)} )] 
		\\
		&\;=\;& 
		\frac{1}{N(\alpha)} \cdot \sum_{\iota \in I(\ell, n)} [ \alpha_1 p_{\iota(1)} + \cdots + \alpha_\ell p_{\iota(\ell)} ] .
	\end{eqnarray*} 
	For very general $v$ this $0$-cycle contains no multiple points, since we assumed $\cc(P,\alpha)$ to be reduced. Therefore, each point 
	\[ p \;=\; \alpha_1 p_{\iota(1)} + \cdots + \alpha_\ell p_{\iota(\ell)} \]
	enters in the above expression for $\cc(P,\alpha)_v$ only for $N(\alpha)$ different choices of $\iota$. But on the other hand $p$ does not change if we replace the map $\iota\colon \{1,\dots, \ell\} \into \{1,\dots, n\}$ by $\iota \circ \tau$ for any permutation 
	\[
	\tau \;\in\; \frS_\alpha \;=\; \prod_{i \in \bbZ} \frS_{\ell_i} \;\subset\; \frS_\ell.
	\]
	Since $N(\alpha)=|\frS_\alpha|$, it follows from the above that for all $\iota, \iota^* \in I(\ell, n)$ we have the equivalence
	\[
	\sum_{\nu = 1}^\ell \alpha_\nu p_{\iota(\nu)} 
	\;=\; 
	\sum_{\nu = 1}^\ell \alpha_\nu p_{\iota^*(\nu)} 
	\quad 
	\Longleftrightarrow 
	\quad 
	\exists \tau \in \prod_{i \in \bbZ} \frS_{\ell_i}: \; 
	\iota^* = \iota \circ \tau
	\]
	Since $\ell < n$, this forces $p_1, \dots, p_n\in A(\bbC)$ to be pairwise distinct, so the $0$-cycle~$\PLambda_v$ is reduced and $\PLambda$ must be reduced as well. Now let $U\subset \bbP_A$ be an open neighborhood of~$v$ such that the Gauss map $\PLambda_{\vert U} \to U$ is finite and flat. For the complement of the big diagonal, we then get a bijection
	\[
	I(\ell, n) \;\too\; (\PLambda_{\vert U}^{\times \ell} \smallsetminus \Delta_\ell)_v, 
	\quad 
	\iota  \;\longmapsto\; (p_{\iota(1)}, \dots, p_{\iota(\ell)}, v)
	\]
	and therefore\medskip
	\[
		\cc(P,\alpha)_v \;=\; 
		\frac{1}{N(\alpha)} \cdot \sum_{\iota} [\alpha_1 p_{\iota(1)} + \cdots + \alpha_{\ell} p_{\iota(\ell)}] 
		\;=\; 
		\frac{1}{N(\alpha)} \cdot (\sigma_{\alpha *}(\PLambda^{\times \ell}_{\vert U} \smallsetminus \Delta_\ell))_v
	\]
	which is by definition the fiber of the cycle $\PLambda^\alpha$ over $v$.
	Since this holds for very general $v$ it follows that $\cc(P, \alpha)=\PLambda^\alpha$ as claimed.
	
	\medskip 
	
	If we moreover assume that the cycle $\cc(P, \alpha)$ is integral, then it follows from \cref{rem:symmetric-fiber-product-of-conormal} that the cycle
	\[
	 \tilde{\PLambda}^\alpha \;=\; \PLambda^{[\ell]}/\frS_\alpha 
	\]
	is integral as well. We claim that in this case, the cycle $\PLambda$ must be integral. Indeed,  suppose for a  contradiction that $\PLambda = \PLambda_1 + \PLambda_2$ with effective cycles $\PLambda_1, \PLambda_2 \in \cL(A)$. We already know that the fiber of the Gauss map $\gamma_{\PLambda}\colon \PLambda \to \bbP_A$ over a general point $v\in \bbP_A(\bbC)$ is reduced, hence every point in this fiber lies either on $\PLambda_1$ or on $\PLambda_2$ but not on both. Thus,  each point on the big diagonal comes from a point on the big diagonal of one of the two summands. For the complement of the big diagonal, we therefore obtain that
	\[
	\PLambda_{\vert U}^{\times \ell} \smallsetminus \Delta_\ell 
	\;\supset\; 
	\sum_{\ell_1 + \ell_2 = \ell}
	(\PLambda_{1\vert U}^{\times \ell_1} \smallsetminus \Delta_{1, \ell}) \times_U (\PLambda_{2\vert U}^{\times \ell_2} \smallsetminus \Delta_{2, \ell}) 
	\]
	where we denote the big diagonals in the summands by $\Delta_{i, \ell} \subset \PLambda_{i\vert U}^{\times \ell_i}$. Taking Zariski closure gives
	\[
	\PLambda^{[\ell]} \; \supset \; \sum_{\ell_1+\ell_2=\ell} \PLambda^{[\ell_1, \ell_2]} \] where 
	\[ 
	\PLambda^{[\ell_1, \ell_2]} 
	\;:=\; 
	\overline{(\PLambda_{1\vert U}^{\times \ell_1} \smallsetminus \Delta_{1, \ell}) \times_U (\PLambda_{2\vert U}^{\times \ell_2} \smallsetminus \Delta_{2, \ell}) }.
	\]
	Now from a look at the degree of the respective Gauss maps we see that $\PLambda^{[\ell_1, \ell_2]}\neq 0$ if and only if $ 0\le \ell_1 \le \deg(\PLambda_1)$ and  $ 0\le \ell_2 \le \deg(\PLambda_2) $. For $\ell_1 + \ell_2 = \ell$ these four inequalities are equivalent to
	\[
	 \max\{0, \ell - \deg(\PLambda_2)\} \le \ell_1 \le \min\{ \deg(\PLambda_1), \ell \}.
	\]
	This set of inequalities has at least two different solutions $\ell_1 \in \bbZ$, indeed we have $ \max\{0, \ell - \deg(\PLambda_2)\} < \min\{ \deg(\PLambda_1), \ell \}$ because all the occurring degrees of Gauss maps are strictly positive and because $\ell < n = \deg(\PLambda) = \deg(\PLambda_1) + \deg(\PLambda_2)$. In conclusion, this shows that there are at least two irreducible components of the form $\PLambda^{[\ell_1, \ell_2]}$ in $\PLambda^{[\ell]}$. One easily sees that no two such components are related to each other by a permutation of the factors in $A^\ell \times \bbP_A$, using again that the fibers of the Gauss maps for $\PLambda_1$ and for $\PLambda_2$ are disjoint. Hence, the quotient $\PLambda^{[\ell]}/\frS_\alpha$ has more than one irreducible component, which contradicts our assumption. 
\end{proof} 

\subsection{Weyl orbits for type \texorpdfstring{$B$}{B}} For $\tilde{G} = \Spin_{2n+1}$, let $\epsilon_{\pm 1}, \dots, \epsilon_{\pm n} \in \tilde{\sfX}$ be the nontrivial weights of the standard representation of the orthogonal group, with the relations $\epsilon_{-i}=-\epsilon_i$. The fundamental weights are 
\[
\varpi_i \;=\; 
\begin{cases} 
\epsilon_1 + \cdots + \epsilon_i & \textnormal{for $i<n$}, \\
\tfrac{1}{2} \cdot (\epsilon_1 + \cdots + \epsilon_n) & \textnormal{for $i=n$}.
\end{cases} 
\]
The first $n-1$ fundamental weights are again highest weights of wedge powers of the standard representation; the last fundamental weight is the highest weight of the spin representation. The dominant integral weights are the weights of the form
\[
\alpha \;=\; \sum_{\nu=1}^n \alpha_\nu \epsilon_\nu
\quad \textnormal{with} \quad \alpha_1 \ge \cdots \ge \alpha_n \ge 0,
\]
where $2\alpha_i\in \bbZ$ are either all even or all odd. Put $\ell = \max\{\nu \mid \alpha_\nu \neq 0\}$, and let~$I(\ell, \pm n)$ be the set of maps
\[
\begin{tikzcd}[column sep=15pt]   \iota \colon \hspace{-11pt}& \{1,\dots, \ell\} \ar[r,hookrightarrow] & \{\pm 1, \dots, \pm n\}\end{tikzcd}
\]
with the property that the map $\nu\mapsto |\iota(\nu)|$ is still injective. Then $\alpha$ has the Weyl group orbit
\[
W.\alpha \;=\; 
\bigl\{
\alpha_1 \epsilon_{\iota(1)} + \cdots + \alpha_\ell \epsilon_{\iota(\ell)} \mid 
\iota \in I(\ell, \pm n)  
\bigr\}.
\]
Each weight in this orbit occurs for precisely $N(\alpha)$ different choices of $\iota \in I(\ell, \pm n)$, where the number $N(\alpha)$ is defined as above; note that different sign choices will lead to different weights and hence the extra signs do not change the count. 

\medskip 

In order to translate this back to geometry, we need to adapt \cref{defn:fiber-product-minus-diagonal} to the symmetric case:

\begin{definition} \label{defn:cc-in-type-B}
Let $\PLambda \in \cL(A)$ be a reduced clean cycle with $[-1]_\ast \PLambda = \PLambda$, and let $U\subset \bbP_A$ be an open dense subset over which all components of the cycle are finite and flat. For an integer $\ell \ge 1$ let
\begin{eqnarray*}
\Delta_\ell 
&\;:=\;& 
\{ (p_1, \dots, p_\ell, v) \mid p_i = p_j \; \textnormal{for some $(i,j)$ with $i\neq j$} \} \;\subset\; \PLambda^{\times \ell}_{\vert U}
\\	
\Delta^-_\ell 
&\;:=\;& 
\{ (p_1, \dots, p_\ell, v) \mid p_i = -p_j \; \textnormal{for some $(i,j)$ with $i\neq j$} \} \;\subset\; \PLambda^{\times \ell}_{\vert U}
\end{eqnarray*}	
be the big diagonal resp.~antidiagonal in the fiber product, and consider the Zariski closure
\[
\PLambda^{[\ell]}_{{\sym}} \;:=\; \overline{\PLambda^{\times \ell}_{\mid U} \smallsetminus (\Delta_\ell \cup \Delta_\ell^-) }
\;\subset\; A^\ell \times \bbP_A.
\]
For $\alpha = (\alpha_1, \dots, \alpha_\ell)\in \bbZ^\ell$ we obtain a clean cycle
\[
 \PLambda^\alpha_{\sym} 
 \;:=\; 
 \frac{1}{N(\alpha)} \cdot \sigma_{\alpha \ast}(\PLambda^{[\ell]}_{\sym}) 
 \;\in\; \cL(A)
\]
as the pushforward under the morphism \[\sigma_\alpha\colon A^\ell \times \bbP_A \too A\times \bbP_A, \quad (p_1, \dots, p_\ell, v) \longmapsto (\alpha_1 p_1 + \cdots + \alpha_\ell p_\ell, v)\]
\end{definition}

With this notation, we obtain from \cref{thm:cc-and-weights} the following description of the clean cycles corresponding to the dominant weights:

\begin{lemma} \label{lem:cc-in-type-B}
	Let $G=G_\omega(P)$ be semisimple with universal cover $\tilde{G} \simeq \Spin_{2n+1}$, and let $\alpha \in \bbZ \epsilon_1 + \cdots + \bbZ \epsilon_n \subset \tilde{\sfX}$ be a dominant weight. If the clean cycle $\cc(P, \alpha)$ is reduced, then 
	$
	\PLambda := \cc(P, \epsilon_1)
	$
	is reduced, and 
	\[ \cc(P, \alpha) \;=\; \PLambda^{\alpha}_{{\sym}} \]
	where on the right-hand side we identify $\alpha = \sum_{i=1}^n \alpha_i \epsilon _i$ with $(\alpha_1, \dots, \alpha_n)$.
\end{lemma} 

\begin{proof} 
	Let $v\in \bbP_A$ be a very general cotangent direction, and let $p_{\pm 1}, \dots, p_{\pm n}$ be the projection in $A(\bbC)$ of the $2n$ points in the fiber of the Gauss map $\PLambda \to \bbP_A$ over this direction, counted with multiplicities so that 
	\[ \PLambda_v \;=\; [p_1] + \cdots + [p_n] + [p_{-1}] + \cdots + [p_{-n}]\] 
	inside $\bbZ[A(\bbC)]$. Recall from \cref{thm:cc-and-weights} that the points in the fiber are precisely the images of the weights in the Weyl group orbit $W. d\epsilon_1 = \{ \pm d\epsilon_1, \dots, \pm d\epsilon_n\}\subset \sfX$ under the homomorphism $\varphi_v\colon \sfX\to A(\bbC)$. Up to relabelling indices, we may assume that $\varphi_v(\pm d\epsilon_i) = p_{\pm i} = \pm p_i$ for all $i$. Then as in the proof of \cref{lem:cc-in-type-A} one obtains
	\begin{eqnarray*}
		 \cc(P, \alpha)_v 
		&\;=\;& 
		\frac{1}{N(\alpha)} \cdot \sum_{\iota \in I(\ell, \pm n)} [ \alpha_1 p_{\iota(1)} + \cdots + \alpha_\ell p_{\iota(\ell)} ] .
	\end{eqnarray*} 
	For very general $v$ this $0$-cycle contains no multiple points, since we assumed $\cc(P, \alpha)$ to be reduced. The same counting argument as in \cref{lem:cc-in-type-A} then shows that for any~$\iota, \iota^* \in I(\ell, \pm n)$ we have the equivalence
	\[
	\sum_{\nu = 1}^\ell \alpha_\nu p_{\iota(\nu)} 
	\;=\; 
	\sum_{\nu = 1}^\ell \alpha_\nu p_{\iota^*(\nu)} 
	\quad 
	\Longleftrightarrow 
	\quad 
	\exists \tau \in \prod_{i \in \bbZ} \frS_{\ell_i}: \; 
	\iota^* = \iota \circ \tau.
	\]
	This forces $p_{\pm 1}, \dots, p_{\pm n}\in A(\bbC)$ to be pairwise distinct: Indeed, assuming that we had
	\[ p_i = p_j \quad \textnormal{for certain} \quad i, j \in \{\pm 1, \dots, \pm n\}
	\quad \textnormal{with} \quad i \;\neq\; \pm j, \]
	then the implication $\Longrightarrow$ in the above equivalence would fail for any $\iota, \iota^*\in I(\ell, \pm n)$ with 
	\[ \iota(1)=i=-\iota^*(1), \quad \iota(2)=-j=-\iota^*(2) \quad \textnormal{and} \quad \iota(\nu)=\iota^*(\nu)
	\quad \textnormal{for $\nu = 3,\dots, \ell$}, 
	\]
	a contradiction. This shows that the $0$-cycle $\PLambda_v$ is reduced, hence $\PLambda$ must be reduced as well. For the complement of the big diagonal and antidiagonal we then get a bijection
	\[
	I(\ell, \pm n) \;\too\; (\PLambda^{\times \ell}_{\vert U} \smallsetminus (\Delta_\ell \cup \Delta_\ell^-))_v, 
	\quad 
	\iota  \;\longmapsto\; (p_{\iota(1)}, \dots, p_{\iota(\ell)}, v)
	\]
	and can conclude as in \cref{lem:cc-in-type-A} that $\cc(P, \alpha) = \PLambda^\alpha_{\sym}$.
\end{proof}

\subsection{Weyl orbits for type \texorpdfstring{$D$}{D}} For $\tilde{G} = \Spin_{2n}$ we again let $\epsilon_{\pm 1}, \dots, \epsilon_{\pm n} \in \tilde{\sfX}$ denote the weights of the standard representation of the orthogonal group, with the relations $\epsilon_{-i}=-\epsilon_i$. The fundamental weights are 
\[
\varpi_i \;=\; 
\begin{cases} 
\epsilon_1 + \cdots + \epsilon_i & \textnormal{for $i<n-1$}, \\
\tfrac{1}{2}(\epsilon_1 + \cdots + \epsilon_{n-1} - \epsilon_n) & \textnormal{for $i=n-1$}, \\
\tfrac{1}{2}(\epsilon_1 + \cdots + \epsilon_{n-1} + \epsilon_n) & \textnormal{for $i=n$}. \\ 
\end{cases}
\]
The first $n-2$ fundamental weights are highest weights of wedge powers of the standard representation; the last two fundamental weights are the highest weights of the two spin representations. The dominant integral weights are the weights of the form 
\[
\alpha \;=\; \sum_{\nu=1}^n \alpha_\nu \epsilon_\nu
\quad \textnormal{with} \quad 
\alpha_1 \ge \cdots \ge \alpha_{n-1} \ge |\alpha_n| \ge 0,
\]
where $2\alpha_i \in \bbZ$ are either all even or all odd. Put $\ell = \max\{\nu \mid \alpha_\nu \neq 0\}$, and let~$I_{\textup{even/odd}}(\ell, \pm n)$ be the set of maps
\[
\begin{tikzcd}[column sep=15pt]   \iota \colon \hspace{-11pt}& \{1,\dots, \ell\} \ar[r,hookrightarrow] & \{\pm 1, \dots, \pm n\}\end{tikzcd}
\]
with the property that the map $\nu \mapsto |\iota(\nu)|$ is still injective and the number of negative values of $\iota$ is even resp.~odd. Then $\alpha$ has the Weyl group orbit
\[
W.\alpha \;=\; 
\bigl\{
\alpha_1 \epsilon_{\iota(1)} + \cdots + \alpha_\ell \epsilon_{\iota(\ell)} \mid \iota \in I_{\textup{even/odd}}(\ell, \pm n)
\bigr\}
\quad 
\textnormal{with} 
\quad 
\begin{cases} 
\text{even} & \text{if $\alpha_n \ge 0$}, \\
\text{odd} & \text{if $\alpha_n < 0$}.
\end{cases}
\]
Each weight in this orbit occurs for precisely $N(\alpha)$ different choices of $\iota$, again the signs do not matter for this count.

\medskip 

In order to translate this back to geometry, we need to refine \cref{defn:cc-in-type-B} as follows:

\begin{definition} \label{defn:cc-in-type-D}
Let $\PLambda \in \cL(A)$ be a reduced clean cycle with $[-1]_\ast \PLambda = \PLambda$, and let $U\subset \bbP_A$ be an open dense subset over which all components of the cycle are finite and \'etale. Labelling the points in a general fiber of the Gauss map $\PLambda_{\vert U} \to U$ in pairs of opposite points as 
\[ \gamma_{\PLambda}^{-1}(u) \;=\; \{ p_{\pm 1}, \dots, p_{\pm n}\}
\quad \textnormal{with} \quad 
p_{-i} \;=\; -p_i,
\]
we identify the monodromy group of the finite \'etale cover $\gamma_{\PLambda\vert U} \colon \PLambda_{\vert U} \to U$ as a subgroup of $(\pm 1)^n \rtimes \frS_n$. We say that the monodromy of the Gauss map is {\em even} if it is contained in the subgroup
\[
 (\pm 1)^n_+ \rtimes \frS_n 
 \quad \textnormal{where} \quad 
 (\pm 1)^n_+ \;:=\; \{ (a_1, \dots, a_n) \in (\pm 1)^n \mid a_1 \cdots a_n = +1 \}.
\] 
Then for $\ell = n$ in \cref{defn:cc-in-type-B} we obtain that the cycle $\PLambda^{[n]}_\sym$ on $A^n \times \bbP_A$ splits as a sum
\[
 \PLambda^{[n]}_\sym \;=\; \PLambda^{[n]}_{\sym,+} + \PLambda^{[n]}_{\sym,-} 
\]
where $\PLambda^{[n]}_{\sym, \pm}\subset \PLambda^{[n]}_\sym$ are defined by the condition that their fiber over $u\in U(k)$ contains an even resp. odd number of points with a negative sign, in other words
\[
 (\PLambda^{[n]}_{\sym, \pm})_u \;:=\; 
 \{ (p_{i_1}, \dots, p_{i_n}, u) \mid \prod_{\nu=1}^n \mathrm{sgn}(i_\nu) = \pm 1 \}.
\]
Note that this condition depends on the way we have labelled the points in the fiber $\gamma_Z^{-1}(u)$: The labelling by $\pm$ has no intrinsic meaning and is only used as a notational device to separate the two pieces in the decomposition, none of the two pieces is distinguished. We define
\[
 \PLambda^\alpha_{\sym, \pm} \;:=\; \sigma_{\alpha \ast} (\PLambda^{[n]}_{\sym,\pm}) \;\in\; \cL(A)
\]
for $\alpha = (\alpha_1, \dots, \alpha_n)\in \bbZ^n$ and the sum morphism $\sigma_\alpha$ as in \cref{defn:cc-in-type-B}. Note that while there is no intrinsic meaning to the labels $\pm$, it might nevertheless happen that $\pr_A(\PLambda^\alpha_{\sym, -})\neq \pr_A(\PLambda^\alpha_{\sym, +})$. We also note that by symmetry of $Z\subset A$ we can assume without loss of generality that $\alpha_n \ge 0$.
\end{definition}

\begin{lemma} \label{lem:cc-in-type-D}
	Let $G=G_\omega(P)$ be semisimple with universal cover $\tilde{G} \simeq \Spin_{2n}$, and let $\alpha \in \bbZ \epsilon_1 + \cdots + \bbZ \epsilon_n \subset \tilde{\sfX}$ be a dominant weight of length $\ell$. If $\cc(P, \alpha)$ is reduced, then 
	$
	\PLambda := \cc(P, \epsilon_1)
	$
	is reduced. In this case 
	\[ 
	\cc(P, \alpha) \;=\;
	\begin{cases}
	\PLambda^\alpha_\sym & \text{for $\ell < n$}, \\
	\PLambda^\alpha_{\sym, \epsilon} & \text{for $\ell = n$ and suitable $\epsilon \in \{+, -\}$},
	\end{cases}
	\]
	where on the right-hand side we identify $\alpha = \sum_{i=1}^n \alpha_i \epsilon_i$ with $(\alpha_1, \dots, \alpha_n)$.
\end{lemma}

\begin{proof} 
Similar to the argument for type $B_n$. 
\end{proof} 

\section{Simplicity of the Tannaka group} 

We now take a closer look at the Tannaka group of the perverse intersection complex on a smooth nondivisible subvariety. By \cref{cor:minuscule}, the corresponding representation is minuscule; the goal of this section is to show that, under suitable positivity assumptions, the Tannaka group is simple modulo its center.

\subsection{The simplicity criterion} \label{sec:simplicity-criterion}

For the rest of this section, we assume that $k$ is algebraically closed of characteristic zero.  Throughout, we fix a subvariety $X\subset A$ and denote by
\[
 \omega\colon \quad \langle \delta_X \rangle \;\too\; \Vect(\bbF)
\]
a fiber functor on the Tannaka category generated by the perverse intersection complex $\delta_X\in \Perv(A, \bbF)$. Consider the reductive Tannaka group $G_{X, \omega} := G_\omega(\delta_X)$ and denote by
\[
 G_{X, \omega}^\ast \;:=\; [G_{X, \omega}^\circ, G_{X, \omega}^\circ]
\]
the derived group of its connected component. This is a connected semisimple group, hence its Lie algebra is a product of simple Lie algebras. Recall that a connected algebraic group is \emph{simple} if its Lie algebra is simple, or equivalently, it does not contain connected (reduced) normal subgroups. Writing $g:= \dim A$, the goal of this section is the following simplicity criterion:

\begin{theorem} \label{Thm:TannakaGroupSimple} Suppose $g \ge 3$ and let $X$ be a smooth nondivisible subvariety with ample normal bundle. Then the following are equivalent:
\begin{enumerate}
\item The algebraic group $G_{X, \omega}^\ast$ is not simple;
\item There are smooth positive-dimensional subvarieties $X_1, X_2 \subset A$ such that the sum morphism induces an isomorphism
\[ X_1 \times X_2 \stackrel{\sim}{\too} X.\]
\end{enumerate}
\end{theorem}

\noindent The proof of this result will occupy the rest of this section, but let us first observe that the criterion applies in many cases. First, a smooth projective curve $X \subset A$ generating $A$ has ample normal bundle \cite[prop.~4.1]{Har71}, thus $G_{X, \omega}^\ast$ is simple as soon as $X$ is nondivisible and $g \ge 3$. More generally we have:

\begin{corollary} \label{Cor:TannakaGroupSimpleSpecialCases} Suppose $g \ge 3$ and let $X\subset A$ be a smooth nondivisible subvariety with ample normal bundle. Assume that one of the following conditions holds:\smallskip 
\begin{enumerate}
\item for $x\in X(k)$ the image of the Albanese morphism $\alb_{X, x} \colon X\to \Alb(X)$  is a nondegenerate subvariety of $\Alb(X)$ in the sense of \cref{section positivity}, or\smallskip
\item the natural map $\Alb(X) \to A$ is an isogeny.\smallskip
\end{enumerate}
Then the algebraic group $G_{X, \omega}^\ast$ is simple.
\end{corollary}

\begin{proof}[{Proof of \cref{Cor:TannakaGroupSimpleSpecialCases}}] For smooth connected subvarieties $X_1,X_2 \subset A$ the image of all Albanese morphisms $X_1 \times X_2 \to \Alb(X_1) \times \Alb(X_2)$ is degenerate provided that $0 < \dim X_1 + \dim X_2 < 2 \dim A$. This shows the statement assuming (1). Now hypothesis (2) implies hypothesis (1): Indeed, $X \subset A$ is nondegenerate by \cref{Thm:PositivityNotions} and nondegeneracy is invariant under isogenies (\cref{Prop:NonDegerateIsogeny}).
\end{proof}

\begin{remark}\label{Rem:complete_intersection}  For a smooth integral subvariety $X \subset A$, condition (2) holds if the normal bundle of $X$ is the direct sum of vector bundles $\cV_1, \dots, \cV_r$ with $r \ge 1$ such that~$\cV_i$ is $d_i$-ample\footnote{A line bundle $\cL$ on an integral variety $X$ is \emph{$d$-ample} if there is an integer $n \ge 1$ such that~$\cL^{\otimes n}$ is globally generated and the fibers of the morphism $X \to \bbP(\rH^0(X, \cL^{\otimes n})^\vee)$ have dimension $\le d$. Ordinary ampleness is equivalent to $0$-ampleness.
A vector bundle $\cE$ on $X$ is \emph{$d$-ample} if the line bundle $\cO(1)$ on $\bbP(\cE^\vee)$ is $d$-ample.} 
and 
\[ \dim X > \max_{i = 1, \dots, r}  \rk \cV_i + d_i , \]
see~\cite[th.~4.5]{Deb95}. By \cref{Thm:PositivityNotions} the ampleness of the normal bundle of a smooth subvariety implies nondegeneracy, thus \cref{Thm:TannakaGroupSimple} applies to any smooth nondivisible subvariety $X\subset A$ such that
\begin{itemize}
\item the normal bundle of $X$ is ample and $\dim X > g/2$, or
\item   $\dim X \ge 2$ and $X$ is a complete intersection of ample divisors in $A$.
\end{itemize}
\end{remark}

\subsection{Product decomposition: from geometry to groups} We start by showing how to obtain a product decomposition of the Lie algebra of the Tannaka group starting from a product decomposition of the subvariety. Keeping the notation of \cref{sec:simplicity-criterion} we write $G_{X, \omega}^\ast := G_\omega^\ast(\delta_X)$ for any fixed fiber functor $\omega\colon \langle \delta_X \rangle \to \Vect(\bbF)$. 

\begin{lemma} \label{Lemma:TensorProductOfMinuscule} Let $G$ be a simple simply connected algebraic group over $\bbF$ and $V$,~$W$ nontrivial irreducible representations of $G$. Then $V \otimes W$ is not minuscule.
\end{lemma}

\begin{proof} The minuscule representations of $G$ are given up to isomorphism by the table on~\cpageref{Table:MinusculeRepresentations}. In particular, the highest weight of any minuscule representation is a fundamental weight. But the highest weight in $V\otimes W$ is the sum of the highest weights of $V$ and $W$, hence it is a sum of two dominant integral weights and therefore cannot be a fundamental weight.
\end{proof}

\begin{proposition} \label{Lemma:TannakaGroupProduct} Let $X_1, X_2 \subset A$ be smooth subvarieties such that the sum morphism $\sigma \colon X_1 \times X_2 \to X:= X_1 + X_2$ is an isomorphism and $X$ is nondivisible. Then we have an isomorphism of Lie algebras
\[  \Lie G^\ast_{X, \omega} \iso \Lie G^\ast_{X_1, \omega} \oplus \Lie G^\ast_{X_2, \omega}.\]
\end{proposition}

\begin{proof} Let $G$ be the universal cover of the derived connected component of the Tannaka group of $\delta_{X_1} \oplus \delta_{X_2}$. The group $G$ decomposes as a product
\[G = G_1 \times \cdots \times  G_n\]
with $G_1, \dots, G_n$ simply connected simple (nontrivial) algebraic groups. Both perverse sheaves $\delta_{X_1}$ and $\delta_{X_2}$ belong to the tensor category $\langle \delta_{X_1} \oplus \delta_{X_2} \rangle$ and therefore define representations of $G$. Since $X_1$, $X_2$ are smooth and nondivisible (otherwise $X$ would not be nondivisible), by \cref{cor:minuscule} such representations are minuscule. By seeing them as  representations of the product $G_1 \times \cdots \times G_n$ they decompose as an external tensor product,
\[  \omega(\delta_{X_i})  =   V_{i,1} \boxtimes \cdots \boxtimes V_{i,n}, \qquad i = 1, 2,\] 
for representations $V_{i,\ell}$ of $G_\ell$. Note that for all $i$, $\ell$ the representation $V_{i, \ell}$ is necessarily minuscule. By hypothesis the sum morphism $X_1 \times X_2 \to X$ is an isomorphism, thus $\delta_X = \delta_{X_1} \ast \delta_{X_2}$ by definition of the convolution product. In particular $\delta_X$ belongs to $\langle \delta_{X_1} \oplus \delta_{X_2} \rangle$ and as representations of $G$ we have
\[  \omega(\delta_X)  =  \omega(\delta_{X_1}) \otimes \omega(\delta_{X_2})  =   (V_{1,1}\otimes V_{2,1}) \boxtimes \cdots \boxtimes (V_{1,n}\otimes V_{2,n}).\]
Again by \cref{cor:minuscule}, the representation $\omega(\delta_X)$ of $G$ is minuscule because $X$ is smooth and nondivisible. Thus  the representation   $V_{1,\ell} \otimes V_{2,\ell}$ of $G_\ell$ is also minuscule for each $1 \le \ell \le n$. Now \cref{Lemma:TensorProductOfMinuscule} implies that for $1 \le \ell \le n$ the representation $V_{i, \ell}$ is trivial for exactly one $i \in \{1, 2\}$. Strictly speaking \cref{Lemma:TensorProductOfMinuscule} gives only the existence of such an $i$; however, if $V_{1, \ell}$ and $V_{2, \ell}$ were both trivial, then $G_\ell$ would act trivially on $\omega(\delta_{X_1}) \oplus \omega(\delta_{X_2})$ contradicting the fact that $G_\ell$ is a nontrivial simple factor of $G$. Resuming the proof, for $i = 1, 2$ let $L_i \subset \{ 1, \dots, n \}$ be the subset made of those $\ell$ for which $V_{i, \ell}$ is nontrivial. Then $L_1, L_2 \subset \{ 1, \dots, n \}$ are complementary subsets and
\begin{align*}
G^\ast_{X_i, \omega} &= \im(\textstyle \prod_{\ell \in L_i} G_\ell  \to \GL(\omega(\delta_{X_i}))), \qquad i = 1, 2,\\
G^\ast_{X, \omega} &= \im(G \to \GL(\omega(\delta_{X}))).
\end{align*}
Rather generally, for a simple simply connected algebraic group $H$ and a nontrivial minuscule representation $W$ of $H$, the kernel of $H \to \GL(W)$ is finite. This gives isomorphisms of Lie algebras
\[ \Lie G^\ast_{X, \omega} \iso \Lie G \iso \bigoplus_{\ell \in L_1} \Lie G_{\ell} \oplus \bigoplus_{\ell \in L_2} \Lie G_{\ell} \iso \Lie G^\ast_{X_1} \oplus \Lie G^\ast_{X_2},\]
as desired.
\end{proof}

\subsection{Conic maps}
In this section, we introduce the notion of conic map, which will turn out useful in dealing with conormal varieties (see \cref{definition conic morphism}). Recall that the \emph{domain of definition} of a rational map is the maximal open subset of its source on which the map is well-defined.

\begin{proposition} \label{Prop:RationalMapBetweenCones} 
Let $X$, $X' \subset A$ be integral subvarieties and $F \colon \PLambda_X \dashto \PLambda_{X'}$ a rational map between their conormal varieties. Then there exists a unique rational map $f \colon X \dashto X'$ such that the following diagram commutes:
\[
\begin{tikzcd}
\PLambda_X \ar[r, dashed, "F"] \ar[d, swap, "\pr_X"] 
& \PLambda_{X'} \ar[d, "\pr_{X'}"]
\\
 X  \ar[r, dashed, "f"] 
& X'
\end{tikzcd}
\]
Furthermore, the domain of definition of $f$ contains the smooth locus $X^\reg \subset X$.
\end{proposition}

\begin{proof} 
Let $U$ be the domain of definition of the rational map $F\colon \PLambda_X \dashrightarrow \PLambda_{X'}$. For any smooth point $x\in X^\reg(k)$ the fiber~$\PLambda_{X, x}$ is a projective space. Every rational map from a projective space to an abelian variety is constant~\cite[cor.~3.9]{MilneAV}, so for $x\in X^\reg(k)$ the morphism 
\[
\pr_{X'} \circ\ F_{\rvert U_x} \colon \quad U_x \;:=\; U\cap \PLambda_{X, x} \;\too \; X'
\;\subset\; A
\] 
must be constant. Therefore, the morphism 
\[ \pr_{X'} \circ\ F_{\rvert V} \colon \quad V \;:=\; U \cap \pr_X^{-1}(X^\reg) \;\too \; X'
\;\subset\; A \]
is constant along the fibers of the smooth morphism $\pr_X \colon V \to X^\reg$. Over the open subset $\pr_X(V) \subset X^\reg$ the morphism $\pr_X$ locally has sections, so we have
\[ \pr_{X'} \circ F = f \circ \pr_X \]
for a unique $f \colon \pr_X(V) \to X'$. The latter extends to a morphism $f\colon X^\reg \to X'$ because a rational map from a variety to an abelian variety is defined at every smooth point of the source~\cite[th.~3.1]{MilneAV}.
\end{proof}

In the above proof, we have not used anything specific about conormal varieties. In fact, the only thing we used was that $\pr_X\colon \PLambda_X \to X$ is a projective bundle over $X^\reg\subset X$ and that $X'$ embeds in an abelian variety. However, the conormal geometry will be taken into account by the following notion of a conic map:

\begin{definition} \label{definition conic morphism}
In the setup of proposition~\ref{Prop:RationalMapBetweenCones} we call the rational map $f \colon X \dashto X'$ the \emph{base} of $F\colon \PLambda_X \dashto \PLambda_{X'}$. We say that a rational map $F\colon \PLambda_X \dashto \PLambda_{X'}$ is~\emph{conic} if the diagram
\[
\begin{tikzcd}
 \PLambda_X \ar[r, "\gamma_X"] \ar[d, swap, dashed, "F"]& \bbP_A \ar[d, equal] \\
\PLambda_{X'} \ar[r, "\gamma_{X'}"] & \bbP_A.
\end{tikzcd}
\]
commutes, i.e.,~if $F$ is compatible with the respective Gauss maps. 
\end{definition}

\begin{example} \label{Ex:GaussDegreeBirational} Let $X$, $X' \subset A$ be integral subvarieties and $F \colon \PLambda_X \dashto \PLambda_{X'}$ a conic map whose base is birational. Then
\[ \deg \PLambda_X \;=\; \deg \PLambda_{X'}\]
because the Gauss degree can be computed over any nonempty open subset of $\bbP_A$.
\end{example}

Note that even when the base $f\colon X\dashto X'$ of a conic map is defined everywhere, it is still not clear whether it is the restriction of an endomorphism of the abelian variety $A$. However, the results about conic maps in the rest of this section will suffice for our purpose:

\begin{proposition} \label{Prop:BasicPropertiesConicMorphism} 
Let $X$, $X'\subset A$ be integral subvarieties and $F \colon \PLambda_X \dashto \PLambda_{X'}$ a conic map. If the algebraic group $\Stab(X)$ is finite, then the rational map $F$ is dominant and generically finite.
\end{proposition}

\begin{proof} By \cref{Thm:PositivityNotions} (1), the Gauss map $\gamma_X$ is generically finite. It then follows from the commutative diagram in \cref{definition conic morphism} that the rational map $F$ is also generically finite. Since $\dim(\PLambda_X)=\dim(\PLambda_{X'})=g-1=\dim(\bbP_A)$, it follows that~$F$ is also dominant, being a generically finite map between varieties of the same dimension.
\end{proof}

\subsection{Product decompositions: from groups to geometry} 
We now explain how to obtain from a product decomposition for the Lie algebra of the Tannaka group a product decomposition for conormal varieties, using the above results about conic maps. Borrowing notation from \cref{sec:simplicity-criterion} we write $G_{X, \omega}^\ast := G_\omega^\ast(\delta_X)$ for any fixed fiber functor $\omega\colon \langle \delta_X \rangle \to \Vect(\bbF)$.

\begin{proposition} \label{Prop:WhatNonSimplicityImplies} 
Assume $X\subset A$ is a smooth nondivisible subvariety and $G_{X, \omega}^\ast$ is not simple. Then for $i=1,2$ there are integral subvarieties $X_i \subset A$ with $\PLambda_{X_i} \in \langle \PLambda_X\rangle$ of Gauss degree $\deg(\PLambda_{X_i})>1$, conic maps $F_i \colon \PLambda_X \dashto \PLambda_{X_i}$ and integer $n \ge 1$ with the following properties:\smallskip 
\begin{enumerate}
    \item We have an identity of cycles $[n]_\ast \PLambda_X = \PLambda_{[n](X)} = \PLambda_{X_1} \circ \PLambda_{X_2}$;\smallskip 
            \item The following square is commutative
    \[
    \begin{tikzcd} 
    \PLambda_X  \ar[d, dashed, "F_1 \times F_2"'] \ar[r, "{[n]}"] & \PLambda_{ [n] (X)} \ar[d, equal] \\
    \PLambda \ar[r, "\sigma"]& \PLambda_{X_1} \circ \PLambda_{X_2}
    \end{tikzcd}
    \]
    where $ \PLambda$ and $\sigma$ are as in \cref{defn:convolution-of-clean-cycles}; \smallskip

    \item The rational maps $F_i\colon \PLambda_X \dashrightarrow \PLambda_{X_i}$ are dominant and generically finite;\smallskip 

    \item  The base of $F_i$ is a morphism $f_i\colon X\to X_i$ which is surjective.
\end{enumerate}
\end{proposition}

\begin{proof} 
By \cref{cor:algebraically-closed} and \cref{lem:reduction-to-complex-case} we may assume $k=\bbC$, which will allow us later to use the results about characteristic cycles in \cref{thm:cc-and-weights}.
By \cref{cor:derived-group} the group $G_{X, \omega}^\ast$ does not change if we replace~$X$ by $X+a$ for any $a\in A(k)$, and it clearly suffices to achieve properties (1)-(4) for any such translate. We will therefore assume  
\[ \det(\delta_X) \;=\; \delta_0 \]
so that the connected component $G:=G_{X, \omega}^\circ$ is semisimple by \cref{prop:connected-component}. If the group~$G_{X, \omega}^\ast$ is not simple modulo its center, then by the structure theory of semisimple groups there are simply connected semisimple groups $G_1,G_2\not \simeq \{1\}$ and an isogeny 
\[
\begin{tikzcd}[column sep=15pt] p\colon \hspace{-11pt} & \tilde{G} \;:=\; G_1 \times G_2 \ar[r,twoheadrightarrow] & G \end{tikzcd}
\]
Then $V:=\omega(\delta_X)$ restricts to an irreducible representation of the covering group $\tilde{G}$ and as such it decomposes as
\[
 V_{\rvert\tilde{G}} \;\simeq\; V_1 \otimes V_2 
 \quad \textnormal{with irreducible} \quad 
 V_i \;\in\; \Rep_\bbF(G_i) \;\subset\; \Rep_\bbF(\tilde{G}).
\]
Note that both factors $G_1$ and $G_2$ must act nontrivially on $V$ since otherwise they would not appear in the Tannaka group. In particular, we have $\dim V_i \ge 2$ as any one-dimensional representation of a connected semisimple group is trivial.

\medskip 

Let $n=m\cdot \deg(p)$ for the smallest integer $m \ge 1$ with $m\cdot \Gamma_{\cc(\delta_X)} = \{0\}$. Since $X$ is smooth, its perverse intersection complex has characteristic cycle $\cc(\delta_X)=\PLambda_X$. Via the first part of \cref{thm:cc-and-weights}, the above decomposition as a tensor product of two representations of the universal covering group yields clean cycles $\PLambda_1,\PLambda_2\in \langle \PLambda_X\rangle$ such that
\[ [n]_\ast \PLambda_X = \PLambda_1 \circ \PLambda_2,\]
and the second part of the theorem implies that $\deg \PLambda_i = \dim V_i\ge 2$.  
Moreover, by \cref{lem:weight_multiplicities} the cycles $\PLambda_i$ are effective. Since we assumed the subvariety $X\subset A$ to be nondivisible, the morphism $[n] \colon X \to {[n](X)}$ is dominant and birational, in fact we have 
$
 [n]_\ast \PLambda_X = \PLambda_{[n](X)} 
$
as an identity of cycles by \cref{lem:image-of-conormal-under-isogeny}. 
Altogether, it follows that
\[ \PLambda_{[n](X)} \;=\; [n]_* \PLambda_X  \;=\; \PLambda_1 \circ \PLambda_2. \]
The cycle on the left-hand side is reduced and irreducible, so the same must hold for both factors on the right-hand side because the convolution product $\circ$ on cycles is bilinear and the convolution of any two clean effective cycles is again a clean effective cycle. Hence, there exist integral subvarieties $X_i\subset A$ with $\PLambda_i = \PLambda_{X_i}$. By definition of $\circ$, we have
\[
 \PLambda_1 \circ \PLambda_2 \;=\; \sigma_*(\overline{\PLambda_{X_1\mid U} \times_U\PLambda_{X_2\mid U}})
\]
where $\sigma\colon A\times A \times \bbP_A \to A\times \bbP_A$ denotes the sum morphism and $U\subset \bbP_A$ is as in \cref{defn:convolution-of-clean-cycles}. The multiplicities of the cycle-theoretic pushforward on the right-hand side are given by the degree of the sum morphism 
\[
 \sigma\colon \quad \PLambda \;:=\; \overline{\PLambda_{X_1\mid U} \times_U\PLambda_{X_2\mid U}} \;\too\; \PLambda_{[n](X)}
\]
on the various components of its source. Since the cycle $\PLambda_1\circ \PLambda_2 = \PLambda_{[n](X)}$ is integral as observed above, it follows in fact that the fiber product $\PLambda$ is integral and is mapped birationally onto its image by $\sigma$.
Consider then the composition of rational maps 
\[
F_i \colon 
\begin{tikzcd} 
\PLambda_X \ar[r, "{[n]\times \id}"]
& \PLambda_{[n](X)} 
\ar[r, dashed, "\sigma^{-1}"] 
& \PLambda \;=\; \overline{\PLambda_{X_1\rvert U}\times_U \PLambda_{X_2\rvert U}}
\ar[r, "\pr_i"]
& \PLambda_{X_i}
\end{tikzcd}
\] 
where $\pr_i$ denotes be the projection onto the $i$-th factor.
By construction, $F_i$ is a conic map, and by \cref{Prop:RationalMapBetweenCones} its base $f_i \colon X \dashto X_i$ is defined on all of $X$ because we assumed $X$ to be smooth. Moreover, by \cref{Prop:BasicPropertiesConicMorphism} the rational map $F_i$ is dominant and generically finite, so the morphism $f_i$ is surjective.
\end{proof}

\subsection{Proof of \cref{Thm:TannakaGroupSimple}} We can now prove the simplicity criterion for Tannaka groups as follows. Thanks to \cref{Lemma:TannakaGroupProduct} only the implication (1) $\implies$ (2) is left to be shown. Suppose that the algebraic group $G_{X, \omega}^\ast$ is not simple. According to \cite[Lemma 4.6]{LS20}, this is never the case when $X$ is a smooth ample divisor, thus from now we may assume $\dim X < g -1$. Let $n$, $X_i$, $F_i$ and~$f_i\colon X\to X_i$ be as in \cref{Prop:WhatNonSimplicityImplies}. 
For $i = 1, 2$, the Gauss map of $X_i$ is a finite morphism because the one of $X$ is by \cref{Thm:PositivityNotions} and $\PLambda_{X_i} \in \langle \PLambda_X \rangle$ by construction. It is therefore possible to apply \cref{cor:DimensionOfConvolutionViaSegre} and deduce the equality
\[ \dim X = \dim X_1 + \dim X_2\]
from the identity of cycles $\PLambda_{X_1} \circ \PLambda_{X_2} = \PLambda_{[n](X)}$ and the hypothesis $\dim X < g - 1$. Moreover, the subvarieties $X_1$ and $X_2$ are nondegenerate by \cref{Thm:PositivityNotions}. Thus, \cref{Lem:SumOfNondegenerate} implies that the sum 
\[ X_1 + X_2 \;\subset\; A \]
is a nondegenerate subvariety in $A$ and the sum morphism $\sigma \colon X_1 \times X_2 \to X_1 + X_2$ is generically finite. Therefore, the conormal cone $\PLambda_{X_1 + X_2}$ appears as a summand with multiplicity $\deg(\sigma)$ in the cycle $\PLambda_{X_1} \circ \PLambda_{X_2}$. But we know that 
$\PLambda_{X_1} \circ \PLambda_{X_2} = \PLambda_{[n](X)}$, hence  
\[
 \deg(\sigma) \;=\; 1 
 \quad \textnormal{and} \quad 
 X_1 + X_2 \;=\; [n](X).
\] 
Together with \cref{Prop:WhatNonSimplicityImplies} (2) this gives the following commutative square:
\[ 
\begin{tikzcd}
X \ar[d, "f_1 \times f_2"'] \ar[r, "{[n]}"] & \left[ n \right] (X) \ar[d, equal]\\
X_1 \times X_2 \ar[r, "\sigma"] & X_1 + X_2
\end{tikzcd}
\]
The nondivisibility assumption on $X$ implies that the morphism $[n] \colon X \to [n](X)$ is finite birational. This forces $f = (f_1 \times f_2)$ to be finite birational. As $X$ is smooth, this says that $f$ is the normalization morphism. For $i = 1, 2$ let $\tilde{X}_i$ be the normalization of $X_i$. The morphism $\tilde{f} \colon X \to \tilde{X}_1 \times \tilde{X}_2$ induced by $f$ is an isomorphism, thus $\tilde{X}_1$ and $\tilde{X}_2$ are smooth. Identifying $X$ with $\tilde{X}_1 \times \tilde{X}_2$ permits to embed $\tilde{X}_1$ and $\tilde{X}_2$ in $A$ and to write $\Alb(X) = \Alb(\tilde{X}_1) \times \Alb(\tilde{X_2})$. By suitably embedding $X$, $\tilde{X}_1$, and $\tilde{X}_2$ in their Albanese variety, we have the following identity:
\[ (\tilde{X}_1 \times \{ 0 \}) + (\{ 0 \} \times \tilde{X}_2) = X \subset \Alb(X). \]
The commutativity of the following square
\[ 
\begin{tikzcd}
(\Alb(\tilde{X}_1) \times \{ 0 \}) \times ( \{ 0 \} \times \Alb(\tilde{X}_2))  \ar[r] \ar[d]& \Alb(X) \ar[d] \\
A \times A \ar[r] & A 
\end{tikzcd}
\]
where the horizontal arrows are the sum morphisms and the vertical ones are given by universal property of the Albanese, implies the desired equality $X = \tilde{X}_1 + \tilde{X}_2$.
\qed

\section{Wedge powers} \label{sec:wedge}

We now characterize low-dimensional smooth subvarieties whose Tannaka group is the image of a special linear group acting on a nontrivial wedge power of its standard representation. In the case of hypersurfaces, Lawrence and Sawin show that such groups do not occur~\cite[lemma~4.10]{LS20} using combinatorial properties of Eulerian numbers. In higher codimension, wedge powers do occur, but we show by geometric arguments that they only arise from symmetric powers of curves. 

\subsection{Statement of the main result}

As in the previous section, we assume the field $k$ to be algebraically closed of characteristic zero. Fix a subvariety $X\subset A$ and a fiber functor
$\omega\colon \langle \delta_X \rangle \to \Vect(\bbF)$ on the Tannaka category generated by the perverse intersection complex of $X$. As in \cref{sec:simplicity-criterion}, put $G_{X, \omega} := G_\omega(\delta_X)$ and denote by
\[
 G_{X, \omega}^\ast \;:=\; [G^\circ_{X, \omega}, G^\circ_{X, \omega}]
\] 
the derived group of its connected component of the identity. We are interested in the following situation:

\begin{definition} \label{Def:WedgePower} 
Let $r \ge 1$ be an integer. We say $X\subset A$ is an  \emph{$r$-th wedge power} if we have
\[
 G_{X, \omega}^\ast 
 \simeq \Alt^r(\SL_n(\bbF))
 \quad \textnormal{with the standard action on} \quad
 \omega(\delta_X)
 \simeq \Alt^r(\bbF^n)
\]
for some $n \ge 1$. Notice that if $r$ is given, then $n$ is determined by the topological Euler characteristic
\[
\chi(\delta_X) \;=\; \dim_\bbF(\omega(\delta_X)) \;=\; \tbinom{n}{r}.
\]
If $X$ is an $r$-th wedge power, then by duality it is also an $(n-r)$-th wedge power. 
\end{definition} 

The typical example of wedge powers arises from symmetric powers of curves, as announced in the introduction:

\begin{lemma}  \label{ex:symmetric-power-of-curve}
Let $C\subset A$ be a smooth projective curve and $r\ge 2$ an integer. If the sum morphism 
\[ s\colon \quad \Sym^r C \;\too\; X \;=\; C + \cdots + C \;\subset\; A\]
is an isomorphism onto its image, then this image $X\subset A$ is an $r$-th wedge power.
\end{lemma} 

\begin{proof}
Replacing $A$ by the abelian subvariety $\langle C\rangle \subset A$, we may assume that $C$ generates $A$. Fix a point $p\in C(k)$. The morphism $\alb_{C, p}\colon C \into \Alb(C)$ induces a sum morphism
\[
\Sym^r(C) \;\too\; \Alb(C). 
\] 
Let $W_r \subset \Alb(C)$ denote the image of this sum morphism. The embedding $C\into A$ also induces a morphism $\Alb(C)\to A$; this last morphism is surjective because we assumed $A=\langle C \rangle$.
We obtain the following commutative diagram:
\[
\begin{tikzcd} 
\Sym^r(C) \ar[r, two heads] \ar[d, equals]
& W_r \ar[r, hook] \ar[d, two heads]
& \Alb(C) \ar[d, two heads] 
\\
\Sym^r(C) \ar[r, "\sim"]
& X \ar[r, hook] 
& A
\end{tikzcd} 
\] 
The square on the left implies that the morphism $W_r \to X$ is an isomorphism, therefore~$W_r$ is smooth (hence by the Riemann singularity theorem the curve~$C$ is not $r$-gonal, so in particular it is not hyperelliptic if $r\ge 2$).

\medskip

Fixing any $r-2\ge 0$ points on the curve and varying the remaining two points, one obtains from the above also that the morphisms $\Sym^2(C)\to W_2 \subset \Alb(C)$ and $\Sym^2(C) \to C+C\subset A$ have the same fibers. So we can apply the variant of Larsen's alternative in~\cite[section~6]{KWSmall} to the perverse sheaf $\delta_C \in \Perv(A, \bbF)$ to see that
\[
G^\ast_\omega(\delta_C) \simeq  \SL_n(\bbF)
\quad \textnormal{with the standard action on}
\quad 
\omega(\delta_C) \simeq \bbF^n 
\quad \textnormal{for} \quad
n = \chi(\delta_C).
\] 
The above diagram then shows
$G_{X, \omega}^\ast \simeq \Alt^r(\SL_n(\bbF))$ 
and
$\omega(\delta_X) \simeq \Alt^r(\bbF^n)$. 
\end{proof} 

The goal of this section is to show a converse to the above lemma. More precisely, we obtain the following complete classification of wedge powers for all nondivisible smooth subvarieties $X\subset A$ of high codimension whose Gauss map $\gamma_X\colon \PLambda_X \to \bbP_A$ is finite:

\begin{theorem} \label{Thm:SmallWedgePowersAreSumsOfCurves} Let $X\subset A$ be a nondivisible smooth subvariety with ample normal bundle, and suppose that its Euler characteristic is $\chi(\delta_X)=\tbinom{n}{r}$ for some integer $r$ with $1<r\le n/2$.  If $2 \dim X < \dim A - 1$, then the following are equivalent:\smallskip 
\begin{enumerate} 
\item The subvariety $X\subset A$ is an $r$-th wedge power.\smallskip 
\item There is a nondegenerate irreducible smooth projective curve $C\subset A$ such that\smallskip 
\begin{itemize}
\item $X=C+\cdots + C \subset A$ is the sum of $r$ copies of $C$, and\smallskip 
\item the sum morphism $\Sym^r C \to X$ is an isomorphism.
\end{itemize}
\end{enumerate} 
\end{theorem}

In view of \cref{ex:symmetric-power-of-curve} we only need to show the implication $(1)\Longrightarrow (2)$, which will take up the rest of this section.

\subsection{Structure of the proof}

The proof relies on three independent steps. The first step is to show that the structure of wedge powers is reflected by characteristic cycles: To any subvariety $Z\subset A$ and an integer $r \ge 1$ we will attach a clean effective cycle 
\[ \Alt^r \PLambda_Z \;\in\; \cL(A), \]
and we show:

\begin{theorem} \label{Thm:FromRepresentationToPhysicalWedgePowers} 
Let $X\subset A$ be a smooth nondivisible subvariety with $\dim X > 0$ that is an $r$-th wedge power for some integer $r \ge 1$. Then there is an integral subvariety $Z \subset A$ with $\PLambda_Z \in \langle \PLambda_X \rangle$ such that
\[ \Alt^r \PLambda_{Z} \;=\; \PLambda_{[e](X)} \]
for some integer $e \ge 1$. The Gauss degrees are related by
\[ 
 \deg \PLambda_X = \tbinom{n}{r} \quad \textnormal{\em where} \quad n=\deg \PLambda_Z,
\]
and if the Gauss map $\gamma_X\colon \PLambda_X \to \bbP_A$ is a finite morphism, then so is $\gamma_Z\colon \PLambda_Z \to \bbP_A$.
\end{theorem}

For the construction of the clean cycle $\Alt^r \PLambda_Z$ and the proof of the above result, see \cref{sec:cc-of-wedge-powers}. Once we have this, the second step in our classification of wedge powers will be to prove a monotonicity statement for the cycles $\Alt^r(\PLambda_Z)$ as a function of $r$. Let 
\[
 \Alt^r Z \;:=\; \im\Bigl(\pr_A\colon \Supp(\Alt^r \PLambda_Z) \to A \Bigr)
 \;\subset\; A 
\]
be the image of the support of the clean effective cycle $\Alt^r \PLambda_Z \in \cL(A)$ under the projection to the abelian variety, then we will show:

\begin{theorem} \label{Thm:SmallWedgePowersHaveSmallDegree} Let $Z\subset A$ be an integral subvariety whose Gauss map is a finite morphism of degree $n=\deg(\gamma_Z)$. Suppose there exists an integer $r \ge 1$ with $r \le n/2$ such that
\[ \dim \Alt^r Z < (\dim A-1)/2.
\]
Then we have
\[ r \dim Z < \dim A.\]
\end{theorem}

The proof of this will be given in \cref{sec:dimension-estimate}. Finally, the last step for our classification of wedge powers will be to show that in the given dimension range, the smoothness of $\Alt^r \PLambda_Z$ forces $Z$ to be a curve. More precisely:

\begin{theorem} \label{Thm:WedgePowersOfSmallDegrees} 
Let $X,Z\subset A$ be integral subvarieties whose Gauss maps are finite morphisms. Suppose moreover that $X$ is smooth and nondivisible and that there are integers $e,r \ge 1$ such that 
\[ \Alt^r \PLambda_{Z} = \PLambda_{[e](X)}.\]
If $r \dim Z < \dim A$, then $Z$ is a curve and \smallskip 
\begin{enumerate}
	\item the normalization $C$ of $Z$ embeds in $A$,\smallskip
	\item $X=C+\cdots + C\subset A$ is the sum of $r$ copies of $C$, embedded suitably in $A$,\smallskip 
	\item the sum morphism $\Sym^r C \to X$ is an isomorphism.
\end{enumerate}
\end{theorem}

We will prove this in \cref{sec:smooth-wedge-powers}. Before coming to the details, let us note how the above three results conclude the classification of wedge powers:

\begin{proof}[{Proof of \cref{Thm:SmallWedgePowersAreSumsOfCurves}}]
Suppose that $X\subset A$ satisfies the assumptions of the theorem and that it is an $r$-th wedge power for some $r\in \bbN$ with $1<r\le n/2$ where $\deg \gamma_X= \tbinom{n}{r}$. By \cref{Thm:FromRepresentationToPhysicalWedgePowers} then
\[ \Alt^r \PLambda_{Z} = \PLambda_{[e](X)}. \]
for some integer $e \ge 1$ and some integral subvariety $Z\subset A$. 
Since $[e]\colon A\to A$ is an isogeny, we have 
$ \dim [e](X) = \dim X < (\dim A - 1)/2$
by our dimension assumption on $X$. \Cref{Thm:SmallWedgePowersHaveSmallDegree} then shows 
\[ r \dim Z  < \dim A, \]
and hence \cref{Thm:WedgePowersOfSmallDegrees} gives the desired result.
\end{proof}                                                                                                                 

\subsection{Characteristic cycles of wedge powers} \label{sec:cc-of-wedge-powers}

We now explain how to compute characteristic cycles of wedge powers. 
Let $Z\subset A$ be a subvariety with dominant Gauss map $\gamma_Z\colon \PLambda_Z \to \bbP_A$. For an integer $r\ge 1$ we consider as in \cref{defn:fiber-product-minus-diagonal} the Zariski closure
\[
 \PLambda_{Z}^{[r]} \;:=\; 
 \overline{\PLambda^{\times r}_{Z\mid U} \smallsetminus \Delta_r } 
 \;\subset\; A^r \times \bbP_A
\]
where $U\subset \bbP_A$ is any open dense subset over which the Gauss map $\gamma_Z$ is finite and flat.
Let $\sigma\colon A^r \times \bbP_A \to A\times \bbP_A, (z_1, \dots, z_r, \xi) \mapsto (z_1+\cdots + z_r, \xi)$ be the sum morphism and put 
\[
 \Alt^r \PLambda_Z \;:=\; \tfrac{1}{r!} \, \sigma_*(\PLambda_{Z}^{[r]})
 \;\in\; \cL(A),
\]
which is the special case $\alpha = (1,\dots, 1) = (1^r)$ of the cycle in \cref{defn:fiber-product-minus-diagonal}.

\begin{proof}[Proof of \cref{Thm:FromRepresentationToPhysicalWedgePowers}] By \cref{cor:algebraically-closed} and \cref{lem:reduction-to-complex-case} we may assume that $k=\bbC$ and $\bbF = \bbC$ so that we can use the results about characteristic cycles from \cref{section from rep to geo}. Replacing~$X$ by a translate we may assume by \cref{cor:derived-group} that the connected component of the group $G_\omega(\delta_X)$ is semisimple. Its universal cover is then isomorphic to $\tilde{G}\simeq \SL_n(\bbC)$ for some $n \ge 1$, so the setup in \cref{subsec:typeA} applies to $P=\delta_X$. By assumption
\[
 \omega(P) \;\simeq\; \Alt^r(\bbC^n).
\]
Since $\dim \omega(P)>1$ for any non-negligible perverse sheaf $P$ which is not a skyscraper sheaf~\cite{WeissauerAlmostConnected}, we have $1<r<n$. The highest weight of the above wedge power representation of $\SL_n(\bbC)$ is the fundamental weight $\alpha = \varepsilon_1 + \cdots + \varepsilon_r$. With notation as in \cref{defn:cc-lambda}, the Weyl group orbit $W.\alpha$ consists precisely of the weights in the representation $\xi(P^\circ)$, so $\xi^{-1} [W.\alpha] = [P^\circ]$ in the Grothendieck ring $K(\langle P^\circ \rangle)$. It follows that 
\begin{align*}
 \cc(P, \alpha) 
 &\;=\; \cc(\xi^{-1}[W.d\alpha]) && \textnormal{with notation as in \cref{defn:cc-lambda}} \\
 &\;=\; [d]_* \cc(P^\circ) && \textnormal{since $\xi^{-1}[W.\alpha] = [P^\circ]$} \\
 &\;=\; [e]_* \cc(P) && \textnormal{for $e:=dm$ and $P^\circ := [m]_\ast P$} \\
 &\;=\; [e]_* \PLambda_X && \textnormal{since $\cc(\delta_X) = \PLambda_X$ for smooth $X$} \\
 &\;=\; \PLambda_{[e](X)} && \textnormal{by \cref{lem:image-of-conormal-under-isogeny}, since $X$ is nondivisible.}
\end{align*}	
In particular, the cycle $\cc(P, \alpha)$ is integral. So by \cref{lem:cc-in-type-A} the cycle $\cc(P, \varepsilon_1)$ is integral as well, hence of the form $\cc(P, \varepsilon_1) = \PLambda_Z$ for an integral subvariety $Z\subset A$, and we have
\[
 \PLambda_{[e](X)} \;=\; \cc(P, \alpha) \;=\; \PLambda_Z^\alpha \;=\; \Alt^r \PLambda_Z.
\] 
Then $\deg \PLambda_X = \deg \cc(P, \alpha) = \dim \omega(P) = \tbinom{n}{r}$. For the statement about the finiteness of the Gauss map, recall from \cref{Rmk:CharCycleWeightDefinedOnSubfield} that the cycle $\PLambda_Z=\cc(P, \epsilon_1)$ lies in the subring $\langle \PLambda_X \rangle \subset \cL(A)$. As such, it appears in some convolution power of the cycle $\PLambda_X$. Recalling the definition of the convolution product, this implies that if the Gauss map for $\gamma_X$ is finite, then so is the one for $\gamma_Z$.  
\end{proof}

\begin{remark} \label{rem:gauss-degree-of-complement-of-diagonal} 
The discussion in the proof of \cref{Thm:FromRepresentationToPhysicalWedgePowers} shows that for any reduced subvariety $Z\subset A$ whose Gauss map is dominant of degree $n=\deg(\gamma_Z)$ and any integer $r \ge 1$ we have:\smallskip 
\begin{itemize}
\item The projection $\gamma_{Z,r}\colon \PLambda_{Z}^{[r]} \to \bbP_A$ is dominant (and hence generically finite) if and only if $r\le n$. In that case 
\[ \deg \gamma_{Z,r} \;=\; r! \tbinom{n}{r}. \] 
\item If $\gamma_Z\colon \PLambda_Z \to \bbP_A$ is a finite morphism, then so is $\gamma_{Z, r}\colon \PLambda_{Z}^{[r]} \to \bbP_A$.\smallskip
\end{itemize} 
\end{remark}

\subsection{Dimension estimates for wedge powers} 
\label{sec:dimension-estimate}

We will deduce \cref{Thm:SmallWedgePowersHaveSmallDegree} from a monotonicity property of wedge powers. To formulate this, recall that for $Z\subset A$ and an integer $i \ge 1$ we put
$
 \Alt^i Z := \pr_A(\Supp(\Alt^i \PLambda_Z)) \subset A$.
We are interested in its dimension
\[
 d_Z(i) \;:=\; \dim \Alt^i Z.
\]
\begin{example} \label{ex:wedge-powers-of-a-curve}
Let $Z$ be a smooth projective curve of genus $g>1$, embedded via a suitable translate of the Abel-Jacobi map in its Jacobian variety $A=\Alb(Z)$. Then by~\cite[examples~3.1 and 4.1]{KraemerThetaSummands},
\[
 d_Z(i) \;=\; 
 \begin{cases} 
 i & \textnormal{for $1\le i\le g-1$}, \\
 g-1-i & \textnormal{for $g-1\le i \le 2g-2$}.
 \end{cases} 
\]
Here $d_Z$ is not monotonous, but it is so between zero and $g-1=\deg \gamma_Z/2$. 
\end{example}

The following monotonicity result shows that the behavior in the above example is typical for wedge powers in small degrees: 

\begin{proposition} \label{prop:linear-growth-for-small-wedge-powers}
	Let $Z\subset A$ be an integral nondegenerate subvariety, and let $r\ge 1$ be an integer such that $r\dim Z < \dim A$. 
	\begin{enumerate} 
	\item Let $Y:=Z+\cdots + Z \subset A$ be the sum of $r$ copies of $Z$. Then the sum morphism 
	\[
	\sigma\colon \quad \Sym^r Z \;=\; Z^r/\frS_r \;\too\; Y 
	\]
	is generically finite, and the subvariety $Y\subset A$ is nondegenerate. 
	\item The clean cycle $\Alt^r \PLambda_Z$ contains $\PLambda_{Y}$ with multiplicity $d=\deg(\sigma) \ge 1$, and we have
	\[
		\Alt^r Z \;=\; Y.
	\]
	\item For general $z = (z_1, \dots, z_r)\in Z^r(k)$, the fiber of $\PLambda_Y$ over the point $y=\sigma(z)$ is
	\[
	 \PLambda_{Y, y}\;=\; \PLambda_{Z, z_1} \cap \cdots \cap \PLambda_{Z, z_r} 
	 \;\subset\; \bbP_A.
	\]
	\end{enumerate} 
\end{proposition}

\begin{proof} 
The image $Y=Z+\cdots + Z$ of the sum morphism $\sigma \colon \Sym^r Z \to Y$ is a proper subvariety of $A$ because $\dim Y \le \dim \Sym^r Z = r\dim Z < \dim A$. As we assumed~$Z\subset A$ to be nondegenerate, it then follows by~\cref{Lem:SumOfNondegenerate} that the morphism $\sigma$ is generically finite and that $Y\subset A$ is again nondegenerate.

\medskip 

This last property implies that the Gauss map $\gamma_Y\colon \PLambda_Y \to \bbP_A$ is generically finite and dominant, see \cref{Thm:PositivityNotions}. Let $W \subset Y^\reg$ be a nonempty open subset such that $V:= \sigma^{-1}(W)$ is contained in $(Z^\reg)^r \smallsetminus \Delta_r$ and the sum morphism $\sigma \colon V \to W$ is finite \'etale. If we view the tangent spaces as subspaces of $\Lie(A)$, then for every point $z = (z_1, \dots, z_r) \in V(k)$ and $y=\sigma(z)$ the tangent map
\[ T_z(\sigma)\colon \quad T_{z_1}(Z) \times \cdots \times T_{z_r}(Z) \;\too\; T_y(Y)
\]
is the restriction of the sum map $\Lie(A)^r \to \Lie(A)$. In particular,  $T_{z_i}(Z)\subset T_y(Y)$ and hence $\PLambda_{Y, y} \subset \PLambda_{Z, z_i}$ for all $i$. 

\medskip 

Let $U \subset \bbP_A$ be a nonempty open subset over which the Gauss maps $\gamma_Y$ and $\gamma_Z$ are finite \'etale, and such that $\pr_Y(\gamma_Y^{-1}(U))\subset W$. For any $\xi \in U(k)$ and $z \in V(k)$ with $\sigma(z) \in \gamma_Y^{-1}(\xi)$, we have 
\[
 (z_1, \dots, z_r, \xi) \;\in\; \PLambda_{Z\vert U}^{\times r} \smallsetminus \Delta_r \;\subset\; \PLambda_{Z}^{[r]}
\]
So the image of $\pr_{Z,r}\colon \PLambda_{Z}^{[r]} \to Z^r$ contains an open dense subset of~$Z^r$. Hence by properness, the morphism $\pr_{Z,r}$ is surjective and then $\Alt^r Z = Y$ by definition. In particular, $\dim \Alt^r Z = \dim Y = r\dim Z$, where the last equality follows from the generic finiteness of the sum morphism $\sigma \colon \Sym^r Z \to Y$. 
The statement about the general fiber of $\pr_Y\colon \Alt^r \PLambda_Z\to Y$ follows from the fact that for general $(z_1, \dots, z_r)$ we have
\[
 T_y(Y) \;=\; T_{z_1}(Z) \oplus \cdots \oplus T_{z_r}(Z)
\]
because the summands on the right-hand side span $T_y(Y)$ and their dimension adds up to $r\dim Z = \dim Y = \dim T_y(Y)$; passing to the corresponding normal spaces gives $\PLambda_{Y, y} = \PLambda_{Z, z_1} \cap \cdots \cap \PLambda_{Z, z_r}$ and the claim follows since $\Supp(\Alt^r \PLambda_Z)_y = \PLambda_{Y, y}$ for $y\in Y(k)$ general. Finally, it is also clear from the above discussion that the clean cycle $\Alt^r \PLambda_Z$ contains the component $\PLambda_Y$ with multiplicity $d=\deg(\sigma)$.
\end{proof} 

\begin{corollary} \label{cor:gauss-degree-bound}
	Let $Z\subset A$ be an integral nondegenerate subvariety of positive dimension. Then 
	\[
	\deg(\gamma_Z) \;\ge\; 2 \frac{ \dim A}{\dim Z} - 2.
	\]	
\end{corollary} 

\begin{proof} 
	\Cref{prop:linear-growth-for-small-wedge-powers} shows that the map $r\mapsto d_Z(r) := \dim \Alt^r Z$ is strictly increasing on the interval $\{0,1,\dots, r_0\}$, where $r_0=\max\{ r\mid r\dim Z < \dim A\}$. On the other hand, we have
	\[
	\Alt^r \PLambda_Z \;=\; \Alt^{n-r} \PLambda_{p-Z}
	\]
	where $n=\deg(\gamma_Z)$ and where the point $p\in A(k)$ is defined by $\Alt^n \PLambda_Z = \PLambda_{\{p\}}$. It follows that
	\[
	\dim \Alt^r Z \;=\; \dim \Alt^{n-r} Z. 
	\]
	The left-hand side is strictly increasing for $r\in \{0,1,\dots, r_0\}$, the right-hand side is strictly decreasing for $r\in \{n-r_0, \dots, n\}$. 
	This forces $r_0 \le n-r_0$, hence $n\ge 2r_0$. Then $(r_0 + 1)\dim Z \ge \dim A$ implies
	$
	(n/2 + 1) \cdot \dim Z \ge \dim A
	$
	as desired.
\end{proof} 

For smooth subvarieties $Z\subset A$, the degree of the Gauss map $\gamma_Z\colon \PLambda_Z \to \bbP_A$ is equal to the topological Euler characteristic of $\delta_Z$, see \cref{subsec:clean-cc}. So for smooth curves $Z\subset A$ with $\langle Z \rangle = A$, \cref{cor:gauss-degree-bound} says that the curve has genus $\ge \dim A$; this follows of course also directly from the fact that for such curves the morphism $\Alb(Z)\to A$ must be surjective. Note that here the bound in \cref{cor:gauss-degree-bound} is sharp.

\medskip 

\Cref{prop:linear-growth-for-small-wedge-powers} gives a monotonicity statement for wedge powers, but in order to apply it we need to have an a priori bound on $\dim Z$. The following result will allow us to start instead from a bound only on some wedge power $\dim \Alt^r Z$, since the bound will be inherited by all lower wedge powers:

\begin{lemma} \label{Lemma:DecreasingDimensionInSmallDimension} 
Let $Z\subset A$ be a reduced subvariety whose Gauss map is a finite morphism of degree $n$. Let $r$ be an integer with $1 \le r \le n / 2$, and assume that
 \[ \dim \Alt^r Z < (\dim A-1)/2.\]
Then the function $d_Z\colon \{1,\dots, r\} \to \bbN, i\mapsto d_Z(i) := \dim \Alt^i Z$ is nondecreasing, in particular 
\[ \dim \Alt^i Z < (\dim A - 1)/2 \quad \textnormal{for all} \quad i\in \{1,\dots, r\}. \] 
\end{lemma}

\begin{proof} We only need to show $d_Z(r - 1) \le d_Z(r)$, because we can then proceed by descending induction. Put
\[ 
 \PLambda_i \;:=\; (\Alt^i \PLambda_Z) \circ (\Alt^i \PLambda_{-Z})
\quad \textnormal{for} \quad i \;=\; 1,\dots, n. \]
The support of this cycle is the closure of the subset of points in $(A \times \bbP_A)(k)$ of the form
$ (x_1 + \cdots + x_i - y_1 - \cdots - y_i, \xi)$
where \smallskip 
\begin{itemize}
\item  $\xi\in \bbP_A(k)$ is a cotangent direction such that $\pr_A(\gamma_{Z}^{-1}(\xi))\subset Z^\reg(k)$,\smallskip 
\item    $\{ x_1, \dots, x_i \}, \{ y_1, \dots, y_i\} \subset \pr_A(\gamma_Z^{-1}(\xi))$ are subsets of cardinality~$i$.\smallskip
\end{itemize}
For general $\xi \in \bbP_A(k)$ the fiber $\gamma_Z^{-1}(\xi)$ consists of $n \ge 2 r > 2(r-1) + 1$ distinct points, hence for any two subsets $\{ x_1, \dots, x_{r-1} \}$ and $\{ y_1, \dots, y_{r- 1} \}$ as above there is a point $p$ in $\pr_A(\gamma_{Z}^{-1}(\xi))$ which belongs to neither of the two subsets. By writing the point 
\[ z \;:=\;
 x_1 + \cdots + x_{r - 1} - y_1 - \cdots - y_{r - 1}\] 
as $z +p - p$, we find that the point $(z, \xi)$ lies in the support of $\PLambda_r$. Varying $\xi$ and the chosen subsets of points in the fiber of the Gauss map, we obtain the inclusion
$\Supp(\PLambda_{r - 1}) \subset \Supp(\PLambda_r)$
for the supports. Since both cycles are linear combinations of conormal varieties and hence pure of the same dimension $\dim A - 1$, it follows that we have
\begin{equation} \label{Eq:LangrangianCycleWedgePowersInduction}
	\PLambda_r = c \PLambda_{r-1} + E
\end{equation}
where $c > 0$ is a rational number and $E$ is an effective cycle with rational coefficients. 
	
\medskip 

The finiteness of the Gauss map $\gamma_Z\colon \PLambda_Z \to \bbP_A$ implies that $\Alt^i \PLambda_Z \to \bbP_A$ is finite as well. Indeed $\PLambda^{[i]}_Z$ is by definition a subvariety inside the $i$-fold fiber product $\PLambda_Z \times_{\bbP_A} \cdots \times_{\bbP_A} \PLambda_Z$ and the fiber product of finite morphisms is a finite morphism. By design the sum map  $\PLambda^{[i]}_Z \to \Supp(\Alt^i \PLambda_Z)$ is surjective and compatible with Gauss maps, implying the desired finiteness. By \cref{Lemma:SegreClassAndConvolution}, the finiteness of the Gauss map for $\Alt^i \PLambda_Z$ allows to compute the total Segre class of $\PLambda_i = \Alt^i \PLambda_Z \circ \Alt^i \PLambda_{-Z}$ as a Pontryagin product 
\[ s(\PLambda_i) \;=\; s(\Alt^i \PLambda_Z) \ast s(\Alt^i \PLambda_{-Z}) 
\quad \textnormal{in} \quad 
\CH_{<g}(A) \;:=\; \CH_\bullet(A)/\CH_g(A). \]
Comparing this with \eqref{Eq:LangrangianCycleWedgePowersInduction}, we obtain in the truncated Chow ring $\CH_{<g}(A)$ an identity
\[
s(\Alt^r \PLambda_Z) \ast s(\Alt^r \PLambda_{-Z}) 
\;=\;
c s(\Alt^{r-1} \PLambda_Z) \ast s(\Alt^{r-1} \PLambda_{-Z}) 
+ \cdots 
\]
where $\cdots$ stands for effective cycle classes (possibly zero).
The left-hand side of this equality vanishes in all degrees $> 2 d_Z(r)$ since $s_i(\Alt^r \PLambda_{Z}) = s_i(\Alt^r \PLambda_{-Z})= 0$ for all $i>d_Z(r)$. On the other hand, \cref{Lemma:NonVanishingSegreClassPontryaginProduct} says that on the right-hand side the Pontryagin product 
\[
 s(\Alt^{r-1} \PLambda_Z) \ast s(\Alt^{r-1} \PLambda_{-Z})
\]
is nonzero and effective in all degrees $\le \min \{ 2 d_Z(r - 1), \dim A - 1 \}$ because the subvariety $\Alt^{r - 1} Z$ has finite Gauss map and hence all its irreducible components are nondegenerate by \cref{Thm:PositivityNotions}. So a comparison of the left and right-hand side yields
\[ \min \{ 2 d_Z(r - 1), \dim A - 1\} \le 2 d_Z(r) < \dim A -1,\]
whence $2 d_Z(r - 1) \le \dim A - 1$ and $d_{Z}(r - 1) \le d_Z(r)$.
\end{proof}

\begin{proof}[{Proof of \cref{Thm:SmallWedgePowersHaveSmallDegree}}] Let $g := \dim A$. 
We argue by contradiction: If $r \dim Z \ge g$, consider the integer
\[ s \;:=\; \max \{ i \in \bbN \mid i \dim Z < g\} \;<\; r. \]
On the one hand, we have
\begin{align*}
 s \dim Z &\;=\; \dim \Alt^s Z  && \textnormal{by \cref{prop:linear-growth-for-small-wedge-powers}, since $s\dim Z < g$} \\
 &\;<\; (g-1)/2 && \textnormal{by \cref{Lemma:DecreasingDimensionInSmallDimension}, since $s<r$}.	
\end{align*} 
In particular, $\dim Z < (g-1)/2$. But on the other hand
\begin{align*} 
 s \dim Z \;=\; (s+1)\dim Z - \dim Z 
 & \;\ge \; g - \dim Z && \textnormal{since $(s+1)\dim Z \ge g$} \\
 & \; >\; (g+1)/2 && \textnormal{since $\dim Z < (g-1)/2$}
\end{align*} 
which contradicts the previous displayed inequality. 
\end{proof}

\subsection{Smooth wedge powers come from curves} \label{sec:smooth-wedge-powers}

It remains to show that the only smooth wedge powers in the dimension range in question are those coming from smooth curves, as announced in \cref{Thm:WedgePowersOfSmallDegrees}.
In what follows, let $X, Z\subset A$ be integral subvarieties such that\smallskip 
\begin{enumerate} 
\item the subvariety $X\subset A$ is smooth and nondivisible,\smallskip 
\item the Gauss maps $\gamma_X$ and $\gamma_Z$ are finite morphisms, and\smallskip 
\item we have $\Alt^r \PLambda_Z = \PLambda_{[e](X)}$ for some integers $e,r \ge 1$ with $r\dim Z < \dim A$. \smallskip
\end{enumerate} 
We claim that then $Y:=[e](X)$ is birational to the $r$-th symmetric power of $Z$, more precisely:

\begin{proposition} \label{prop:sum-morphism-finite}
If the above conditions (1), (2), (3) hold, then $Y=Z+\cdots + Z$ is a sum of $r$ copies of~$Z$, and the sum morphism $\sigma\colon Z^r \to Y$ is finite and factors through a finite birational morphism
\[
\tau \colon \quad  \Sym^r Z \;=\; Z^r/\frS_r \;\longrightarrow\; Y.
\]	
\end{proposition}

\begin{proof} 
The finiteness of the Gauss map $\gamma_Z$ implies that $Z\subset A$ is nondegenerate by \cref{Thm:PositivityNotions}. Since $r\dim Z < \dim A$, \cref{prop:linear-growth-for-small-wedge-powers} says that $\pr_{Z, r}\colon \PLambda_Z^{[r]} \to Z^r$ is surjective and the sum morphism 
\[ \sigma\colon \quad Z^r \too Y=Z+\cdots + Z 
\]
is generically finite. Moreover, $\sigma$ induces a morphism $\tau\colon \Sym^r Z \to Y$ and our assumption $\Alt^r \PLambda_Z = \PLambda_Y$ implies by  part (2) of~\cref{prop:linear-growth-for-small-wedge-powers} that~$\tau$ has generic degree $\deg(\tau)=1$, i.e.~it is birational. It remains to show that this morphism is finite. Since we have no control on the singularities of $Z$, we cannot use~\cref{Lem:BirationalMapIsIso}. Instead, to show that $\sigma$ and hence $\tau$ is finite, we
consider the following commutative diagram:
\[
\begin{tikzcd}[column sep=40pt]
 Z^r  \ar[d, swap, "\sigma"] &  \PLambda_Z^{[r]}  \ar[d, "\tilde{\sigma}"]  \ar[l, swap, "\pr_{Z,r}"]  \ar[r,"\gamma_{Z,r}"] &  \bbP_A  \ar[d, equals] \\
Y &  \PLambda_Y \ar[l, "\pr_Y"] \ar[r, swap, "\gamma_Y"]  & \bbP_A
\end{tikzcd} 
\]	
Here, the fibers of the morphism $\pr_Y$ are all of pure dimension $N=\codim_A Y - 1$ by \cref{Corollary:Equidimensional-Fibers}. Moreover, the morphism $\tilde{\sigma}$ is finite since the rightmost square in the above diagram commutes and since in that square the horizontal arrows~$\gamma_{Z,r}$ and~$\gamma_Y$ are finite morphisms by our finiteness assumptions on Gauss maps. Hence, it follows that all fibers of $\pr_Y \circ \tilde{\sigma} \colon \PLambda_{Z, r} \to Y$ are of dimension $N$. Since $\sigma$ is generically finite, it follows from the commutativity of the leftmost square in the above diagram that the generic fiber of the morphism $\pr_{Z, r}$ has dimension $N$ as well. We can now argue by contradiction: Any positive-dimensional fiber of $\sigma$ would give rise to a fiber of $\pr_Y \circ \tilde{\sigma}$ of dimension $\ge N+1$, by semi-continuity of dimension of fibers for proper morphisms~\cite[\href{https://stacks.math.columbia.edu/tag/0D4I}{lemma 0D4I}]{stacks-project}. This shows that $\sigma$ is finite.
\end{proof} 

\begin{proof}[Proof of \cref{Thm:WedgePowersOfSmallDegrees}]
In \cref{prop:sum-morphism-finite}, we have seen that the sum morphism~$\sigma$ factors through a finite birational morphism $ \tau \colon \Sym^r Z \to Y$.
We claim that a similar finite birational morphism from a symmetric product also exists for $X$ rather than for $Y=[e](X)$. For this, we use the following general remark: 

\medskip

For any integral subvariety $W \subset A$, the scheme-theoretic preimage $[e]^{-1}(W)$ is reduced because it is the preimage of a reduced subvariety under an \'etale morphism \cite[\href{https://stacks.math.columbia.edu/tag/03PC}{prop.~03PC (8)}]{stacks-project}. If $W'$ is an irreducible component of $[e]^{-1}(W)$, then any other irreducible component is of the form $W' + x$ for an $e$-torsion point $x\in A[e]$, and hence the morphism
$[e]\colon W' \to Z$ is surjective. In particular, by \cref{Prop:NonDegerateIsogeny} the subvariety $W\subset A$ is nondegenerate if and only if $W'\subset A$ is so.
 
\medskip 

Applying this to $W = Z$, we see that any irreducible component $Z'$ of $[e]^{-1}(Z)$ is nondegenerate. Hence, if we define $X' = Z' + \cdots + Z'\subset A$ to be the sum of $r$ copies of $Z'$, then \cref{Lem:SumOfNondegenerate} implies
\[ \dim X' = r \dim Z'  = r \dim Z = \dim Y.\]
It follows that $X'$ is an irreducible component of $[e]^{-1}(Y)$. On the other hand, we have
\[ [e]^{-1}(Y) \;=\; [e]^{-1}([e](X))
\;=\; \bigcup_{t \in A[e]} X + t.\]
Therefore, there is an $e$-torsion point $t\in A[e]$ such that $X = X' +t$ and $X$ is the sum of $r$ copies of 
\[ \tilde{Z} := Z' + u, \]
where $u\in A(k)$ is any point with $ru = t$. Since the stabilizer of $X=\tilde{Z}+\cdots + \tilde{Z}$ is trivial by assumption, it follows that the stabilizer of $\tilde{Z}$ is trivial, so the finite morphism $[e] \colon \tilde{Z} \to Z$ is birational. Then the morphism
\[
 \Sym^r [e] \colon \quad \Sym^r \tilde{Z} \;\too\; \Sym^r Z
\]
is finite birational by \cref{Prop:FiniteBirationMorphismBetweenSymmetricPowers}. Now the sum morphism $\tilde{\sigma} \colon \tilde{Z}^r \to X$ is invariant under the permutation action of $\frS_r$, so it factors through a morphism~$\tilde{\tau}$ as shown in the following commutative diagram:
\begin{equation}\label{eq:WedgePowersOfSmallDegrees}
\begin{tikzcd}
\Sym^r \tilde{Z} \ar[r, "\tilde{\tau}"] \ar[d, "{\Sym^r [e]}"']& X\ \ar[d, "{[e]}"] \\
\Sym^r Z \ar[r, "\tau"] & Y.
\end{tikzcd}
\end{equation}
The morphisms $\tau$, $[e]$ and $\Sym^r [e]$ in this diagram are finite birational, hence $\tilde{\tau}$ must be finite birational, too. On the other hand, the variety $X$ is smooth by hypothesis, which forces $\tilde{\tau}$ to be an isomorphism. In particular, the symmetric power $\Sym^{r} \tilde{Z}$ is smooth, which for $r>1$ implies that $\tilde{Z}$ is a smooth curve by \cref{Prop:SmoothSymmetricPower}.
\end{proof}

\begin{remark}
Given $X, Y, Z$ as in the proof of theorem~\ref{Thm:WedgePowersOfSmallDegrees}, it follows immediately from the fact that the morphisms $[e]$ and $\tau$ from \eqref{eq:WedgePowersOfSmallDegrees} are finite and birational that $X$ is dominated by the normalization $W$ of $\Sym^r Z$ which in turn is easily seen to be isomorphic to $\Sym^r \tilde Z$ where $\tilde Z$ is the normalization of $Z$. Smoothness of $X$ implies that in fact $\Sym^r \tilde Z \simeq X$. The argument we used instead is maybe a bit longer but yields more, namely a canonical (up to an $e$-torsion point) embedding of $\tilde Z$ into $A$.
\end{remark}

\section{Spin representations} \label{sec:spin}

We now show that under suitable assumptions on a smooth subvariety of an abelian variety, its Tannaka group cannot be the image of a spin group acting via a spin representation. For hypersurfaces this can be done by showing that their topological Euler characteristic is not a power of two~\cite[lemma~4.9]{LS20}; we here discuss the case of higher codimension, where we do not know the Euler characteristic but consider characteristic cycles as in the previous section.

\subsection{Statement of the main result} \label{sec:StatementResultsSpin}

Recall that for $N\ge 3$ the group $\SO_N(\bbF)$ admits a double cover
\[
 \Spin_N(\bbF) \;\too\; \SO_N(\bbF)
\] 
by the spin group. The spin group is a simply connected algebraic group and admits a faithful representation $\bbS_N\in \Rep_\bbF(\Spin_N(\bbF))$, the {\em spin representation} of dimension $\dim S_N = 2^n$ where $n=\lfloor N/2 \rfloor$. The behavior of this representation depends on the Dynkin type (see \cite[\S 20]{FultonHarris}): \medskip 
\begin{enumerate}
\item[$B_n$:] If $N=2n+1$ is odd, then the spin representation $\bbS_N$ is irreducible.\medskip
\item[$D_n$:] If $N=2n$ is even, then $\bbS_N \simeq \bbS_N^+ \oplus \bbS_N^-$ splits as the direct sum  of two irreducible representations called the {\em half-spin representations}. They both have dimension $\dim \bbS_N^+ = \dim \bbS_N^- = 2^{n-1}$ but are not isomorphic to each other, they are only related by an outer automorphism of the spin group. The dual of the half-spin representations and the center of the spin group are given by the following table:
\begin{center}
 \renewcommand{\arraystretch}{1.3}
\begin{tabular}{c|c|c}
\ & dual of $\bbS_N^+$ & center of $ \Spin_{2n}(\bbF)$ \\
\hline \hline 
$n$ even & $\bbS_N^+$ & $\bbZ/2\bbZ \times \bbZ/2\bbZ$ \\
$n$ odd & $\bbS_N^-$ & $\bbZ/4\bbZ$
\end{tabular}
 \renewcommand{\arraystretch}{1}
\end{center}

For $n=2m+1$ odd, the half-spin representations are faithful. For $n=2m$ even, the half-spin representation $\bbS_\pm$ is self-dual and the natural pairing is symmetric if $m$ is even and alternating if $m$ is odd. The images of $\Spin_{4m}(\bbF)$ via the half-spin representations
\[
 \Spin_{4m}^\pm(\bbF) \;\subset\; \GL(\bbS_{4m}^\pm)
\]
are called the {\em half-spin groups}. They are isomorphic to each other and fit in the following diagram of isogenies given by dividing out the subgroups of $Z(\Spin_{4m}(\bbF))\simeq \bbZ/2\bbZ\times \bbZ/2\bbZ$:
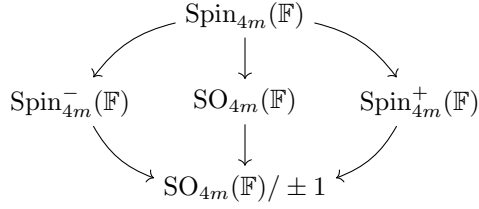
\[
\begin{tikzpicture}[scale=1.15]
\node at (0, 1) (a) {$\Spin_{4m}(\bbF)$};
\node at (-2, 0) (b) {$\Spin_{4m}^-(\bbF)$};
\node at (0, 0) (c) {$ \SO_{4m}(\bbF)$};
\node at (2, 0) (d) {$\Spin_{4m}^+(\bbF)$};
\node at (0, -1) (e) {$\SO_{4m}(\bbF)/\pm 1$};
\draw[->] (a) edge[bend right=20] (b);
\draw[->] (a) -- (c);
\draw[->] (a) edge[bend left=20] (d);
\draw[->] (b) edge[bend right=20] (e);
\draw[->] (c) -- (e);
\draw[->] (d) edge[bend left=20] (e);
\end{tikzpicture}
\]
\end{enumerate}

\noindent We are interested in geometric incarnations of the above:

\begin{definition} 
Let $X\subset A$ be a subvariety. Let $G=G_{X, \omega}^\ast$ be the derived group of the connected component of the corresponding Tannaka group, and consider the faithful representation $V=\omega(\delta_X)_{\rvert G} \in \Rep_\bbF(G)$. Let $n \ge 1$ be an integer. We say that $X\subset A$ is \smallskip 
\begin{itemize} 
\item of {\em spin type $B_n$} if $G \simeq \Spin_{2n+1}(\bbF)$ and $V\simeq \bbS_{2n+1}$.\smallskip 
\item of {\em spin type $D_n$} if $G \simeq \Spin_{2n}^\epsilon(\bbF)$ and $V\simeq \bbS_{2n}^\epsilon$ for some $\epsilon \in \{+, -\}$.\smallskip 
\end{itemize} 
\end{definition} 

In both cases, the subvariety $X$ is irreducible because the representation $V$ is so.

\begin{remark} \label{Rmk:ParitySpinRep} Suppose that $X$ is of spin type $D_n$ for $n = 2m$. Since half-spin representations are self-dual in this case, the subvariety $X$ is symmetric up to translation. The Poincar\'e pairing on $X$ is symmetric if $d = \dim X$ is even and alternating if $d$ is odd. By comparison with the natural pairing on $\bbS_n^{\pm}$, the integers $m$ and $d$ must have the same parity.
\end{remark}

\noindent The goal of this section is to show that for smooth subvarieties of small dimension this cannot happen. To state our results, let $g = \dim A$:

\begin{theorem}\label{Thm:SmallSpinDoNotExist}
Let $X\subset A$ be a $d$-dimensional nondivisible smooth subvariety with ample normal bundle and $d < (g- 1)/2$. Then, for any integer $n \ge 1$,
\begin{enumerate}
\item $X$ is not of spin type $B_n$; \smallskip
\item if $X$ is of spin type $D_n$, then $d \ge (g - 1)/4$, $n = 2m$ with $m \in \{3,\dots,d\}$ having the same parity as $d$.
\end{enumerate}
\end{theorem}

The list of Dynkin types in the above theorem starts with $B_2$ and $D_3$, and at least for these smallest cases the result is optimal in the sense that the dimension bound cannot be weakened:

\begin{example} \label{Exa:low_dim_Spin}
Let $X$ be a smooth projective curve of genus $g=3$, embedded in its Jacobian variety $A=\Alb(X)$. By~\cite[th.~6.1]{KWSmall}, \cite{WeissauerBNSheaves} there are two cases:\smallskip 
\begin{enumerate} 
\item[$B_2$:] If $X$ is hyperelliptic, then $G_{X, \omega}^\ast \simeq \Sp_4(\bbF)$ and  $\omega(\delta_X)\simeq \bbF^4$ is its standard representation. This representation corresponds to the spin representation under the isomorphism $\Spin_5(\bbF)\simeq \Sp_4(\bbF)$. \smallskip 
\item[$D_3$:] If $X$ is not hyperelliptic, then $G_{X, \omega}^\ast \simeq \SL_4(\bbF)$ and $\omega(\delta_X) \simeq \bbF^4$ is its standard representation or its dual. These two representations correspond to the two half-spin representations under the isomorphism $\Spin_6(\bbF) \simeq \SL_4(\bbF)$. \smallskip 
\end{enumerate} 
So spin representations do occur, but in this example $2\dim X = \dim A - 1$.
\end{example}

\subsection{Structure of the proof} 

The proof of \cref{Thm:SmallSpinDoNotExist} relies on three independent steps. The first is to show that the structure of spin representations is reflected by characteristic cycles: We say that a cycle $\PLambda \in \cL(A)$ is {\em symmetric} if $[-1]_* \PLambda = \PLambda$. In this case, if the cycle is reduced and effective of Gauss degree $\deg \PLambda = 2n$, then for any integer $r\ge 1$ we will define via the formalism in \cref{lem:weight_multiplicities} a clean effective cycle 
\[ 
\Alt^{r}_\sym \PLambda \;\in\; \cL(A) 
\quad \textnormal{of Gauss degree} \quad 
 \deg \Alt^r_\sym \PLambda \;=\; 2^r \tbinom{n}{r}. 
\]
These cycles are closely related to the wedge powers from~\cref{sec:wedge}, for instance we have
\[
 \Supp \Alt^r \PLambda \;=\; \bigcup_{i=0}^{\lfloor r/2\rfloor} \Supp \Alt^{r-2i}_\sym \PLambda.
\]
For $r=n$ the cycle $\Alt^n_\sym \PLambda$ will correspond to the spin representation. 
For half-spin representations we consider the monodromy of the Gauss map $\gamma_{\PLambda}\colon \PLambda \to \bbP_A$: If the Gauss map has even monodromy in the sense of \cref{defn:cc-in-type-D}, we will obtain a decomposition
\[
 \Alt^n_\sym \PLambda \;=\; \Alt^n_{\sym, +} \PLambda + \Alt^n_{\sym, -} \PLambda
\]
where $ \Alt^n_{\sym,\pm} \PLambda\in \cL(A)$ are clean effective cycles of Gauss degree $2^{n-1}$. We show:

\begin{theorem} \label{Thm:FromRepresentationToPhysicalSpin} 
Let $X\subset A$ be a smooth nondivisible subvariety of spin type $B_n$ or~$D_n$ for some $n$. Then there exist $a\in A(k)$, a reduced symmetric effective cycle~$Z$ on $A$ of Gauss degree $\deg \PLambda_Z = 2n$ and an integer $e\ge 1$ with the following properties:\smallskip 
\begin{enumerate} 
\item If $X$ is of spin type $B_n$, then 
$\Alt^{n}_\sym \PLambda_{Z} = \PLambda_{[e](X+a)}$.\medskip
\item  If $X$ is of spin type $D_n$, then the Gauss map $\gamma_{\PLambda}$ has even monodromy and we have
\[ \Alt^n_{\sym,\epsilon} \PLambda_Z = \PLambda_{[e](X + a)} 
\quad \textnormal{\em for suitable} \quad \epsilon \in \{+, -\}.
\]
\end{enumerate} 
In both cases, if $\gamma_X\colon \PLambda_X \to \bbP_A$ is a finite morphism, then so is $\gamma_Z\colon \PLambda_Z \to \bbP_A$.
\end{theorem}

For the precise definition of the clean cycles appearing above and the proof of the theorem, we refer to \cref{sec:cc-of-spin}. The second step will be a dimension estimate for the images
\[
  \Alt^{n}_{\sym, \epsilon} Z 
 \;:=\;
 \pr_A( \Supp(\Alt^{n}_{\sym, \epsilon} \PLambda_Z) 
 \;\subset\; A,
\]
for $\epsilon \in \{ +, - , \varnothing \}$. We show:

\begin{theorem} \label{Thm:SmallSpinHaveSmallDegree} Let $Z$ be a reduced symmetric effective cycle on $A$ whose Gauss map is a finite morphism of even degree $\deg(\gamma_Z)=2n$.~\medskip 
\begin{enumerate} 
\item If 
$\dim \Alt^{n}_\sym Z < (g-1)/2$, then $n\dim Z < g - 1$.~\medskip
\item If the Gauss map $\gamma_Z$ has even monodromy and there is $\epsilon \in \{ +, -\}$ for which the dimension $d: = \dim \Alt^n_{\sym, \epsilon} Z$ satisfies
\[ 
d < \begin{cases}
(g-1)/4 &\text{if $n = 2m$ is even and $m \le d + 1$}, \\
(g-1)/2 &\text{otherwise}, 
\end{cases}
\]
then $n\dim Z < g - 1$.
\end{enumerate} 
\end{theorem}

The proof of this is given in \cref{sec:spin-dimension-estimate}. Finally, the last step for ruling out spin representations will be to show that in the given dimension range, the cycle $\Alt^n_{\sym, \epsilon} \PLambda_Z$ cannot be smooth and integral. More precisely:

\begin{theorem} \label{Thm:SpinOfSmallDegrees}
Let $X \subset A$ be a smooth nondivisible subvariety with ample normal bundle and $Z$ a reduced symmetric effective cycle on $A$ whose Gauss map is finite of even degree $\deg(\gamma_Z)=2n\ge 4$. Suppose that for some integer~$e\ge 1$ one of the following two conditions holds: 
\begin{enumerate} 
\item $\PLambda_{[e](X)} = \Alt^{n}_\sym \PLambda_Z$.
\item $\gamma_Z$ has even monodromy and $\PLambda_{[e](X)} = \Alt^{n}_{\sym,\epsilon} \PLambda_Z$ for some $\epsilon\in \{+,-\}$.	
\end{enumerate}  
Then $n \dim Z \ge \dim A$. 
\end{theorem}

\begin{proof}[{Proof of \cref{Thm:SmallSpinDoNotExist}}] 
	Let $X\subset A$ be a $d$-dimensional smooth nondivisible subvariety of spin type $B_n$ or $D_n$ for some integer $n$ with $d < (g-1)/2$. After replacing $X$ by a translate, there exists by \cref{Thm:FromRepresentationToPhysicalSpin} an integer $e\ge 1$ and a reduced symmetric effective cycle $Z$ on $A$ with $\deg(\gamma_Z)=2n$ such that
	\[ \Alt^{n}_{\sym, \epsilon} \PLambda_Z \;=\; \PLambda_{[e](X)}
	\quad \textnormal{for some} \quad \epsilon \;\in\; \{+, -, \varnothing\}. \]
	It follows that
	\[ \dim \Alt^{n}_{\sym, \epsilon} Z \;=\; \dim [e](X) \;=\; \dim X.
	\]
	Note that we have $\epsilon \neq \varnothing$ only in spin type $D_n$ and in that case $\gamma_Z$ has even monodromy. 
	Moreover, if $X$ is of spin type $D_n$, $n = 2m$ is even and $m \le d + 1$, assume $d < (g-1)/4$.  Then~\cref{Thm:SmallSpinHaveSmallDegree} implies
	$n \dim Z  < \dim A$, which contradicts \cref{Thm:SpinOfSmallDegrees}. The remaining cases are of type $D_n$ with $n = 2m$ even, $d \ge (g-1)/4$ and $m \le d + 1$. \Cref{Rmk:ParitySpinRep} implies that $d - m$ must be even.  \Cref{Lemma:LowerBoundTopCharAmpleNormalBundle,,Lemma:LowerBoundTopCharSurfaces} imply that in the current dimension range the absolute value of the topological Euler characteristic of $X$ is never $8$, so the case $m = 2$ does not occur.
	\end{proof} 

\subsection{Characteristic cycles of spin representations} 
\label{sec:cc-of-spin}

We now explain how to compute characteristic cycles for subvarieties of spin type. Consider a symmetric reduced clean effective cycle $Z$ on $A$ with dominant Gauss map $\gamma_Z\colon \PLambda_Z \to \bbP_A$. For an integer $r\ge 1$ we put 
\[
\PLambda_{Z,\sym}^{[r]} \;:=\; 
\overline{\PLambda_{Z\vert U}^{\times r} \smallsetminus (\Delta_r \cup \Delta_r^-)} \;\subset\; \PLambda_Z^{\times r}
\] 
as in \cref{defn:fiber-product-minus-diagonal}, where $U\subset \bbP_A$ is any open dense subset over which the Gauss map $\gamma_Z$ is finite and \'etale. For a partition $\alpha = (\alpha_1, \dots, \alpha_r)$ consider the sum morphism 
\[ \sigma_\alpha \colon A^r \times \bbP_A \too A\times \bbP_A, \qquad (z_1, \dots, z_r, \xi) \longmapsto (\alpha_1 z_1+\cdots + \alpha_rz_r, \xi). \]
We put
\[
 \PLambda_{Z, \sym}^{\alpha} \;:=\; \sigma_{\alpha*}(\PLambda_{Z, \sym}^{[r]}).
\]
We are mostly interested in the special case of the partition $\alpha=(1^r)=(1,\dots, 1)$ and write
\[
 \Alt^{r}_\sym \PLambda_Z \;:=\; \PLambda_{Z, \sym}^{(1^r)}
\]
in this case. As in \cref{defn:cc-in-type-D}, we say that the Gauss map $\gamma_Z\colon \PLambda_Z \to \bbP_A$ has {\em even monodromy} if its degree is an even integer $\deg(\gamma_Z)=2n$ and the finite \'etale cover $\gamma_{Z\vert U}\colon \PLambda_{Z\vert U} \to U$ has as its monodromy group a subgroup of $(\pm 1)^n_+\rtimes \frS_n$. In this case  
\[
 \PLambda_{Z, \sym}^{[n]} \;=\; \PLambda_{Z, \sym, +}^{[n]} +  \PLambda_{Z, \sym, -}^{[n]}
\]
for $\PLambda_{Z, \sym, \pm}^{[n]}$ as in \cref{defn:cc-in-type-D}, and for partitions $\alpha$ of length $r=n$ we put
\[
 \PLambda_{Z, \sym, \pm}^\alpha \;=\; \sigma_{\alpha*}(\PLambda_{Z, \sym, \pm}^{[n]}).
\] 
Again for $\alpha=(1^n)$ we instead write $\Alt^n_{\sym, \pm} \PLambda_Z := \PLambda_{Z,\sym, \pm}^{(1^n)}$.

\begin{proof}[Proof of \cref{Thm:FromRepresentationToPhysicalSpin}]
Replacing $X\subset A$ by a translate we may assume $\det(\delta_X)=\delta_0$. As in the proof of \cref{Thm:FromRepresentationToPhysicalWedgePowers}, but replacing \cref{lem:cc-in-type-A} by \cref{lem:cc-in-type-B}, we then find a reduced symmetric effective cycle $Z$ on $A$ with $\PLambda_Z \in \langle \PLambda_X \rangle$ such that
\[
 \PLambda_{[e](X)} \;=\; \Alt^{n}_{\sym, \epsilon} \PLambda_Z
 \quad \textnormal{and} \quad 
 \deg(\gamma_Z) \;=\; 2n
\]
for some integer $e\ge 1$ and $\epsilon \in \{+,-,\varnothing\}$ as claimed. 
\end{proof}

\subsection{Dimension estimates in the spin case}
\label{sec:spin-dimension-estimate}

Let $Z$ be a symmetric reduced effective cycle in $A$.  
As before, we put
\begin{align*} 
\Alt^{r}_{\sym} Z \;&:=\; \pr_A(\Supp \Alt^{r}_{\sym} \PLambda_Z) \;\subset\; A,\\
\Alt^{n}_{\sym, \epsilon} Z \;&:=\; \pr_A(\Supp \Alt^{n}_{\sym, \epsilon} \PLambda_Z) \;\subset\; A.
\end{align*}
We want to show that under certain assumptions the dimension of these images grows linearly with $r$ in a suitable range. The first step is to relate the cycles associated with the two half-spin representations:

\begin{lemma} \label{Lemma:BadHalfSpinBound} Let $Z \subset A$ be a symmetric reduced effective cycle whose Gauss map~$\gamma_Z$ is finite of degree $\deg(\gamma_Z)=2n\ge 4$ and has even monodromy. Suppose that $n $ is odd, or that $n= 2m$ is even and $m -1 > \min_{\epsilon \in \{ +, - \}} \dim \Alt^n_{\sym, \epsilon} Z$. Then,
\[ \dim \Alt^n_{\sym, +} Z = \dim \Alt^n_{\sym, -} Z. \]
\end{lemma}

\begin{proof} We may suppose $k = \bbC$. For $\epsilon \in \{ + , - \}$ set 
\[ \PLambda_\epsilon := \Alt^n_{\sym, \epsilon} \PLambda_Z, \qquad  X_\epsilon = \Alt^n_{\sym, \epsilon} Z = \pr_A(\PLambda_\epsilon).\]
 If $n=2m+1$ is odd, then the definitions imply $X_+ = [-1]_* X_-$ and the statement follows. The representation-theoretic analogue of the previous identity is that the half-spin representations $\bbS_+$ and $\bbS_-$ are dual to each other. Suppose henceforth that $n = 2m$ is even and, up to changing signs, that we have $\dim X_+ \le \dim X_-$. With this notation $m-1 > d:= \dim X_+$. Similarly, drawing inspiration from the isomorphism of representations $\Alt^2 (\bbS_+) \iso \Alt^2(\bbS_-)$ we obtain the following:
\begin{claim} The following identity of Lagrangian cycles holds:
\[ \PLambda_+ \circ \PLambda_+ - [2]_\ast \PLambda_+ =  \PLambda_- \circ \PLambda_- - [2]_\ast \PLambda_- \in \cL(A).\]
\end{claim}

\begin{proof}[Proof of the claim]
Pick a general cotangent direction $v\in \bbP_A(k)$, and denote by
\[ \PLambda_{Z, v} \;=\; \{\pm p_1,\ldots, \pm p_n\} \;\subset\; A \]
the corresponding fiber of the Gauss map, which we identify as a $0$-cycle on the abelian variety via the projection $\pr_A\colon \PLambda_Z \to A$. Let $\{ \pm 1 \}^n_\epsilon \subset \{ \pm 1\}^n$ be the subset made of $n$-tuples $a = (a_1, \dots, a_n)$ such that the sign of $a_1 \cdots a_n$ is $\epsilon$. As $0$-cycles on $A$, we have
\begin{align*}
	(\PLambda_\epsilon)_v \;&=\;  \sum_{a \in \{ \pm 1 \}^n_\epsilon} [a_1 p_1 + \cdots +  a_n p_n] \in \rZ_0(A), \\
	(\PLambda_\epsilon \circ \PLambda_\epsilon)_v 
	\;&=\; 
	\sum_{a, b \in \{ \pm 1 \}^n_\epsilon} [ (a_1 + b_1) p_1 + \cdots +  (a_n+ b_n) p_n] \in \rZ_0(A),
\end{align*}
where the summation sign refers to the sum as cycles and $[x]$ is the $0$-cycle given by a point $x \in A(k)$. We split the sum in the preceding equation in two: The sum ranging on couples $(a, b)$ with $a= b$ gives the $0$-cycle $([2]_\ast \PLambda_\epsilon)_v$. For the remaining couples, notice that we have a bijection
\[  \left\{ \left. (a, b) \in (\{ \pm 1 \}^n_+ )^2  \, \right| \, a \neq b \right\} \too \left\{ \left. (a', b') \in (\{ \pm 1 \}^n_-)^2 \, \right| \, a' \neq b' \right\} \]
sending $(a, b)$ to the couple $(a', b')$ obtained by changing the sign of the first entry in which $a$ and $b$ differ. Such a bijection is compatible with sum, that is $a + b = a' + b'$. Letting vary $v$ gives the desired identity of cycles.
\end{proof}

The identity in the claim can be rewritten as:
\[ (\PLambda_+ + \PLambda_-) \circ (\PLambda_+ - \PLambda_-) = [2]_\ast (\PLambda_+ - \PLambda_-).\]
Since $\PLambda_Z$ has finite Gauss map, so do $\PLambda_+$ and $\PLambda_-$ by construction. By \cref{Lemma:SegreClassAndConvolution} we obtain, by passing to Segre classes,
\[ s(\PLambda_+ + \PLambda_-) \ast s(\PLambda_+ - \PLambda_-) = [2]_\ast s(\PLambda_+ - \PLambda_-) \in \CH_{<g}(A).\]
Let $s$ and $\delta$ be the image of $s(\PLambda_+)$ and $s(\PLambda_+ - \PLambda_-)$ in the homology $\rH_\bullet(A, \bbZ)$, so that 
\[ (2s + \delta) \ast \delta = [2]_\ast \delta \in \rH_\bullet(A, \bbZ),\] 
where $\ast$ denotes the Pontryagin product in homology. Write $s_i, \delta_i \in \rH_{2i}(A, \bbZ)$ for the pieces of $s, \delta$ in degree $2i$. Then in degree $2r$ the previous identity reads as
\[ \sum_{i = 0}^r (2 s_i + \delta_i) \ast \delta_{r - i} = 2^{2r} \delta_r \in \rH_{2r}(A, \bbZ),\]
because the multiplication by $2$ on $A$ acts as multiplying by $2^{2r}$ on $\rH_{2r}(A, \bbZ)$. Now for $\epsilon \in \{ +, -\}$, the $0$-th Segre class of $\PLambda_\epsilon$ has degree
\[  \deg \gamma_{\PLambda_\epsilon} = |\{ \pm 1\}^n_\epsilon| = 2^{2m - 1}.\]
Plugging in the identities $s_0 = 2^{2m - 1}$ and $\delta_0 = 0$ obtained in this way yields the recursion formula
\[ (2^{2r} - 2^{2m}) \delta_r = \sum_{i = 1}^{r - 1} (2 s_i + \delta_i) \ast \delta_{r - i} \in \rH_{2r}(A, \bbZ).\]
Then the vanishing of $\delta_0$ inductively implies
\[ \delta_r = 0, \qquad r = 0, \dots, m-1.\]
Suppose by contradiction $\dim X_- > d= \dim X_+$. Then
\[ \delta_{d + 1} = s_{d+1}(\PLambda_-) - s_{d + 1}(\PLambda_+) = s_{d+1}(\PLambda_-)  \neq 0\]
where we abusively identified Segre classes with their images in homology. Then the vanishing of $\delta_r$ for $r < m$  implies $d +1 \ge m$, contradicting $m - 1 > d$.
\end{proof}

The second step is the following analog of \cref{Lemma:DecreasingDimensionInSmallDimension} (note that at first we only get a weaker estimate in the spin case since we here only start from a dimension bound on the support of $\Alt^r_\sym Z$, which a priori might be strictly smaller than the one of $\Alt^r Z$):  

\begin{proposition} \label{Prop:dimension_of_Lambda_1_k} 
Let $Z$ be a symmetric reduced effective cycle on $A$ whose Gauss map has degree $\deg(\gamma_Z)=2n\ge 4$. \smallskip
\begin{enumerate} 
\item If $\dim \Alt^{n}_\sym Z < (g-1)/2$, then $\dim \Alt^{i}_\sym Z <  g-1$ for $i = 1, \dots, n -1$.	
\item If $\gamma_Z$ is finite and has even monodromy and if there is $\epsilon \in \{+,-\}$ such that the dimension  $d:= \dim \Alt^n_{\sym, \epsilon} Z$ satifies
\[ 
d < 
\begin{cases}
(g-1)/4 &\text{if $n = 2m$ is even and $m \le d + 1$}, \\
(g-1)/2 &\text{otherwise}, 
\end{cases}
\] then $\dim \Alt^i_{\sym} Z <  g - 1$ for $1 \le i \le n$ and $\dim \Alt^n_{\sym} Z < \frac{g - 1}{2}$ if $n$ is odd.
\end{enumerate} 
\end{proposition}

\begin{proof} We may suppose $k = \bbC$. (1) Put
$\PLambda^{(s)} := \Alt^{s}_\sym \PLambda_Z$. Pick a general cotangent direction $v\in \bbP_A(k)$, and denote by
\[ \PLambda_{Z, v} \;=\; \{\pm p_1,\ldots, \pm p_n\} \;\subset\; A \]
the corresponding fiber of the Gauss map, which we identify as usual with a set of points on the abelian variety via the projection $\pr_A\colon \PLambda_Z \to A$. Writing $p_{-i} := -p_i$ we have
\begin{align*}
 \Supp (\PLambda^{(s)})_v \;&=\; \{
 p_{i_1} + \cdots + p_{i_s} \mid |i_1|, \dots, |i_s| \textrm{ pairwise distinct } 
 \}, \\
	\Supp (\PLambda^{(n)} \circ \PLambda^{(n)})_v 
	\;&=\; 
	\left\{ \left. \sum_{i = 1}^n (\delta_i + \epsilon_i) p_i   
	\right| \delta_1, \dots, \delta_n,  \epsilon_1, \dots, \epsilon_n \in \{ \pm 1\}
	\right\}.
\end{align*}
By specializing to the case $\delta_i = \pm \epsilon_i$ and varying the cotangent direction $v\in \bbP_A(k)$ we find
\[
\Supp \left([2]_\ast \PLambda^{(1)} + [2]_\ast \PLambda^{(2)} + \cdots + [2]_{\ast} \PLambda^{(n)} \right) 
\;\subset\; 
\Supp(\PLambda^{(n)} \circ \PLambda^{(n)}) .
\] 
Hence, for the images $Z^{(s)} := \pr_A(\Supp(\PLambda^{(s)})) = \Supp(\Alt^{s}_\sym Z) \subset A$ we obtain the inclusions
\[
[2]\left( Z \cup Z^{(2)} \cup \cdots\cup Z^{(n)}    \right)
	\; \subset \;  \pr_A(\Supp(\PLambda^{(n)} \circ \PLambda^{(r)})) \subset  Z^{(n)}  + Z^{(n)}  .
\] 
Since $[2]\colon A\to A$ is an isogeny, a look at dimension then shows that for $1\le s\le n$ we have
\[ 
\dim Z^{(s)} \le \dim(Z^{(n)}+Z^{(n)}) \le 2\dim (Z^{(n)}) < g-1 \]
where the last inequality holds by our dimension assumption. This proves (1).

\medskip 

(2) If $n=2m+1$ is odd or $n = 2m$ is even with $m > d + 1$, then \cref{Lemma:BadHalfSpinBound} implies
\[ d = \dim \Alt^n_{\sym, +} = \dim \Alt^n_{\sym, -} = \dim \Alt^n_\sym Z\]
so we are done by (1). Now assume $n=2m$ and $m  -1\le d < (g-1)/4$. In that case, a set-theoretic look at the fibers of Gauss maps gives 
\[
\Supp \left([2]_\ast \PLambda^{(2)} + [2]_\ast \PLambda^{(4)} + \cdots + [2]_{\ast} \PLambda^{(n-2)} \right) 
\;\subset\; 
\Supp(\Alt^{n}_{\sym, \epsilon} \PLambda_Z \circ \Alt^{n}_{\sym, \epsilon} \PLambda_Z) .
\]
For all $i<m$ then
\[
 \dim Z^{(2i)} \;\le\; 2 \dim \Alt^n_{\sym, \epsilon} \PLambda_Z \;<\; (g-1)/2,
\]
where the second inequality holds by assumption. Since $\Alt^2 Z = \Alt^2_\sym Z + \Alt^0_\sym Z$ it follows that
$
 \dim \Alt^2 Z  < (g-1)/2
$,
and the monotonicity of usual wedge powers in \cref{Lemma:DecreasingDimensionInSmallDimension} gives
\[
 \dim Z \;=\; \dim \Alt^1 Z \;\le\; \dim \Alt^2 Z \;<\; (g-1)/2,
\] 
since we assumed the Gauss map to be finite. Then $\Supp \PLambda^{(2i+1)} \subset \Supp \PLambda^{(2i)}\circ \PLambda^{(1)}$ also implies $\dim Z^{(2i+1)} \le \dim Z^{(2i)} + \dim Z^{(1)} < (g-1)/2 + (g-1)/2 = g-1$ for all $i<m$. Likewise, $\Supp \PLambda^{(n)} \subset \Supp \PLambda^{(n-2)}\circ \PLambda^{(2)}$ gives $\dim Z^{(n)} < g-1$.
\end{proof}

\begin{proof}[Proof of \cref{Thm:SmallSpinHaveSmallDegree}] 
Let $Z$ be a reduced symmetric effective cycle on $A$ whose Gauss map is a finite morphism of degree $\deg(\gamma_Z)=2n$ for some integer $n\ge 2$. Assume that we are in one of the following two cases:\smallskip 
\begin{enumerate} 
\item $\dim \Alt^{n}_\sym Z < (g-1)/2$, or\smallskip 
\item $n = 2m$ is even, the Gauss map $\gamma_Z$ has even monodromy, and for some $\epsilon \in \{+,-\}$ we have
\[ \dim \Alt^n_{\sym, \epsilon} < (g-1)/4.
\]
\end{enumerate} 
By \cref{Prop:dimension_of_Lambda_1_k} then
$ \dim \Alt^i_\sym Z < g-1$
for all $i\in \{1,\dots, n\}$, 
so it will be enough to show that 
\[
 \dim \Alt^{i}_\sym Z \;=\; i\dim Z
\]
for all those $i$. For $i=1$ there is nothing to show. We now use induction: Suppose that
\[
 \dim \Alt^{i}_\sym Z \;=\; i\dim Z
 \quad \textnormal{for all} \quad  
 i \;\in\; \{1,2,\dots, s-1\},
\] 
where $s\le r$ is a positive integer. We want to conclude $\dim \Alt^{s}_\sym Z = s\dim Z$. 

\medskip 
 
We may assume $\dim Z > 0$. For simplicity, we put $\PLambda^\beta := \PLambda_{Z, \sym}^\alpha$ where $\beta=\alpha^t$ denotes the transpose of a partition $\alpha$, extending our notation from the previous proof. The definitions imply
\begin{equation} \label{eq:convolution-support}
\Supp \left( \PLambda^{(s-1)} \circ \PLambda^{(1)}\right) = \PLambda^{(s)} \cup  \PLambda^{(s-1,1)} \cup \PLambda^{(s-2)}.
\end{equation} 
We now compare dimensions. For the last two pieces on the right-hand side, the inclusions $\PLambda^{(s-1,1)} \subset \PLambda_{2Z} \circ  \PLambda^{(s-2)}$ and $\PLambda^{(s-2)}\subset \PLambda_Z\circ \cdots \circ \PLambda_Z$ (with $s-2$ factors) imply 
\begin{align*}
\dim \pr_A\left( \PLambda^{(s-1,1)} \cup \PLambda^{(s-2)} \right) 
&\;\le \; (s-1)\dim Z && \\
&\;=\; \dim \Alt^{s-1}_\sym Z && \textnormal{by our induction assumption} \\
&\;<\; g-1 && \textnormal{by \cref{Prop:dimension_of_Lambda_1_k} as $s -1 \le n$.}
\end{align*}
Since $\dim Z > 0$, it follows that
\begin{eqnarray} \label{eq:upper-estimate}
\dim \pr_A\left( \PLambda^{(s-1, 1)} \cup \PLambda^{(s-2)} \right)  
&<&  \min \{ s\dim Z, g-1 \}  
\\ \nonumber 
&=& 
\dim 
\pr_A \left( \PLambda^{(s-1)} \circ \PLambda^{(1)} \right),
\end{eqnarray}
where the last equality holds by \cref{cor:DimensionOfConvolutionViaSegre} since by definition $\PLambda^{(1)} = \PLambda_Z$ and by induction $\dim \pr_A(\PLambda^{(s-1)}) = \dim \Alt^{s-1}_\sym Z = (s-1) \dim Z$. Note that the corollary does not require integrality of the occurring clean cycles; we only need that each irreducible component of their support has finite Gauss map, which follows from our assumption that the Gauss map $\PLambda_Z \to \bbP_A$ is a finite morphism.
Comparing~\eqref{eq:convolution-support} and~\eqref{eq:upper-estimate} we conclude
\[ \dim \pr_A (\PLambda^{(s)}) = \min \{ s\dim Z, g-1 \}. \]
But $ \dim \pr_A (\PLambda^{(s)}) < g-1$ again by \cref{Prop:dimension_of_Lambda_1_k} because $s\leq n$. Hence, it follows that
\[
 \dim \pr_A (\PLambda^{(s)}) \;=\; s \dim Z
\] 
which completes the induction step.
\end{proof}

\subsection{Spin representations do not occur in small dimension} 

It remains to show that smooth nondivisible subvarieties of small enough dimension are not of spin type. Recall the setting of \cref{Thm:SpinOfSmallDegrees}:
We are given an integral subvariety $X\subset A$ and a reduced symmetric effective cycle $Z$ on $A$ such that \smallskip 
\begin{itemize} 
	\item the subvariety $X\subset A$ is smooth and nondivisible,  \smallskip 
	\item the Gauss maps $\gamma_X$, $\gamma_Z$ are finite, and $\deg(\gamma_Z)=2n$ for an integer $n\ge 2$, \smallskip 
	\item we have
	$ \Alt^{n}_{\sym, \epsilon} \PLambda_Z = \PLambda_{[e](X)}$ 
	for some $e\ge 1$ and suitable $\epsilon \in \{+,-,\varnothing\}$.\smallskip
\end{itemize} 
Here we make the convention that the labels $\epsilon=\pm$ will only be used if $\gamma_Z$ has even monodromy, so that the corresponding cycles are defined.
We want to show that in the above situation $n\dim Z \ge \dim A$:

\begin{proof}[Proof of \cref{Thm:SpinOfSmallDegrees}]  
Suppose by contradiction that $n\dim Z < \dim A$. Let $Z'\subset A$ be an irreducible component of maximal dimension in $Z$. The finiteness of the Gauss map implies by \cref{Thm:PositivityNotions} that $Z'\subset A$ is nondegenerate. Hence, \cref{cor:gauss-degree-bound} gives 
\[
 \frac{\deg(\gamma_{Z'}) + 2}{2}\, \dim Z \;\ge\; \dim A.
\]
This is not quite strong enough to contradict our assumption, but it allows us to reduce to the case of symmetric components: If $Z'\neq -Z'$ were not symmetric, then the effective symmetric cycle $Z$ would contain the two distinct components~$\pm Z'$ and these two components clearly have the same Gauss degree, which would lead to the estimate
$
 2n =\deg(\gamma_Z) \ge \deg(\gamma_{Z'}) + \deg(\gamma_{-Z'}) = 2\deg(\gamma_{Z'})$. 
The previous inequality then leads to
\[ \frac{n+2}{2} \dim Z  \;\ge\; \dim A \]
and for $n\ge 2$ this contradicts our assumption that $n\dim Z < \dim A$.

\medskip 

So for the remainder of the proof we will assume that the subvariety $Z'=-Z'$ is symmetric.
Since by assumption $n\dim Z' < \dim A$, it follows from \cref{prop:linear-growth-for-small-wedge-powers} that the morphism
\[ \pr_{Z', n}\colon \Supp \PLambda_{Z'}^{[n]} \too {Z'}^n \]
is surjective. Let $\PLambda' \subset \Supp \PLambda_{Z'}^{[n]}$ be an irreducible component dominating $Z'^n$. This component is not contained in the union $\Delta \cup \Delta^-\subset A^n\times \bbP_A$ of the big diagonal and the big antidiagonal, hence
\[
 \PLambda' \;\subset\; \Supp \PLambda_{Z, \sym}^{[n]}.
\]
If the Gauss map has even monodromy, then moreover $\PLambda_{Z, \sym}^{[n]} = \PLambda_{Z, \sym, +}^{[n]} + \PLambda_{Z, \sym, -}^{[n]}$ and, by irreducibility,
\[
 \PLambda' \;\subset\; \Supp \PLambda_{Z, \sym, \delta}^{[n]}
 \quad \textnormal{for some} \quad \delta \;\in\; \{+,-\}.
\]
To adjust the signs, use the involution
$
 \varphi := \id^{n-1} \times [-1] \colon A^n \to A^n
$
and consider the subvariety
\[
 \PLambda  \;\subset\; \Supp \PLambda^{[n]}_{Z, \sym, \epsilon}
 \quad \textnormal{defined by} 
 \quad 
 \PLambda \;:=\; 
 \begin{cases} 
 \; \varphi_*\, \PLambda' & \textnormal{if $\epsilon = -\delta$}, \\[0.2em]
 \;\PLambda' & \textnormal{otherwise}.
 \end{cases} 
\]
Since $Z'=-Z'$, we still have 
\[ \pr_{A^n}(\PLambda) \;=\; (Z')^{n-1} \times (\pm Z') \;=\; (Z')^n. \]
By construction, $\PLambda$ is an irreducible component of $\Supp \PLambda_{Z, \sym, \epsilon}^{[n]}$ and by assumption we have 
\[
 \PLambda_Y \;=\;
  \Alt^n_{\sym, \epsilon} \PLambda_Z \;:=\; 
  \frac{1}{N(\alpha)} \, \sigma_{\alpha *} \bigl( \PLambda_{Z, \sym, \epsilon}^{[n]} \bigr)  
  \quad \textnormal{for} \quad 
  \alpha \;=\; (1^n),
\]
so the irreducibility of the left-hand side forces $\PLambda_Y = \Supp \sigma_{\alpha *}(\PLambda)$. This gives a commutative diagram
\[
\begin{tikzcd}[column sep=40pt]
(Z')^n  \ar[d, swap, "\sigma"] &  \PLambda \ar[d, "\sigma_\alpha"]  \ar[l, swap, "\pr_{Z,r}"]  \ar[r,"\gamma_{Z,r}"] &  \bbP_A  \ar[d, equals] \\
Y &  \PLambda_Y \ar[l, "\pr_Y"] \ar[r, swap, "\gamma_Y"]  & \bbP_A
\end{tikzcd} 
\]	
where $\sigma$ is the sum morphism, which is generically finite by \cref{Lem:SumOfNondegenerate}.

\medskip 

In fact, $\sigma$ must be finite by the same argument as in \cref{prop:sum-morphism-finite}: The fibers of $\pr_Y$ are of pure dimension $N=\codim_A Y - 1$ by \cref{Corollary:Equidimensional-Fibers}. Moreover, $\sigma_\alpha$ is finite since the rightmost square in the above diagram commutes and since in that square $\gamma_{Z,r}$ and~$\gamma_Y$ are finite morphisms by our finiteness assumptions on Gauss maps. So all fibers of $\pr_Y \circ \sigma_\alpha \colon \PLambda_{Z, r} \to Y$ are of dimension $N$. Since $\sigma$ is generically finite, it follows from the commutativity of the leftmost square in the above diagram that the generic fiber of the morphism $\pr_{Z, r}$ has dimension $N$ as well. We can now argue by contradiction: Any positive-dimensional fiber of $\sigma$ would give rise to a fiber of $\pr_Y \circ \sigma_\alpha$ of dimension $\ge N+1$, by semi-continuity of dimension of fibers for proper morphisms~\cite[\href{https://stacks.math.columbia.edu/tag/0D4I}{lemma 0D4I}]{stacks-project}. This shows that $\sigma$ is finite.

\medskip 

But $Z'=-Z'$ is symmetric and $\dim Z' > 0$, hence looking at antidiagonals one sees that the sum morphism~$\sigma$ cannot be finite. Contradiction.
\end{proof}

\appendix

\section{Reduction to the complex case}

For reference we include the following well-known fact about $\ell$-adic constructible sheaves with coefficients in $\bbF = \overline{\bbQ}_\ell$ on varieties over an algebraically closed field  $k$ of characteristic zero:

\begin{lemma} \label{lem:reduction-to-complex-case} Let   $k \subset K$ be an extension of  algebraically closed   fields of characteristic zero and let $X$ be a variety over $k$. Then the base change functor
	\[
	(-)_K\colon \Dbc(X, \bbF) \too \Dbc(X_K, \bbF)
	\]
	is fully faithful. Moreover, for every $P\in \Dbc(X_K, \bbF)$ there exists a subfield $k'\subset K$ which is the algebraic closure of a finitely generated extension of $k$ such that $P$ is in the
	essential image of the functor
	\[
	(-)_K\colon \Dbc(X_{k'}, \bbF) \too \Dbc(X_K, \bbF).
	\]
\end{lemma}

\begin{proof} 
	For the full faithfulness, consider two complexes $P,P'\in \Dbc(X, \bbF)$. To see $\Hom_{\Dbc(X, \bbF)}(P, P') \simeq \Hom_{\Dbc(X_K, \bbF)}(P_K, P'_K)$, we only need to take  $Q=R\cH om(P, P')$ in the isomorphism
	\[
	H^\bullet(X, Q) \;\stackrel{\sim}{\too}\; 
	H^\bullet(X_K, Q_K)
	\] 
	which is obtained by base change (see e.g.~\cite[cor.~VI.4.3]{MilneEC} for the case of \'etale torsion sheaves, the case of $\ell$-adic sheaf complexes then follows formally). 
	
	\medskip 
	
	Now let $P\in \Dbc(X_K, \bbF)$. We want to show that it descends to a subfield $k'\subset K$ which is the algebraic closure of a finitely generated extension of $k$. We use induction on the number of nonvanishing cohomology sheaves. Let $m \in \bbZ$ be maximal with $\cH^m(P)\neq 0$, and consider the  triangle
	\[
	\tau_{<m}(P) \too P \too \cH^m(P)[-m] \too
	\]
	Rotating the triangle, we obtain
	\[
	P \;\simeq\; \cone(\cH^m(P)[1-m] \to \tau_{<m}(P))
	\]
	If the source and the target of a morphism descend to a given subfield, then so does the morphism by full faithfulness, and hence also the cone descends to the same subfield. By induction, it therefore suffices to discuss the case where $P$ is a single constructible $\overline{\bbQ}_\ell$-sheaf. Then by~\cite[Rapport, prop.~2.5]{SGA4H}, there is an open dense subset of $X_K$ on which $P$ is smooth. Let $k'\subset K$ be the algebraic closure of a finitely generated extension of $k$ such that the open dense subset has the form $j_K\colon U_K \to X_K$ for some open $j\colon U \into X_{k'}$. Looking at the adjunction morphism
	\[
	j_{K!}j_K^*(P) \too P
	\]
	and arguing by induction on $\dim \Supp(P)$ it will suffice to show that $j_K^*(P)\simeq L_K$ for some local system $L$ on $U$. But this is clear because of the equivalence between $\ell$-adic local systems and representations of the \'etale fundamental group~\cite[Rapport, prop.~2.4]{SGA4H} and the invariance of the geometric \'etale fundamental group \cite[Exp.~XIII, prop.~4.6]{SGA1}.
\end{proof}

\section{Symmetric powers of varieties}\label{section symmetric powers}

Let $X$ be a variety over an algebraically closed field $k$ of characteristic $0$, and fix an integer $n \ge 1$. The \emph{$n$-fold symmetric product} of $X$ is defined as the categorical quotient 
\[ \Sym^n X \;:=\; X^n / \frS_n \]
of $X^n=X\times \cdots \times X$ by the permutation action of the symmetric group $\frS_n$. This quotient exists for instance if $X$ is quasiprojective \cite[\S 9.3, p. 253]{NeronModels}. 
In that case, we denote by $$\pi_X \colon \quad X^n \; \too \; \Sym^n X$$ 
the quotient morphism. Let $U\subset X^n$ be the complement of the big diagonal, i.e.,~the open subset of all $n$-tuples of pairwise distinct points. Then $\pi_X(U)$ is open in $\Sym^n X$ and $\pi_X \colon U \to \pi_X(U)$ is a principal $\frS_n$-bundle~\cite[exp.~V, th.~4.1]{SGA3}, in particular
\[ \deg \pi_X = n!. \]
On the other hand, $\pi_X$ is not \'etale at any point $x \in (X^n \smallsetminus U)(k)$.

\begin{proposition} \label{Prop:PropertiesOfSymmetricPowers}
Let $X$ be a quasiprojective variety. If $X$ is reduced, irreducible, integral or normal, then the respective property holds also for $\Sym^n X$.
\end{proposition}

\begin{proof} 
If $X$ has one of the stated properties, then $X^n$ has the same property since $k$ is algebraically closed. Therefore, the claim follows from the fact that the properties are stable under categorical quotients~\cite[p. 5]{GIT}. 
\end{proof}

\begin{proposition} \label{Prop:FiniteBirationMorphismBetweenSymmetricPowers} 
If $f \colon X' \to X$ is a finite birational morphism between integral quasiprojective varieties, then $\Sym^n f \colon \Sym^n X' \to \Sym^n X$ is finite birational.
\end{proposition}

\begin{proof} In the commutative square
	\[
	\begin{tikzcd}[column sep=40pt]
	X'^n \ar[r, "f^n"] \ar[d, "\pi_{X'}"'] & X^n \ar[d, "\pi_{X}"] \\
	\Sym^n X' \ar[r, "\Sym^n f"] & \Sym^n X
	\end{tikzcd}
	\]
	the morphisms $f^n$, $\pi_{X'}$ and $\pi_X$ are finite, hence $\Sym^n f$ is also finite. Since $f^n$ is birational, we obtain
	\[ \deg (\Sym^n f \circ \pi_{X'}) =  \deg (\pi_X \circ f^n) = n! = \deg(\pi_X) \]
	which implies that the morphism $\Sym^n f$ has degree $1$ as required.
\end{proof}

\begin{proposition} \label{Prop:SmoothSymmetricPower} 
Let $X$ be an integral quasiprojective variety with $\dim X > 0$, and let $n\ge 2$ be an integer. Then
\[  \text{$\Sym^n X$ is smooth}\iff \text{$X$ is a smooth curve}. \]
\end{proposition}

\begin{proof} It is well known that for any smooth curve $X$ the symmetric powers $\Sym^n X$ are smooth~\cite[p.~255]{NeronModels}. Conversely, assume that $\Sym^n X$ is smooth. First, we show that $\dim X = 1$. For this, we may replace $X$ by its smooth locus and thus assume $X$ is smooth. By Nagata-Zariski purity \cite[th.~X.3.1]{SGA1}, the branch locus $B$ of $\pi_X \colon X^n \to \Sym^n X$ is empty or a divisor in $\Sym^n X$. On the other hand, the morphism $\pi_X$ is ramified at a $k$-point $x = (x_1, \dots, x_n)$ of $X^n$ if and only if $x_i = x_j$ for some $i \neq j$, that is, if and only if $x$ lies in the big diagonal $\Delta_n$ of $X^n$.   Since $n\ge 2$,  the big diagonal is nonempty, so that the branch locus is a divisor and thus
	\[ n \dim X - 1 = \dim B = \dim \Delta_n = (n - 1) \dim X,\]
	hence $\dim X = 1$. Having shown that $X$ is a curve, we now verify that it must be smooth. Suppose to the contrary that there exists a singular point $x_1\in X(k)$. Pick pairwise distinct points $x_2, \dots, x_n\in X(k)\smallsetminus \{x_1\}$. Then the morphism $\pi_X$ is \'etale at $x = (x_1, x_2, \dots, x_n)$, so $\pi_X(x)$ is a singular point in $\Sym^n X$, a contradiction.
\end{proof}

In~\cref{Sec:SumMorphisms} we needed a criterion for a birational morphism from a symmetric power of a smooth variety to another smooth variety to be an isomorphism. The proof is naturally cast for certain singularities:
To define them, recall that for a coherent sheaf $\cF$ on a variety $V$ and an integer $m \ge 1$ we write $\cF^{[m]}$ for the reflexive hull of $\cF^{\otimes m}$. A reflexive sheaf $\cF$ on $V$ of generic rank one is a \emph{$\bbQ$-line bundle} if there is $m \ge 1$ such that $\cF^{[m]}$ is a line bundle. When $V$ is proper, such a sheaf~$\cF$ is \emph{nef} if the line bundle $\cF^{[m]}$ is. For $V$ normal, the \emph{canonical sheaf} $\cK_V$ is defined as the pushforward to $V$ of the canonical bundle on $V^\reg$ and is reflexive of generic rank one. By definition $V$ is \emph{$\bbQ$-Gorenstein} if it is normal and the canonical sheaf is a $\bbQ$-line bundle. This is the case if $V$ is \emph{$\bbQ$-factorial}, i.e. if~$V$ is normal and any reflexive sheaf of generic rank one on $V$ is a $\bbQ$-line bundle. The singularities of~$V$ are \emph{terminal} if $V$ is $\bbQ$-Gorenstein and the pullback of any local section of~$\cK_V^{[m]}$ for any $m \ge 1$ vanishes along all the components of the exceptional divisor of any resolution of $V$; see \cite[def. 2.34]{KM98}. With this terminology we have:

\begin{proposition} \label{Lem:BirationalMapIsIso}
Let $X$ and $W$ be $\bbQ$-factorial normal, integral projective varieties. Suppose that $W$ has terminal singularities. 
If  $\cK_X$ is nef and $\dim X \ge 2$, then any proper birational morphism $f\colon \Sym^n X \to W$ is an isomorphism.
\end{proposition}

\begin{proof} The symmetric product $S := \Sym^n X$ of $X$ is normal by~\cref{Prop:PropertiesOfSymmetricPowers}  and $\bbQ$-factorial by \cite[lemma 5.16]{KM98}. The hypothesis $\dim X \ge 2$ implies that the quotient morphism $\pi \colon X^n \to S$ is unramified in codimension one, hence the natural morphism $(\pi^\ast \cK_S)^{\vee \vee} \to \cK_{X^n}$ is an isomorphism. Thus $\cK_{S}$ is nef since $\cK_{X^n}$ is so. To conclude apply \cref{SubLem:BirationalMapIsIso} below with $V = S$.
\end{proof}

To keep track of the arguments that enter the proof, we state the lemma in a generality which is slightly broader than actually needed. To do this, for a proper morphism $f \colon V \to W$ a $\bbQ$-line bundle $\cL$ on $V$ is said to be \emph{$f$-nef} if the restriction of $\cL^{[m]}$ to any fiber of $f$ is nef, where $m \ge 1$ is such that $\cL^{[m]}$ is a line bundle.

\begin{lemma} \label{SubLem:BirationalMapIsIso} Let $f \colon V \to W$ be a proper birational morphism between normal quasiprojective varieties. Suppose that $V$ is $\bbQ$-Gorenstein, $\cK_V$ is $f$-nef and~$W$ is~$\bbQ$-factorial with terminal singularities. Then $f$ is an isomorphism.
\end{lemma}

\begin{proof} By assumption the varieties $V$ and $W$ are $\bbQ$-Gorenstein, thus there is  an integer  $m \ge 1$ such that $\cK_V^{[m]}$ and $\cK_W^{[m]}$ are line bundles. Write 
\[\cK_{V}^{[m]} = f^*\cK_W^{[m]} \otimes \cO_V(E)\] 
for some Cartier divisor $E$ on $V$. Since $W$ has terminal singularities, the divisor~$E$ is effective and its support is exactly the divisorial part of the exceptional locus of~$f$. Moreover~$E$ is $f$-nef: Indeed, for any projective curve $C \subset V$ contracted by $f$ we have
\[ E.C = \cK_{V}^{[m]}.C - f^*\cK_W^{[m]}.C  = \cK_{V}^{[m]}.C \ge 0\]
because $\cK_V$ is $f$-nef. Therefore $-E$ is effective by \cite[lemma~3.39~(1)]{KM98}, hence trivial because $E$ is effective. It follows that the exceptional set of $f$ has no divisorial part. On the other hand $W$ is $\bbQ$-factorial, so the exceptional locus of $f$ is pure of codimension one~\cite[1.40]{Deb01}. Thus, $f$ is an isomorphism.
\end{proof}

\small

\bibliography{./../biblio}

\providecommand{\bysame}{\leavevmode\hbox to3em{\hrulefill}\thinspace}
\providecommand{\MR}{\relax\ifhmode\unskip\space\fi MR }
\providecommand{\MRhref}[2]{%
  \href{http://www.ams.org/mathscinet-getitem?mr=#1}{#2}
}
\providecommand{\href}[2]{#2}
\begin{thebibliography}{DMOS82}

\bibitem[Abr94]{Abr94}
D.~Abramovich, \emph{Subvarieties of semiabelian varieties}, Compositio Math.
  \textbf{90} (1994), no.~1, 37--52.

\bibitem[And92]{AndreMumfordTate}
Y.~Andr{\'e}, \emph{Mumford-{Tate} groups of mixed {Hodge} structures and the
  theorem of the fixed part}, Compos. Math. \textbf{82} (1992), no.~1, 1--24.

\bibitem[BBDG18]{BBDG}
A.~Beilinson, J.~Bernstein, P.~Deligne, and O.~Gabber, \emph{Faisceaux pervers.
  {Actes} du colloque ``{Analyse} et {Topologie} sur les {Espaces}
  {Singuliers}''. {Partie} {I}}, 2nd edition ed., Ast{\'e}risque, vol. 100,
  Paris: Soci{\'e}t{\'e} Math{\'e}matique de France (SMF), 2018.

\bibitem[Bea86]{Beauville}
A.~Beauville, \emph{Le groupe de monodromie des familles universelles
  d'hypersurfaces et d'intersections compl\`etes}, Complex analysis and
  algebraic geometry ({G}\"{o}ttingen, 1985), Lecture Notes in Math., vol.
  1194, Springer, Berlin, 1986, pp.~8--18.

\bibitem[BLR90]{NeronModels}
S.~Bosch, W.~L\"{u}tkebohmert, and M.~Raynaud, \emph{N\'{e}ron models},
  Ergebnisse der Mathematik und ihrer Grenzgebiete (3), vol.~21,
  Springer-Verlag, Berlin, 1990.

\bibitem[BSS93]{BeltramettiSchneiderSommese}
M.~C. Beltrametti, M.~Schneider, and A.~J. Sommese, \emph{Applications of the
  {E}in-{L}azarsfeld criterion for spannedness of adjoint bundles}, Math. Z.
  \textbf{214} (1993), no.~4, 593--599.

\bibitem[BSS18]{BSS}
B.~Bhatt, C.~Schnell, and P.~Scholze, \emph{Vanishing theorems for perverse
  sheaves on abelian varieties, revisited}, Sel. Math., New Ser. \textbf{24}
  (2018), no.~1, 63--84.

\bibitem[Cou20]{Coulembier}
K.~Coulembier, \emph{Tannakian categories in positive characteristic}, Duke
  Math. J. \textbf{169} (2020), no.~16, 3167--3219.

\bibitem[DE21]{DAE20}
M.~D'Addezio and H.~Esnault, \emph{On the universal extensions in tannakian
  categories}, International Mathematics Research Notices (2021).

\bibitem[Deb82]{DebarreNoether}
O.~Debarre, \emph{In\'{e}galit\'{e}s num\'{e}riques pour les surfaces de type
  g\'{e}n\'{e}ral}, Bull. Soc. Math. France \textbf{110} (1982), no.~3,
  319--346, With an appendix by A. Beauville.

\bibitem[Deb95]{Deb95}
\bysame, \emph{Fulton-{H}ansen and {B}arth-{L}efschetz theorems for
  subvarieties of abelian varieties}, J. Reine Angew. Math. \textbf{467}
  (1995), 187--197.

\bibitem[Deb01]{Deb01}
\bysame, \emph{Higher-dimensional algebraic geometry}, Universitext,
  Springer-Verlag, New York, 2001.

\bibitem[Deb05]{DebarreAV}
\bysame, \emph{Complex tori and abelian varieties}, SMF/AMS Texts and
  Monographs, vol.~11, American Mathematical Society, Providence, RI;
  Soci\'{e}t\'{e} Math\'{e}matique de France, Paris, 2005, Translated from the
  1999 French edition by Philippe Mazaud.

\bibitem[Dim04]{DimcaSheaves}
A.~Dimca, \emph{Sheaves in topology}, Universitext, Berlin: Springer, 2004.

\bibitem[DMOS82]{DM82}
P.~Deligne, J.~S. Milne, A.~Ogus, and K.~Shih, \emph{Hodge cycles, motives, and
  {S}himura varieties}, Lecture Notes in Mathematics, vol. 900,
  Springer-Verlag, Berlin-New York, 1982.

\bibitem[Ebe84]{Ebeling}
W.~Ebeling, \emph{An arithmetic characterisation of the symmetric monodromy
  groups of singularities}, Invent. Math. \textbf{77} (1984), no.~1, 85--99.

\bibitem[EIL00]{EinLazarsfeldBo}
L.~Ein, B.~Ilic, and R.~Lazarsfeld, \emph{A remark on projective embeddings of
  varieties with non-negative cotangent bundles}, Complex analysis and
  algebraic geometry, de Gruyter, Berlin, 2000, pp.~165--171.

\bibitem[FH91]{FultonHarris}
W.~Fulton and J.~Harris, \emph{Representation theory}, Graduate Texts in
  Mathematics, vol. 129, Springer-Verlag, New York, 1991, A first course,
  Readings in Mathematics.

\bibitem[FK00]{FraneckiKapranov}
J.~Franecki and M.~Kapranov, \emph{The {G}auss map and a noncompact
  {R}iemann-{R}och formula for constructible sheaves on semiabelian varieties},
  Duke Math. J. \textbf{104} (2000), no.~1, 171--180.

\bibitem[Ful98]{FultonIntersectionTheory}
W.~Fulton, \emph{Intersection theory}, second ed., Ergebnisse der Mathematik
  und ihrer Grenzgebiete. 3. Folge. A Series of Modern Surveys in Mathematics,
  vol.~2, Springer-Verlag, Berlin, 1998.

\bibitem[Gab62]{Gabriel62}
P.~Gabriel, \emph{Des cat\'{e}gories ab\'{e}liennes}, Bull. Soc. Math. France
  \textbf{90} (1962), 323--448.

\bibitem[Har71]{Har71}
R.~Hartshorne, \emph{Ample vector bundles on curves}, Nagoya Math. J.
  \textbf{43} (1971), 73--89.

\bibitem[Hum78]{Hum78}
J.~E. Humphreys, \emph{Introduction to {L}ie algebras and representation
  theory}, Graduate Texts in Mathematics, vol.~9, Springer-Verlag, New
  York-Berlin, 1978, Second printing, revised.

\bibitem[Jan83]{Janssen}
W.~A.~M. Janssen, \emph{Skew-symmetric vanishing lattices and their monodromy
  groups}, Math. Ann. \textbf{266} (1983), no.~1, 115--133.

\bibitem[Kat01]{KatzLFM}
N.~M. Katz, \emph{{{\(L\)}}-functions and monodromy: four lectures on {Weil}
  {II}}, Adv. Math. \textbf{160} (2001), no.~1, 81--132.

\bibitem[Kat04]{KatzLarsen}
\bysame, \emph{Larsen's alternative, moments, and the monodromy of {L}efschetz
  pencils}, Contributions to automorphic forms, geometry, and number theory,
  Johns Hopkins Univ. Press, Baltimore, MD, 2004, pp.~521--560.

\bibitem[KM98]{KM98}
J.~Koll\'ar and S.~Mori, \emph{Birational geometry of algebraic varieties},
  Cambridge Tracts in Mathematics, vol. 134, Cambridge University Press,
  Cambridge, 1998, With the collaboration of C. H. Clemens and A. Corti,
  Translated from the 1998 Japanese original.

\bibitem[KM23]{KM}
T.~Kr\"amer and M.~Maculan, \emph{Arithmetic finiteness of irregular
  varieties}, 2023,
  \href{https://arxiv.org/abs/2310.08485}{\texttt{arXiv:2310.08485}}.

\bibitem[Kr{\"a}14]{KraemerSemiabelian}
T.~Kr{\"a}mer, \emph{Perverse sheaves on semiabelian varieties}, Rend. Semin.
  Mat. Univ. Padova \textbf{132} (2014), 83--102.

\bibitem[Kr{\"a}20]{KraemerThetaSummands}
\bysame, \emph{Summands of theta divisors on {J}acobians}, Compos. Math.
  \textbf{156} (2020), no.~7, 1457--1475.

\bibitem[Kr{\"a}21]{KraemerMicrolocalII}
\bysame, \emph{Characteristic cycles and the microlocal geometry of the {G}auss
  map, {II}}, J. Reine Angew. Math. \textbf{774} (2021), 53--92.

\bibitem[Kr{\"a}22]{KraemerMicrolocalI}
\bysame, \emph{Characteristic cycles and the microlocal geometry of the {G}auss
  map, {I}}, Ann. Sci. \'{E}c. Norm. Sup\'{e}r. (4) \textbf{55} (2022),
  1475--1527.

\bibitem[KW15a]{KWGeneric}
T.~Kr{\"a}mer and R.~Weissauer, \emph{On the {Tannaka} group attached to the
  theta divisor of a generic principally polarized abelian variety}, Math. Z.
  \textbf{281} (2015), no.~3-4, 723--745.

\bibitem[KW15b]{KWSmall}
\bysame, \emph{Semisimple super {Tannakian} categories with a small tensor
  generator}, Pac. J. Math. \textbf{276} (2015), no.~1, 229--248.

\bibitem[KW15c]{KWVanishing}
\bysame, \emph{Vanishing theorems for constructible sheaves on abelian
  varieties}, J. Algebraic Geom. \textbf{24} (2015), no.~3, 531--568.

\bibitem[Laz04a]{LazarsfeldPositivityI}
R.~Lazarsfeld, \emph{Positivity in algebraic geometry. {I}}, Ergebnisse der
  Mathematik und ihrer Grenzgebiete. 3. Folge. A Series of Modern Surveys in
  Mathematics, vol.~48, Springer-Verlag, Berlin, 2004, Classical setting: line
  bundles and linear series.

\bibitem[Laz04b]{LazarsfeldPositivityII}
\bysame, \emph{Positivity in algebraic geometry. {II}}, Ergebnisse der
  Mathematik und ihrer Grenzgebiete. 3. Folge. A Series of Modern Surveys in
  Mathematics, vol.~49, Springer-Verlag, Berlin, 2004, Positivity for vector
  bundles, and multiplier ideals.

\bibitem[LS20]{LS20}
B.~Lawrence and W.~Sawin, \emph{The {S}hafarevich conjecture for hypersurfaces
  in abelian varieties}, ar{X}iv, 2020,
  \href{https://arxiv.org/abs/2004.09046v3}{\texttt{arXiv:2004.09046v3}}.

\bibitem[LV20]{LV}
B.~Lawrence and A.~Venkatesh, \emph{Diophantine problems and {$p$}-adic period
  mappings}, Invent. Math. \textbf{221} (2020), no.~3, 893--999.

\bibitem[Mac62]{MacdonaldSymmetric}
I.~G. Macdonald, \emph{Symmetric products of an algebraic curve}, Topology
  \textbf{1} (1962), 319--343.

\bibitem[MFK94]{GIT}
D.~Mumford, J.~Fogarty, and F.~Kirwan, \emph{Geometric invariant theory}, third
  ed., Ergebnisse der Mathematik und ihrer Grenzgebiete (2), vol.~34,
  Springer-Verlag, Berlin, 1994.

\bibitem[Mil80]{MilneEC}
J.~S. Milne, \emph{\'{E}tale cohomology}, Princeton Mathematical Series, No.
  33, Princeton University Press, Princeton, N.J., 1980.

\bibitem[Mil86]{MilneAV}
\bysame, \emph{Abelian varieties}, Arithmetic geometry, {Pap}. {Conf}.,
  {Storrs}/{Conn}. 1984, 103-150 (1986)., 1986.

\bibitem[Sch15]{SchnellHolonomic}
C.~Schnell, \emph{Holonomic {$\mathcal{D}$}-modules on abelian varieties},
  Publ. Math., Inst. Hautes {\'E}tud. Sci. \textbf{121} (2015), 1--55.

\bibitem[SGA 1]{SGA1}
\emph{Rev\^{e}tements \'{e}tales et groupe fondamental ({SGA} 1)}, Documents
  Math\'{e}matiques (Paris), vol.~3, Soci\'{e}t\'{e} Math\'{e}matique de
  France, Paris, 2003, S\'{e}minaire de g\'{e}om\'{e}trie alg\'{e}brique du
  Bois Marie 1960--61. Directed by A. Grothendieck, With two papers by M.
  Raynaud, Updated and annotated reprint of the 1971 original.

\bibitem[SGA 3]{SGA3}
P.~Gille and P.~Polo (eds.), \emph{Sch\'{e}mas en groupes ({SGA} 3). {T}ome
  {I}. {P}ropri\'{e}t\'{e}s g\'{e}n\'{e}rales des sch\'{e}mas en groupes},
  Documents Math\'{e}matiques (Paris), vol.~7, Soci\'{e}t\'{e} Math\'{e}matique
  de France, Paris, 2011, S\'{e}minaire de G\'{e}om\'{e}trie Alg\'{e}brique du
  Bois Marie 1962--64. A seminar directed by M. Demazure and A. Grothendieck
  with the collaboration of M. Artin, J.-E. Bertin, P. Gabriel, M. Raynaud and
  J-P. Serre, Revised and annotated edition of the 1970 French original.

\bibitem[SGA $4\tfrac{1}{2}$]{SGA4H}
P.~Deligne, \emph{Cohomologie \'{e}tale}, Lecture Notes in Mathematics, vol.
  569, Springer-Verlag, Berlin, 1977, S\'{e}minaire de g\'{e}om\'{e}trie
  alg\'{e}brique du Bois-Marie SGA $4\frac{1}{2}$.

\bibitem[{Sta}22]{stacks-project}
The {Stacks Project Authors}, \emph{The {Stacks Project}},
  \url{https://stacks.math.columbia.edu}, 2022.

\bibitem[Uen73]{UenoCompositio}
K.~Ueno, \emph{Classification of algebraic varieties. {I}}, Compositio Math.
  \textbf{27} (1973), 277--342.

\bibitem[Wei06]{WeissauerBNSheaves}
R.~Weissauer, \emph{Brill-noether sheaves}, ar{X}iv, 2006,
  \href{https://arxiv.org/abs/math/0610923}{\texttt{arXiv:math/0610923}}.

\bibitem[Wei11]{WeissauerRigidity}
\bysame, \emph{A remark on rigidity of {BN}-sheaves}, ar{X}iv, 2011,
  \href{https://arxiv.org/abs/1111.6095}{\texttt{arXiv:1111.6095}}.

\bibitem[Wei15a]{WeissauerArxiv2015a}
\bysame, \emph{On {S}ubvarieties of {A}belian {V}arieties with degenerate
  {G}auss mapping}, ar{X}iv, 2015,
  \href{https://arxiv.org/abs/1110.0095}{\texttt{arXiv:1110.0095}}.

\bibitem[Wei15b]{WeissauerAlmostConnected}
\bysame, \emph{Why certain {T}annaka groups attached to abelian varieties are
  almost connected}, ar{X}iv, 2015,
  \href{https://arxiv.org/abs/1207.4039}{\texttt{arXiv:1207.4039}}.

\end{thebibliography}

\bibliographystyle{amsalpha}

\end{document}